\documentclass[12pt,a4paper]{article}
\usepackage{amsfonts}

\usepackage{mathrsfs}

\usepackage{stmaryrd}
\usepackage{amsmath}
\usepackage{amssymb}

\usepackage{bbm}
\usepackage{mathrsfs}
\usepackage{graphicx}
\usepackage{lipsum}

\usepackage{enumerate}
\usepackage{theorem}
\usepackage{indentfirst}
\makeatletter
\def\tank#1{\protected@xdef\@thanks{\@thanks
        \protect\footnotetext[0]{#1}}}
\def\bigfoot{

    \@footnotetext}
\makeatother

\newcommand{\ea}{\end{array}}

\newtheorem{theorem}{Theorem}[section]
\newtheorem{hypothesis}{Hypothesis}[section]

\newtheorem{lemma}{Lemma}[section]
\newtheorem{definition}{Definition}[section]
\newtheorem{Rem}{Remark}[section]

{\theorembodyfont{\rmfamily}
}

\newenvironment{proof}{Proof.}

\def \eref#1{\hbox{(\ref{#1})}}

\renewcommand{\d}{d}

\begin{document}
\title{{\Large \bf Freidlin-Wentzell Type Large Deviation Principle for Multi-Scale Locally Monotone SPDEs}}

\author{{Wei Hong$^{a}$},~~{Shihu Li$^{b}$},~~{Wei Liu$^{b,c}$}\footnote{Corresponding author: weiliu@jsnu.edu.cn}
\\
 \small $a.$ Center for Applied Mathematics, Tianjin University, Tianjin 300072, China \\
 \small $b.$ School of Mathematics and Statistics, Jiangsu Normal University, Xuzhou 221116, China\\
  \small $c.$ Research Institute of Mathematical Sciences, Jiangsu Normal University, Xuzhou 221116, China }
\date{}
\maketitle
\begin{center}
\begin{minipage}{145mm}
{\bf Abstract.} This work is concerned with Freidlin-Wentzell type large deviation principle for a family of multi-scale quasilinear and semilinear stochastic partial differential equations. Employing the weak convergence method and Khasminskii's time discretization approach, the Laplace principle (equivalently, large deviation principle) for a general class of multi-scale SPDEs is derived. In particular, we succeed in dropping the compactness assumption of embedding in the
Gelfand triple in order to deal with the case of bounded and unbounded  domains in applications. Our main results are applicable to various multi-scale SPDE models such as stochastic porous media equations, stochastic p-Laplace equations, stochastic fast-diffusion equations, stochastic 2D hydrodynamical type models, stochastic power
law fluid equations and stochastic Ladyzhenskaya models.

\vspace{3mm} {\bf Keywords:} ~SPDE; Multi-scale;~Large deviation principle;~Porous media equation;~Navier-Stokes equation.


\end{minipage}
\end{center}

\section{Introduction}
The large deviation principle (LDP) mainly investigates the asymptotic property of remote tails of a family of probability distributions, which is one of important topics in the probability theory and has been widely applied in many fields such as thermodynamics, statistics, information theory and engineering.
We refer the interested readers to the classical monographs \cite{DZ,V1}  for the theory and important applications. Owing to the seminal work of  Freidlin and Wentzell \cite{FW}, the well-known small perturbation type (also called Freidlin-Wentzell type) large deviations for stochastic differential equations has been extensively studied in the recent decades, one might refer to \cite{A,S} and references therein.

There are numerous results concerning the LDP for SPDEs with small perturbation within different frameworks in the literature.
 In the classical paper \cite{Fr}  Freidlin  studied the large deviations
 for the small noise limit of stochastic reaction-diffusion equations. We refer the
reader to Da Prato and Zabczyk \cite{DaZa} or Peszat \cite{P} (also
 the references therein) for the extensions to infinite
 dimensional diffusions or SPDE
under global Lipschitz condition.
For the case of local Lipschitz condition we refer to the work \cite{CR} by Cerrai and R\"{o}ckner. Sowers \cite{S1} studied the LDP for a reaction diffusion equation with non-Gaussian perturbations.   The LDP
for semilinear parabolic equations on a Gelfand triple was studied by Chow in
 \cite{C}. R\"{o}ckner et al. \cite{RWW} established the LDP for stochastic porous media equations in both small noise and small time cases, which is the first LDP result for quasilinear SPDE.
  All of the above-mentioned papers used the classical time discretization method and the
 contraction principle, which was first developed by  Freidlin and Wentzell in \cite{FW}.
 But
 the situation in infinite dimensional case became quite involved and complicated
  since different  nonlinear SPDE needs different
 techniques  to verify some exponential estimate and tightness.

 Recently, the weak convergence method systematically developed by Dupuis, Ellis \cite{DE} and Budhiraja et al.~\cite{BD,BDM} has became a very powerful tool to study the LDP, and it mainly relies on the variational representation formula on certain functionals of Wiener process, moreover, the authors in \cite{BCD} also extended the weak convergence method to the case of stochastic dynamical systems driven by Poisson random measure.
 Compared with the time discretization approach, one main advantage of using weak convergence method  is that, instead of proving exponential probability estimates, one only need to establish some priori moment estimates, which significantly simplifies the proof (see \cite{BGJ,BDM,BM,CG,CM,DWZZ,MSS,MSZ,RZ12,RZ2,SS,XZ} and references therein for the recent progress on LDP for various SPDE models).

The main aim of this work is to investigate the LDP for the following multi-scale stochastic evolution equations (SEEs),
\begin{equation}\label{e0}
\left\{ \begin{aligned}
&dX^{\epsilon,\alpha}_t=\big[A(X^{\epsilon,\alpha}_t)+F_1(X^{\epsilon,\alpha}_t,Y^{\epsilon,\alpha}_t)\big]dt+\sqrt{\epsilon}G_1(X^{\epsilon,\alpha}_t)dW_t,\\
&dY^{\epsilon,\alpha}_t=\frac{1}{\alpha}F_2(X^{\epsilon,\alpha}_t,Y^{\epsilon,\alpha}_t)dt+\frac{1}{\sqrt{\alpha}}G_2dW_t,\\
&X^{\epsilon,\alpha}_0=x,~Y^{\epsilon,\alpha}_0=y,
\end{aligned} \right.
\end{equation}
where $\epsilon>0$, $\{W_t\}_{t\in [0,T]}$ is a cylindrical Wiener process, $\alpha:=\alpha(\epsilon)>0$ represents a small parameter (depending on $\epsilon$) which characterizes the ratio of time-scales between processes $X^{\epsilon,\alpha}_t$ and $Y^{\epsilon,\alpha}_t$, and the coefficients $A,F_1,F_2,G_1,G_2$ fulfill certain hypothesises.
The multi-scale processes including slow and fast components  have attracted more and more attentions due to their widespread applications in many fields such as climate dynamics, chemical kinetics, material science and stochastic mechanics (cf.~\cite{BYY,C2,CF,CL,FW1,K,K1} and more references therein). In fact, many physical systems have certain hierarchy so that not all components evolve at the same rate, i.e. some components changed rapidly while other ones changed very slowly. In these cases, it is natural to ask how the multi-time-scales influence the stochastic dynamical systems, for instance, what is the asymptotic behaviour of the solution to Eq.(\ref{e0}) while $\epsilon\to 0$  ?

For this purpose,  we aim to investigate the asymptotic behavior and establish the LDP for a general class of multi-scale models (\ref{e0}). The LDP for stochastic reaction-diffusion equations involving slow-fast components with small Gaussian perturbation was studied by Wang et al. \cite{WRD}, where they obtained some exponential tight estimates and used the contraction principle and certain approximation to get the LDP. In the work \cite{HSS}, the authors also investigated the LDP for a family of multi-scale stochastic reaction-diffusion equations based on the weak convergence method in infinite dimensions. Recently, using the weak convergence method and classical time discretization approach, Sun et~al.~\cite{SWXY} established the Freidlin-Wentzell type LDP for multi-scale stochastic Burgers equation, where some techniques of stopping time were also employed
 in order to deal with the more complicate non-linear term of the Burgers equation.
 To the best of our knowledge, all of the literatures concerning the LDP for infinite-dimensional multi-scale models used the mild solution approach to tackle different semilinear SPDEs, there is no LDP result for multi-scale quasilinear SPDEs such as the stochastic porous media equations, stochastic fast-diffusion equations and stochastic p-Laplace equations.

In order to investigate the LDP
for the above-mentioned multi-scale quasilinear SPDEs,  we adapt the generalized variational framework in this work which is applicable to a large family of quasilinear and semilinear SPDEs with locally monotone coefficients.  The classical variational framework has been established by Pardoux, Krylov and Rozovskii (see e.g. \cite{KR,LR1,RS}), where they employed the famous monotonicity tricks to verify the existence and uniqueness of
solutions for SPDEs fulfilling the classical monotonicity and coercivity assumptions. In recent years, such framework has been substantially
generalized in \cite{LR1,LR2,LR13} to more general circumstances fulfilling the local monotonicity and generalized coercivity, which cover various semilinear and quasilinear SPDEs such as stochastic porous media equations, stochastic fast-diffusion equations, stochastic 2D Navier-Stokes equations and other hydrodynamical type models, stochastic p-Laplace equations,  stochastic power law fluid equations, and stochastic Ladyzhenskaya models etc. We refer the interested readers to \cite{CM,HLL0,LLX,L1,LRS18,LRSX1,MZ,RZ2,W15,XZ,Z1,Z2} and reference therein for the recent development in such framework.

We want to point out that in the current variational framework one can not follow the main strategy of proofs in \cite{HSS,SWXY,WRD}, where they used the mild solution techniques to obtain some energy estimates and time H\"{o}lder continuity  for the solutions of associated stochastic control equations. Here we need to employ different approach to get the desired moment estimates of solutions to the skeleton equation and stochastic control equations (see (\ref{e2}) and (\ref{e5}) below), which are crucial to prove the compactness of level set of rate function and the convergence in distribution of solutions corresponding to the stochastic control equations. In order to overcome this difficulty appearing in the variational setting, some  stopping time techniques  and Khasminskii's time discretization approach will be employed to obtain some estimates involving different spaces in the Gelfand triple. Another difference is that we want to extend the related works (e.g.~\cite{L1,RZ2,SWXY}) to the case of unbounded domains (e.g.~Poincar\'{e} domains), that is, we do not assume any compactness on the Gelfand triple, see \cite[Section 3]{RZ2} or \cite[Lemma 3.3]{SWXY} for the details. To solve this difficulty, the time discretization approach  will also be adapted to deal with an additional perturbation term of stochastic control problem (see (\ref{e5}) below), which is mainly inspired by the work \cite{CM}.
 Comparing to the works \cite{CM,L1,RZ2,SWXY}, here we extend the LDP result to the  multi-scale case and cover a large number of SPDE models such as the stochastic porous media equations, stochastic fast-diffusion equations, stochastic 2D Navier-Stokes equations and other hydrodynamical type models, stochastic p-Laplace equations,  stochastic power law fluid equations, stochastic Ladyzhenskaya models, etc. To the best of our knowledge, the LDP results for most of the above-mentioned multi-scale models seem to be new in the literature.

The remainder of this paper is organized as follows. In Section 2, we formulate our mathematical models and impose some necessary assumptions on the coefficients. Then we introduce the LDP and Laplace principle with their equivalence and  state the main results of this work. In Section 3, some concrete stochastic models are given to illustrate the applications of our main results. In Section 4, we begin with considering the frozen equation and skeleton equation corresponding to Eq.~(\ref{e1}). We show the exponential ergodicity of frozen equations and investigate certain stochastic control problems with respect to Eq.~(\ref{e1}). Section 5 is devoted to proving the main results.

\section{Main Results}
\setcounter{equation}{0}
 \setcounter{definition}{0}
In this section, we first introduce some notations for the function spaces and operators, and provide the definitions of LDP and Laplace principle with their equivalence. Then we  state the main results of the present paper.
\subsection{Mathematical framework}
Let $(U,\langle\cdot,\cdot\rangle_U)$ and $(H_i, \langle\cdot,\cdot\rangle_{H_i})$, $i=1,2$, be the separable Hilbert spaces, and $H_i^*$ the dual space of $H_i$. Let $V_i$ denote some reflexive Banach space such that the embedding $V_i\subset H_i$ is continuous and dense. Identifying $H_i$  with its dual space in terms of the Riesz isomorphism, we are able to obtain the following Gelfand triples
$$V_i\subset H_i(\cong H_i^*)\subset V_i^*, ~ i=1,2 .$$
The dualization between $V_i$ and $V_i^*$ is denoted by $_{V_i^*}\langle\cdot,\cdot\rangle_{V_i}$. Moreover, it is easy to see that $$_{V_i^*}\langle\cdot,\cdot\rangle_{V_i}|_{{H_i}\times{V_i}}=\langle\cdot,\cdot\rangle_{H_i},~i=1,2.$$
Let $L_2(U,H_i)$ be the space of all Hilbert-Schmidt operators from $U$ to $H_i$.

Now we consider the following two-time-scale stochastic evolution equations on $[0, T]$,
\begin{equation}\label{e1}
\left\{ \begin{aligned}
&dX^{\epsilon,\alpha}_t=\big[A(X^{\epsilon,\alpha}_t)+F_1(X^{\epsilon,\alpha}_t,Y^{\epsilon,\alpha}_t)\big]dt+\sqrt{\epsilon}G_1(X^{\epsilon,\alpha}_t)dW_t,\\
&dY^{\epsilon,\alpha}_t=\frac{1}{\alpha}F_2(X^{\epsilon,\alpha}_t,Y^{\epsilon,\alpha}_t)dt+\frac{1}{\sqrt{\alpha}}G_2dW_t,\\
&X^{\epsilon,\alpha}_0=x,~Y^{\epsilon,\alpha}_0=y,
\end{aligned} \right.
\end{equation}
where
$$
A:V_1\rightarrow V_1^*,~~F_1:H_1\times H_2\to H_1,~~G_1:V_1\to L_2(U,H_1),
$$
and
$$F_2:H_1\times V_2\to V_2^*,~~G_2\in L_2(U,H_2),$$
are some measurable maps,
$\{W_t\}_{t\in [0,T]}$ is an $U$-cylindrical Wiener process defined on a complete filtered probability space $\left(\Omega,\mathscr{F},\mathscr{F}_{t\geq0},\mathbb{P}\right)$ (that is, the path
of $W$ take values in  $C([0,T];U_1)$, where $U_1$ is another
Hilbert space in which the embedding $U\subset U_1$ is
Hilbert--Schmidt).

Suppose that the coefficients of (\ref{e1}) satisfy the following two main hypothesises.
\begin{hypothesis}\label{h1}
For the slow component of Eq.~(\ref{e1}), we assume that there exist constants $\gamma_1>1$, $\beta_1\geq0$, $\theta_1>0$ and $K,C>0$ such that for all $u,v,w\in V_1$, $u_1,u_2\in H_1$ and $v_1,v_2\in H_2$, we have
\begin{enumerate}
\item [$({\mathbf{A}}{\mathbf{1}})$](Hemicontinuity) The map $\lambda\mapsto_{V_1^*}\langle A(u+\lambda v),w\rangle_{V_1}$ is continuous on $\mathbb{R}$.
\item [$({\mathbf{A}}{\mathbf{2}})$](Local monotonicity and Lipschitz)
\begin{eqnarray*}
  &&~~~2_{V_1^*}\langle A(u)-A(v), u-v\rangle_{V_1}+
 \|G_1(u)-G_1(v)\|_{L_2(U,H_1)}^2\\
     &&\leq -\theta_1\|u-v\|_{V_1}^{\gamma_1}+
     (K+\rho(v))\|u-v\|_{H_1}^2,
\end{eqnarray*}
where $\rho : V_1\rightarrow[0,+\infty)$ is a measurable and
locally bounded function on $V_1$ and satisfies
$$\rho(v)\leq C(1+\|v\|_{V_1}^{\gamma_1})(1+\|v\|_{H_1}^{\beta_1}).$$
Moreover,
\begin{equation*}
\|F_1(u_1,v_1)-F_1(u_2,v_2)\|_{H_1}\leq C\big(\|u_1-u_2\|_{H_1}+\|v_1-v_2\|_{H_2}\big).
\end{equation*}
and
\begin{equation*}
\|G_1(u)-G_1(v)\|_{L_2(U,H_1)}\leq C\|u-v\|_{H_1}.
\end{equation*}

\item [$({\mathbf{A}}{\mathbf{3}})$](Growth)
\begin{equation*}
\|A(u)\|_{V_1^*}^{\frac{\gamma_1}{\gamma_1-1}}\leq C(1+\|u\|_{V_1}^{\gamma_1})(1+\|u\|_{H_1}^{\beta_1}).
\end{equation*}
\end{enumerate}
\end{hypothesis}

\begin{hypothesis}\label{h2}
For the fast component of Eq.~(\ref{e1}), we assume that there exist constants $\gamma_2>1$, $\beta_2\geq0$, $\kappa,\theta_2>0$ and $C>0$ such that for all $v_1,v_2,v,w\in V_2$, $u_1,u_2,u\in H_1$, we have
\begin{enumerate}
\item [$({\mathbf{H}}{\mathbf{1}})$](Hemicontinuity) The map $\lambda\mapsto_{V_2^*}\langle F_2(u_1+\lambda u_2,v_1+\lambda v_2),w\rangle_{V_2}$ is continuous on $\mathbb{R}$.
\item [$({\mathbf{H}}{\mathbf{2}})$](Strict monotonicity)
\begin{eqnarray}\label{h6}
2_{V_2^*}\langle F_2(u_1,v_1)-F_2(u_1,v_2),v_1-v_2\rangle_{V_2}\leq-\kappa\|v_1-v_2\|_{H_2}^2.
\end{eqnarray}
Moreover,
\begin{equation}\label{h3}
_{V_2^*}\langle F_2(u_1,v)-F_2(u_2,v),w\rangle_{V_2}\leq C\|u_1-u_2\|_{H_1}\|w\|_{H_2}.
\end{equation}

\item [$({\mathbf{H}}{\mathbf{3}})$](Coercivity)
\begin{equation*}
_{V_2^*}\langle F_2(u,v),v\rangle_{V_2}\leq C\|v\|_{H_2}^2-\theta_2\|v\|_{V_2}^{\gamma_2}+C(1+\|u\|_{H_1}^2).
\end{equation*}

\item [$({\mathbf{H}}{\mathbf{4}})$](Growth)
\begin{equation*}
\|F_2(u,v)\|_{V_2^*}^{\frac{\gamma_2}{\gamma_2-1}}\leq C(1+\|v\|_{V_2}^{\gamma_2})(1+\|v\|_{H_2}^{\beta_2})+C\|u\|_{H_1}^2.
\end{equation*}
\end{enumerate}
\end{hypothesis}

\begin{Rem}\label{r1}
(i) By $({\mathbf{A}}{\mathbf{2}})$-$({\mathbf{A}}{\mathbf{3}})$,  the coercivity condition of $A$ and
$G_1$ can be obtained as
$$2{ }_{V_1^*}\langle A(u),
 u\rangle_{V_1}+\|G_1(u)\|^{2}_{L_2(U,H_1)}\leq-\frac{\theta_1}{2}\|u\|_{V_1}^{\gamma_1}
    + C(1+ \|u\|_{H_1}^2).$$
(ii) The assumption (\ref{h6}) is called strictly monotone condition, which ensures the existence and uniqueness of invariant probability measure and the exponential ergodicity for the frozen equation (see Eq.~(\ref{e3}) below) corresponding to the fast component of (\ref{e1}). A typical example satisfying  Hypothesis \ref{h2} is the stochastic reaction-diffusion type equations, for instance, let $V_2:=W_0^{1,2}$ and $H_2:=L^2$,
$$F_2(u,v):=B(u)+\Delta v+c_1v-c_2v^3,~~ u\in H_1,~v\in V_2,$$
where $c_1,c_2\geq 0$ are some constants and map $B: H_1\to H_2$ is Lipschitz.
\end{Rem}

The definition of solution to (\ref{e1}) is stated as follows.
\begin{definition}\label{d1}
For any $\epsilon,\alpha>0$, we call a continuous $H_1\times H_2$-valued $(\mathscr{F}_t)_{t\geq 0}$-adapted process $(X^{\epsilon,\alpha}_t,Y^{\epsilon,\alpha}_t)_{t\in[0,T]}$ is a solution of (\ref{e1}), if for its $dt\times \mathbb{P}$-equivalent class $(\hat{X}^{\epsilon,\alpha}_t,\hat{Y}^{\epsilon,\alpha}_t)_{t\in[0,T]}$ satisfying
\begin{equation*}
\hat{X}^{\epsilon,\alpha}\in L^{\gamma_1}\big([0,T]\times\Omega,dt\times\mathbb{P};V_1\big)\cap L^2\big([0,T]\times\Omega,dt\times\mathbb{P};H_1\big),
\end{equation*}
\begin{equation*}
\hat{Y}^{\epsilon,\alpha}\in L^{\gamma_2}\big([0,T]\times\Omega,dt\times\mathbb{P};V_2\big)\cap L^2\big([0,T]\times\Omega,dt\times\mathbb{P};H_2\big),
\end{equation*}
where $\gamma_1,\gamma_2$ is the same as defined in $({\mathbf{A}}{\mathbf{2}})$ and $({\mathbf{H}}{\mathbf{3}})$, respectively, and $\mathbb{P}$-a.s.
\begin{equation*}
\left\{ \begin{aligned}
&d\bar{X}^{\epsilon,\alpha}_t=\big[A(\bar{X}^{\epsilon,\alpha}_t)+F_1(\bar{X}^{\epsilon,\alpha}_t,Y^{\epsilon,\alpha}_t)\big]dt+\sqrt{\epsilon}G_1(\bar{X}^{\epsilon,\alpha}_t)dW_t,\\
&d\bar{Y}^{\epsilon,\alpha}_t=\frac{1}{\alpha}F_2(\bar{X}^{\epsilon,\alpha}_t,\bar{Y}^{\epsilon,\alpha}_t)dt+\frac{1}{\sqrt{\alpha}}G_2dW_t,
\end{aligned} \right.
\end{equation*}
here $\bar{X}^{\epsilon,\alpha}$ (resp. $\bar{Y}^{\epsilon,\alpha}$) is any $V_1$ (resp. $V_2$) valued progressively measurable $dt\times\mathbb{P}$-version of $\hat{X}^{\epsilon,\alpha}$ (resp. $\hat{Y}^{\epsilon,\alpha}$).
\end{definition}

Following the similar calculations as in the proof of \cite[Theorem 2.3]{LRSX1}, the existence and uniqueness of solutions to system (\ref{e1}) can be formulated as follows.
\begin{lemma}
Suppose that the assumptions $({\mathbf{A}}{\mathbf{1}})$-$({\mathbf{A}}{\mathbf{3}})$ and $({\mathbf{H}}{\mathbf{1}})$-$({\mathbf{H}}{\mathbf{4}})$ hold. For each $\epsilon,\alpha>0$ and starting point $(x,y)\in H_1\times H_2$, Eq.~(\ref{e1}) has a unique solution $(X^{\epsilon,\alpha}_t, Y^{\epsilon,\alpha}_t)_{t\in[0,T]}$.
\end{lemma}

Let us now recall some definitions and classical results of
LDP. Let $\{X^\varepsilon\}$ denote a family of
random variables defined on a probability space
$(\Omega,\mathscr{F},\mathbb{P})$ taking values in a Polish
space $E$. Shortly speaking, the LDP
characterizes the exponential decay of the probability
distributions with respect to certain kinds of extreme or remote tail events. The rate of
such exponential decay is described by the ``rate function".

\begin{definition}(Rate function) A function $I: E\to [0,+\infty)$ is called
a rate function if $I$ is lower semicontinuous. Moreover, a rate function $I$
is called a {\it good rate function} if  the level set $\{x\in E: I(x)\le
K\}$ is compact for each constant $K<\infty$.
\end{definition}

\begin{definition}(Large deviation principle) The random variable sequence
 $\{X^\varepsilon\}$ is said to satisfy
 the LDP on $E$ with rate function
 $I$ if  the following lower and upper bound conditions hold,

(i) (Lower bound) For any open set $G\subset E$:
$$\liminf_{\varepsilon\to 0}
   \varepsilon \log \mathbb{P}(X^{\varepsilon}\in G)\geq -\inf_{x\in G}I(x).$$

(ii) (Upper bound) For any closed set $F\subset E$:
$$ \limsup_{\varepsilon\to 0}
   \varepsilon \log \mathbb{P}(X^{\varepsilon}\in F)\leq
  -\inf_{x\in F} I(x).
$$
\end{definition}

Now we recall the equivalence between
the LDP and the Laplace principle which is defined as follows (cf. \cite{C1,DE,DZ}).

\begin{definition}\label{d1}(Laplace principle) The sequence $\{X^\varepsilon\}$ is
said to satisfy the Laplace principle on $E$ with a rate function
$I$ if for each bounded continuous real-valued function $h$ defined
on $E$, we have
$$\lim_{\varepsilon\to 0}\varepsilon \log \mathbb{E}\left\lbrace
 \exp\left[-\frac{1}{\varepsilon} h(X^{\varepsilon})\right]\right\rbrace
= -\inf_{x\in E}\left\{h(x)+I(x)\right\}.$$
\end{definition}

\begin{lemma}\label{l1}
(Varadhan's Lemma \cite{V1}) Let $E$ be a Polish space and an $E$-valued random sequence $\{X^\varepsilon\}$ fulfills the LDP with a good rate function $I$. Then $\{X^\varepsilon\}$ fulfills the Laplace principle on $E$ with the same rate function $I$.
\end{lemma}

\begin{lemma}\label{l2}
(Bryc's converse \cite{DZ}) The Laplace principle implies the LDP with the same good rate function.
\end{lemma}

Combining Lemma \ref{l1} and \ref{l2} yields that if $E$ is a Polish space and $I$ is a good rate function, then the
LDP and Laplace principle are equivalent.

Let
$$\mathcal{A}=\left\lbrace \phi: \phi\  \text{is  $U$-valued
 $\mathscr{F}_t$-predictable process and}\
  \int_0^T\|\phi_s(\omega)\|^2_U\d s<\infty \  \mathbb{P}\text{-}a.s.\right\rbrace, $$
  and
$$S_M=\left\lbrace \phi\in L^2([0,T], U):
\int_0^T\|\phi_s\|^2_{U}  ds\leq M
 \right\rbrace.$$
It is well-known that $S_M$ endowed with the weak topology is a Polish space (here and in the sequel of this article, we
 always consider the weak topology on $S_M$ unless stated otherwise). We also define
$$\mathcal{A}_M=\left\{\phi\in\mathcal{A}: \phi_{\cdot}(\omega)\in S_M, ~\mathbb{P}\text{-}a.s.\right\}.$$
Let $E$ be a Polish space, for any $\varepsilon>0$, suppose $\mathcal{G}^\varepsilon: C([0,T]; U_1)\rightarrow
E$ is a measurable map and $X^\varepsilon=
\mathcal{G}^\varepsilon(W_{\cdot})$.

We now
formulate the sufficient condition for the Laplace
 principle (equivalently, the LDP) of $X^\varepsilon$ as
 $\varepsilon\rightarrow0$.\\
\textbf{Condition (A)}: There exists a measurable map $\mathcal{G}^0: C([0,T];
U_1)\rightarrow E$ for which the following two conditions hold:

(i) Let $\{\phi^\varepsilon: \varepsilon>0\}\subset \mathcal{A}_M$ for
some $M<\infty$. If $\phi^\varepsilon$ converge to $\phi$ in distribution
as $S_M$-valued random elements, then
$$\mathcal{G}^\varepsilon\left(W_\cdot+\frac{1}{\sqrt{\varepsilon}} \int_0^\cdot
\phi^\varepsilon_s\ d s \right)\rightarrow \mathcal{G}^0\left(\int_0^\cdot
 \phi_s\ d s \right)  $$
in distribution as $\varepsilon\rightarrow 0$.

(ii) For each $M<\infty$, the set
$$K_M=\left\{\mathcal{G}^0\left(\int_0^\cdot \phi_s\d
s\right):
 \phi\in S_M\right\}$$
is a compact subset of $E$.

In \cite{BD} Budhiraja and Dupuis presented the following powerful result for
the Laplace principle (equivalently, the LDP).
\begin{lemma}\label{l3}\cite[Theorem 4.4]{BD}  If
$X^\varepsilon=\mathcal{G}^\varepsilon(W_{\cdot})$ and \textbf{Condition (A)}
holds, then the family $\{X^\varepsilon\}$ satisfies the Laplace
principle (hence LDP) on $E$ with the good
rate function $I$
\begin{equation}\label{rf}
I(f)=\inf_{\left\{\phi\in L^2([0,T]; U):\  f=\mathcal{G}^0(\int_0^\cdot
\phi_sds)\right\}}\left\lbrace\frac{1}{2}
\int_0^T\|\phi_s\|_U^2ds \right\rbrace,
\end{equation}
where infimum over an empty set is taken as $+\infty$.
\end{lemma}


\subsection{Main results}
It is well-known that $\big(C([0,T]; H_1)\cap L^{\gamma_1}([0,T];
V_1),d(\cdot,\cdot)\big)$ is a Polish space with respect to the metric
\begin{equation}\label{rho}
d(f,g):=\sup_{t\in[0,T]}\|f_t-g_t\|_{H_1}
+\left(\int_0^T\|f_t-g_t\|_{V_1}^{\gamma_1}\ d t \right)^{\frac{1}{\gamma_1}}.
\end{equation}
According to the Yamada-Watanabe
theorem, there exists a Borel-measurable
function
\begin{equation}
\mathcal{G}^\varepsilon: C([0,T]; U_1)\rightarrow
 C([0,T]; H_1)\cap L^{\gamma_1}([0,T]; V_1)
 \end{equation}
such that
$X^{\epsilon,\alpha}=\mathcal{G}^\varepsilon (W_{\cdot})$, $\mathbb{P}\text{-}a.s.$, where $X^{\epsilon,\alpha}$ is the unique (strong) solution to the slow equation of (\ref{e1}).

Now we consider the following skeleton equation
\begin{equation}\label{e2}
\frac{d\bar{X}^{\phi}_t}{dt}=\big[A(\bar{X}^{\phi}_t)+\bar{F}_1(\bar{X}^{\phi}_t)\big]+G_1(\bar{X}^{\phi}_t)\phi_t,~~\bar{X}^{\phi}_0=x,
\end{equation}
where $\phi\in L^2([0,T];U)$ and $\bar{F}_1(x):=\int_{H_2}F_1(x,y)\mu^x(dy),~x\in H_1$ for $\mu^x$ being the unique invariant measure of the Markov semigroup to the frozen equation (see Eq.~(\ref{e3})).

The existence and uniqueness of solutions to Eq.~(\ref{e2}) for any $\phi\in L^2([0,T];U)$ will be proved in the next section (see Lemma \ref{l5}).  Furthermore, we define
the map $\mathcal{G}^0: C([0,T]; U_1)\rightarrow C([0,T]; H_1)\cap
L^{\gamma_1}([0,T]; V_1)$ by
$$\mathcal{G}^0\Big(\int_0^{\cdot}\phi_sds\Big):=\bar{X}^{\phi}.$$

Now we can state the first main result of this work.
\begin{theorem}\label{t1}
Assume that Hypothesis \ref{h1} and \ref{h2} hold. If
\begin{equation}\label{h5}
\lim_{\epsilon\to0}\alpha(\epsilon)=0~\text{and}~\lim_{\epsilon\to0}\frac{\alpha}{\epsilon}=0,
\end{equation}
 then as $\epsilon\to 0$, $\{X^{\epsilon,\alpha}:\epsilon>0\}$
satisfies the LDP on $C([0,T]; H_1)\cap L^{\gamma_1}([0,T]; V_1)$ with the
good rate function $I$ given by $(\ref{rf})$.
\end{theorem}

We need to point out that the theorem above is not applicable to the stochastic fast-diffusion equation and singular stochastic p-Laplace equation (i.e. $1<p<2$) directly since the condition $({\mathbf{A}}{\mathbf{2}})$ does not hold. However, if we replace $({\mathbf{A}}{\mathbf{2}})$ by the following local monotonicity and coercivity condition as in \cite{LR2}:
\begin{enumerate}
\item [({\textbf{A}}{\textbf{4}})]
\begin{eqnarray*}
2_{V_1^*}\langle A(u)-A(v), u-v\rangle_{V_1}+
 \|G_1(u)-G_1(v)\|_{L_2(U,H_1)}^2\leq (K+\rho(v))\|u-v\|_{H_1}^2,
\end{eqnarray*}
where $\rho$ is the same as the one defined in $({\mathbf{A}}{\mathbf{2}})$. Moreover,
\begin{eqnarray*}
2{ }_{V_1^*}\langle A(u),
 u\rangle_{V_1}+\|G_1(u)\|^{2}_{L_2(U,H_1)}\leq-\theta_1\|u\|_{V_1}^{\gamma_1}
    +K(1+ \|u\|_{H_1}^2).
\end{eqnarray*}
\end{enumerate}
Then the LDP for multi-scale stochastic fast-diffusion equation and singular stochastic p-Laplace equation can be proved on $C([0,T];H_1)$ as the following theorem stated.

\begin{theorem}\label{t3}
Assume that $({\mathbf{A}}{\mathbf{1}})$, $({\mathbf{A}}{\mathbf{3}})$, $({\mathbf{A}}{\mathbf{4}})$ and Hypothesis \ref{h2} hold. If the condition\ (\ref{h5}) holds, then as $\epsilon\to 0$, $\{X^{\epsilon,\alpha}:\epsilon>0\}$
satisfies the LDP on $C([0,T]; H_1)$ with the
good rate function $I$ given by $(\ref{rf})$.
\end{theorem}

\begin{Rem}
It should be mentioned that compared with $({\mathbf{A}}{\mathbf{4}})$, the condition $({\mathbf{A}}{\mathbf{2}})$ is stronger than $({\mathbf{A}}{\mathbf{4}})$, the key point of using $({\mathbf{A}}{\mathbf{2}})$ is to derive an additional convergence $($i.e.~$\{X^{\epsilon,\alpha}:\epsilon>0\}$ also satisfies the LDP on $L^{\gamma_1}([0,T]; V_1)$$)$.
\end{Rem}

\begin{Rem}
In order to derive the LDP for $\{X^{\epsilon,\alpha}:\epsilon>0\}$, first we will show the existence and uniqueness of solutions with some necessary priori estimates to the skeleton equation and controlled stochastic equations (see Lemma \ref{l5} and \ref{l6} below), then the next step is  to verify two important results on the compactness of the level sets of rate function and the weak convergence of the stochastic control equations. In particular, the Khasminskii's method based on time discretization \cite{K} and some techniques of stopping time will be applied to the proof of the weak convergence and the compactness of the level sets of rate function.

Throughout this paper, we use $C_{p_1,p_2,\cdots}$ to denote some generic positive constant whose value may change from
line to line, and depends only on the designated variables $p_1,p_2,\cdots$.
\end{Rem}
\section{Examples} \label{example}
\setcounter{equation}{0}
 \setcounter{definition}{0}
 The main results formulated in Theorem \ref{t1} and \ref{t3} can be used to deal with a very large family of SPDE models directly, which not only extends or improves some existing works using mild solution approach for a class of two-time-scale semilinear SPDEs such as stochastic reaction-diffusion equations, stochastic Burgers equations (see e.g. \cite{HSS,SWXY,WRD}),  but also obtain the LDP for several new SPDE models with respect to the two-time-scale case.

In this section, we will denote by $\Lambda\subseteq\mathbb{R}^d$ an open bounded domain with a smooth boundary.  Let
$C_0^\infty(\Lambda, \mathbb{R}^d)$ be the space of all infinitely differentiable functions from $\Lambda$ to $\mathbb{R}^d$ with compact support.  For $p\ge 1$, let $L^p(\Lambda, \mathbb{R}^d)$ denote the vector valued $L^p$-space with the norm $\|\cdot\|_{L^p}$.
For each integer $m>0$, we use $W_0^{m,p}(\Lambda, \mathbb{R}^d)$ to denote the classical Sobolev space defined on $\Lambda$
taking values in $\mathbb{R}^d$ with the (equivalent) norm:
$$ \|u\|_{W^{m,p}} = \left( \sum_{0\le |\alpha|\le m} \int_{\Lambda} |D^\alpha u|^pd x \right)^\frac{1}{p}.$$

Below we would like to recall the so-called Gagliardo-Nirenberg interpolation inequality (cf. \cite[Theorem 2.1.5]{Taira}) for the reader's convenience.

If for any $m,n\in\mathbb{N}$ and $q\in[1,\infty]$ satisfying
\[
\frac{1}{q}=\frac{1}{2}+\frac{n}{d}-\frac{m \theta}{d},\ \frac{n}{m}\le\theta\le1,
\]
then there is a constant $C>0$ such that
\begin{equation}
\|u\|_{W^{n,q}}\le C\|u\|_{W^{m,2}}^{\theta}\|u\|_{L^{2}}^{1-\theta},\ \ u\in W^{m,2}(\Lambda, \mathbb{R}^d).\label{GN_inequality}
\end{equation}

\subsection{Stochastic porous media equation}\label{example 1}
Let us denote by $(E,\mathcal{M},\textbf{m})$ a separable probability space and $(L,\mathcal{D}(L))$ a negative definite linear self-adjoint map defined on
$(L^2(\textbf{m}),\langle\cdot,\cdot\rangle)$, which has discrete spectrum with eigenvalues
$$0>-\lambda_1\geq-\lambda_2\geq\cdots\rightarrow-\infty.$$
Let $H_1$ be the topological dual space of $\mathcal{D}(\sqrt{-L})$, which is endowed with the scalar product
$$\langle u,v\rangle_{H_1}:=\int_E\big(\sqrt{-L}u(\xi)\big)\cdot\big(\sqrt{-L}v(\xi)\big)d\xi,~u,v\in H_1,$$
then identify $L^2(\textbf{m})$ with its dual, one can obtain the following dense and continuous embedding
$$\mathcal{D}(\sqrt{-L})\subseteq L^2(\textbf{m})\subseteq H_1.$$
Consequently, due to this embedding, we can define
$$V_2:=\mathcal{D}(\sqrt{-L}),~H_2:=L^2(\textbf{m}).$$
Suppose that $L^{-1}$ is continuous in $V_1:=L^{r+1}(\textbf{m})$, where $r>1$ is a fixed number. Then we can give a presentation of its dual space $V_1^*$ by the following embedding
$$V_1\subset H_1\cong \mathcal{D}(\sqrt{-L})\subset V_1^*,$$
where $\cong$ is understood through $\sqrt{-L}$.

Consider the two-time-scale stochastic porous media equation as follows,
\begin{equation}\label{PME}
\left\{ \begin{aligned}
&dX^{\epsilon,\alpha}_t=[L\Psi(X^{\epsilon,\alpha}_t)+\Phi(X^{\epsilon,\alpha}_t)]dt+F_1(X^{\epsilon,\alpha}_t,Y^{\epsilon,\alpha}_t)dt+\sqrt{\epsilon}G_1(X^{\epsilon,\alpha}_t)dW_t,\\
&dY^{\epsilon,\alpha}_t=\frac{1}{\alpha}[LY^{\epsilon,\alpha}_t+F_2(X^{\epsilon,\alpha}_t,Y^{\epsilon,\alpha}_t)]dt+\frac{1}{\sqrt{\alpha}}G_2dW_t,\\
&X^{\epsilon,\alpha}_0=x\in H_1,~Y^{\epsilon,\alpha}_0=y\in H_2,
\end{aligned} \right.
\end{equation}
here $W_t$ is a cylindrical Wiener process defined
on a probability space $(\Omega,\mathscr {F},\mathscr
{F}_t,\mathbb{P})$ taking values in a sparable Hilbert space $U$, $\Psi,\Phi:\mathbb{R}\rightarrow\mathbb{R}$
are continuous and measurable maps such that there exist some constants $\theta_1>0$ and $K$,
\begin{eqnarray}
&&|\Psi(s)|+|\Phi(s)|\leq K(1+|s|^r),~~s\in \mathbb{R};
\label{PME3}\\&&-\langle\Psi(u)-\Psi(v),u-v\rangle-\langle\Phi(u)-\Phi(v),L^{-1}(u-v)\rangle
\nonumber\\\leq\!\!\!\!\!\!\!\!&&-\theta_1\|u-v\|_{V_1}^{r+1}+K\|u-v\|_{H_1}^2,~~u,v\in V_1,\label{PME4}
\end{eqnarray}
and the measurable maps
$$F_1:H_1\times H_2\to H_1,~G_1:V_1\to L_2(U,H_1),~F_2:H_1\times V_2\to V_2^*,~G_2\in L_2(U,H_2)$$
are Lipschitz continuous, i.e.,
\begin{eqnarray}
\!\!\!\!\!\!\!\!&&\|F_1(u_1,v_1)-F_1(u_2,v_2)\|_{H_1}\leq C\big(\|u_1-u_2\|_{H_1}+\|v_1-v_2\|_{H_2}\big),\label{55}
\\
\!\!\!\!\!\!\!\!&&\|G_1(u)-G_1(v)\|_{L_2(U,H_1)}\leq C\|u-v\|_{H_1},\label{56}
\\
\!\!\!\!\!\!\!\!&&\|F_2(u_1,v_1)-F_2(u_2,v_2)\|_{H_1}\leq C\|u_1-u_2\|_{H_1}+L_{F_2}\|v_1-v_2\|_{H_2},\label{57}
\end{eqnarray}
here $L_{F_2}$  represents the Lipschitz constant with respect to second variable of $F_2$. Furthermore, we assume that
the smallest eigenvalue $\lambda_1$ of map $L$ satisfies
\begin{equation}\label{60}
\lambda_1-L_{F_2}>0.
\end{equation}

\begin{theorem}(stochastic porous media equation)\label{main result PME2}
Assume that $\Psi,\Phi$ satisfy the above conditions \eref{PME3}-\eref{PME4} and $F_1,F_2,G_1$ satisfy (\ref{55})-(\ref{60}), if
\begin{equation*}
\lim_{\epsilon\to0}\alpha(\epsilon)=0~\text{and}~\lim_{\epsilon\to0}\frac{\alpha}{\epsilon}=0,
\end{equation*}
then  $\{X^{\epsilon,\alpha}:\epsilon>0\}$ in (\ref{PME})
satisfies the LDP on $C([0,T]; H_1)\cap L^{r+1}([0,T]; V_1)$ with the
good rate function $I$ given by $(\ref{rf})$.
\end{theorem}

\begin{proof}
It is known that the map $A:=L\Psi+\Phi$ satisfy $({\mathbf{A}}{\mathbf{1}})$-$({\mathbf{A}}{\mathbf{3}})$ for $\rho\equiv0,~\beta_1=0,~\gamma_1=r+1$, we refer to \cite[Example 4.1.11]{LR1} for some details.
Moreover, one can easily prove that the assumptions presented in Theorem \ref{t1} hold via (\ref{55})-(\ref{60}). Therefore, Theorem \ref{main result PME2} is a direct consequence of Theorem \ref{t1}.   \hspace{\fill}$\Box$
\end{proof}

\begin{Rem}
(i) A typical example is that $L=\Delta$, the Laplace operator on a smooth bounded domain in a
complete Riemannian manifold with Dirichlet boundary,  and
$$\Psi(s):=|s|^{r-1}s, ~~\Phi(s):=s,~s\in\mathbb{R}.$$

(ii) To the best of our knowledge, there is no LDP result in the literature established for multi-scale quasilinear SPDEs such as stochastic porous media equation  (see Section \ref{laplace} and \ref{fast} for other types of quasi-linear SPDEs).
\end{Rem}
\subsection{Stochastic $p$-Laplace equation}\label{laplace}
We introduce the two-time-scale stochastic $p$-Laplace equation as follows,
\begin{equation}\label{PLAP}
\left\{ \begin{aligned}
&dX^{\epsilon,\alpha}_t=[div(|\nabla X^{\epsilon,\alpha}_t|^{p-2}\nabla X^{\epsilon,\alpha}_t)-C|X^{\epsilon,\alpha}_t|^{q-2}X^{\epsilon,\alpha}_t]dt+F_1(X^{\epsilon,\alpha}_t,Y^{\epsilon,\alpha}_t)dt+\sqrt{\epsilon}G_1(X^{\epsilon,\alpha}_t)dW_t,\\
&dY^{\epsilon,\alpha}_t=\frac{1}{\alpha}[\Delta Y^{\epsilon,\alpha}_t+F_2(X^{\epsilon,\alpha}_t,Y^{\epsilon,\alpha}_t)]dt+\frac{1}{\sqrt{\alpha}}G_2dW_t,\\
&X^{\epsilon,\alpha}_0=x\in H_1,~Y^{\epsilon,\alpha}_0=y\in H_2,
\end{aligned} \right.
\end{equation}
where $C>0$, $2\leq p\leq\infty,1\leq q\leq p$ and $W_t$ is a cylindrical Wiener process in $U$ defined
on a probability space $(\Omega,\mathscr {F},\mathscr
{F}_t,\mathbb{P})$.

We now consider the following Gelfand triple for the slow component
$$V_1:=W_0^{1,p}(\Lambda)\subset H_1:=L^2(\Lambda)\subset(W_0^{1,p}(\Lambda))^*=V_1^*$$
and
the Gelfand triple for the fast component
$$V_2:=W_0^{1,2}(\Lambda)\subset H_2:=L^2(\Lambda)\subset(W_0^{1,2}(\Lambda))^*=V_2^*.$$

\begin{theorem}(stochastic $p$-Laplace equation)\label{main result PLAP}
Assume that $F_1,F_2,G_1$ satisfy (\ref{55})-(\ref{60}) with $\Delta$ replacing operator $L$, if
\begin{equation*}
\lim_{\epsilon\to0}\alpha(\epsilon)=0~\text{and}~\lim_{\epsilon\to0}\frac{\alpha}{\epsilon}=0,
\end{equation*}
then  $\{X^{\epsilon,\alpha}:\epsilon>0\}$ in (\ref{PLAP})
satisfies the LDP on $C([0,T]; H_1)\cap L^{p}([0,T]; V_1)$ with the
good rate function $I$ given by $(\ref{rf})$.
\end{theorem}

\begin{proof}
It is well-known that the $p$-Laplace operator satisfies the hemicontinuity, classical monotonicity and growth
condition $({\mathbf{A}}{\mathbf{1}})$-$({\mathbf{A}}{\mathbf{3}})$ for $\rho\equiv0,~\beta_1=0,~\gamma_1=p$, we can see e.g.~\cite[Example 5.5]{L1} for the precise proof.
By (\ref{55})-(\ref{60}), one can easily check that the assumptions presented in Theorem \ref{t1} hold. Thus, the conclusion follows from  Theorem \ref{t1}. \hspace{\fill}$\Box$
\end{proof}

\begin{Rem}
This theorem can not be applied to the singular $p$-Laplace equation (i.e.~$1<p<2$) directly, however, one can use Theorem \ref{t3} to derive the LDP for system (\ref{PLAP}) on $C([0,T];H_1)$ (cf.~\cite[Example 4.1.9]{LR1}).
\end{Rem}

\subsection{Stochastic fast-diffusion equation}\label{fast}
Suppose the same setting as in Section \ref{example 1} for the case of $0<r<1$, the two-time-scale stochastic fast-diffusion equation is given by
\begin{equation}\label{FDE}
\left\{ \begin{aligned}
&dX^{\epsilon,\alpha}_t=[L\Psi(X^{\epsilon,\alpha}_t)+F_1(X^{\epsilon,\alpha}_t,Y^{\epsilon,\alpha}_t)]dt+\sqrt{\epsilon}G_1(X^{\epsilon,\alpha}_t)dW_t,\\
&dY^{\epsilon,\alpha}_t=\frac{1}{\alpha}[LY^{\epsilon,\alpha}_t+F_2(X^{\epsilon,\alpha}_t,Y^{\epsilon,\alpha}_t)]dt+\frac{1}{\sqrt{\alpha}}G_2dW_t,\\
&X^{\epsilon,\alpha}_0=x\in H_1,~Y^{\epsilon,\alpha}_0=y\in H_2,
\end{aligned} \right.
\end{equation}
here $W_t$ stands for a cylindrical Wiener process defined
on a probability space $(\Omega,\mathscr {F},\mathscr
{F}_t,\mathbb{P})$ taking values in a sparable Hilbert space $U$, $\Psi:\mathbb{R}\rightarrow\mathbb{R}$
is continuous and measurable map such that there exist some constants $\delta>0$ and $K$,
\begin{eqnarray}
&&|\Psi(s)|\leq K(1+|s|^r),~~s\in \mathbb{R};
\label{FDE1}\\&&(\Psi(s_1)-\Psi(s_2))(s_1-s_2)\geq \delta|s_1-s_2|^2(|s_1|\vee|s_2|)^{r-1},~s_1,s_2\in\mathbb{R}.\label{FDE2}
\end{eqnarray}

\begin{theorem}(stochastic fast-diffusion equation)\label{main result FDE}
Suppose that $\Psi$ satisfies the conditions (\ref{FDE1})-(\ref{FDE2}) and $F_1,F_2,G_1$ satisfy (\ref{55})-(\ref{60}) above, if
\begin{equation*}
\lim_{\epsilon\to0}\alpha(\epsilon)=0~\text{and}~\lim_{\epsilon\to0}\frac{\alpha}{\epsilon}=0,
\end{equation*}
then  $\{X^{\epsilon,\alpha}:\epsilon>0\}$ in (\ref{FDE})
satisfies the LDP on $C([0,T]; H_1)$ with the
good rate function $I$ given by $(\ref{rf})$.
\end{theorem}

\begin{proof}
Following the similar arguments as in Section \ref{example 1}, the map $A:=L\Psi$ satisfies conditions $({\mathbf{A}}{\mathbf{1}})$, $({\mathbf{A}}{\mathbf{3}})$ and  $({\mathbf{A}}{\mathbf{4}})$ for $\rho\equiv0,~\beta_1=0,~\gamma_1=r+1$, one can see also \cite[Example 4.1.11]{LR1} for the detailed calculations. According to (\ref{55})-(\ref{60}), the assumptions given in Theorem \ref{t1} hold.
The assertion formulated in Theorem \ref{main result FDE} follows from Theorem \ref{t3}. \hspace{\fill}$\Box$
\end{proof}

\begin{Rem}
(i) A specific example fulfilling (\ref{FDE1})-(\ref{FDE2}) is that $\Psi(s):=|s|^{r-1}s$, $s\in\mathbb{R}$ for $0<r<1$, which characterizes the classical fast-diffusion equation.

(ii) In this case, for simplicity, we consider the situation that the embedding $L^{r+1}(\textbf{m})\subset H$ is continuous and dense, one can see \cite[Remark 4.1.15]{LR1} for the sufficient condition to guarantee such assumption holds.
\end{Rem}

Besides the above two-time-scale quasilinear type SPDEs, our main results  are also applicable to a large class of semilinear SPDEs  satisfying local monotonicity condition, for instance, the stochastic Burgers type equations, stochastic Navier-Stokes equation and other hydrodynamical type models. For the two-time-scale stochastic Burgers equation,  it has been studied in the work \cite{LRSX1} using the mild solution method.  In this paper we apply the variational approach to get the LDP and our framework can cover more concrete examples.
\subsection{Stochastic Burgers type equation}
The first semilinear example is the two-time-scale stochastic Burgers type equation,
\begin{equation}\label{semilinear}
\left\{ \begin{aligned}
&dX^{\epsilon,\alpha}_t=[\Delta X^{\epsilon,\alpha}_t+\langle f(X^{\epsilon,\alpha}_t),\nabla X^{\epsilon,\alpha}_t\rangle+h(X^{\epsilon,\alpha}_t)]dt+F_1(X^{\epsilon,\alpha}_t,Y^{\epsilon,\alpha}_t)dt+\sqrt{\epsilon}G_1(X^{\epsilon,\alpha}_t)dW_t,\\
&dY^{\epsilon,\alpha}_t=\frac{1}{\alpha}[\Delta Y^{\epsilon,\alpha}_t+F_2(X^{\epsilon,\alpha}_t,Y^{\epsilon,\alpha}_t)]dt+\frac{1}{\sqrt{\alpha}}G_2dW_t,\\
&X^{\epsilon,\alpha}_0=x\in H_1,~Y^{\epsilon,\alpha}_0=y\in H_2,
\end{aligned} \right.
\end{equation}
where $f=(f_1,\cdots,f_d):\mathbb{R}\rightarrow\mathbb{R}^d$ is a
Lipschitz continuous function and $\langle\cdot,\cdot\rangle$ denotes the scalar
product in $\mathbb{R}^d$, $W_t$ stands for a cylindrical Wiener process defined
on a probability space $(\Omega,\mathscr {F},\mathscr
{F}_t,\mathbb{P})$ taking values in a sparable Hilbert space $U$. Let
$h:\mathbb{R}\rightarrow \mathbb{R}$ denote a continuous function with
$h(0)=0$ such that for some constants $C,r,s\in [0,\infty)$
\begin{equation}\label{5.21}
|h(x)|\leq C(|x|^r+1),~x\in \mathbb{R};
\end{equation}
\begin{equation}\label{5.22}
(h(x)-h(y))(x-y)\leq C(1+|y|^s)(x-y)^2,~x,y\in \mathbb{R}.
\end{equation}

Consider the Gelfand triple for the slow component
$$V_1:=W_0^{1,2}(\Lambda)\subset H_1:=L^2(\Lambda)\subset(W_0^{1,2}(\Lambda))^*=V_1^*$$
and also the Gelfand triple for the fast component
$$V_2:=W_0^{1,2}(\Lambda)\subset H_2:=L^2(\Lambda)\subset(W_0^{1,2}(\Lambda))^*=V_2^*.$$

Now we state the main result on the LDP for two-time-scale  stochastic Burgers type equation.

\begin{theorem}\label{SLSPDE} (stochastic Burgers type equation) Assume that $h$ satisfies (\ref{5.21})-(\ref{5.22}), $F_1,F_2,G_1$ satisfy (\ref{55})-(\ref{60}) above with $\Delta$ replacing operator $L$.
 Assume

Case 1: $d=1, r=2, s=2$,

Case 2: $d=2, r=2, s=2$, and $f$ is bounded,

Case 3: $d=3$, $r=2$, $s=\frac{4}{3}$, and $f$ is bounded measurable function independent of $X^{\epsilon,\alpha}_t$.\\
If
\begin{equation*}
\lim_{\epsilon\to0}\alpha(\epsilon)=0~\text{and}~\lim_{\epsilon\to0}\frac{\alpha}{\epsilon}=0,
\end{equation*}then the family $\{X^{\epsilon,\alpha}:\epsilon>0\}$ in (\ref{semilinear}) satisfies the LDP on $C([0,T]; H_1)\cap L^{2}([0,T]; V_1)$ with the
good rate function $I$ given by $(\ref{rf})$.
\end{theorem}
\begin{proof}
Let us denote the operator
$$A(u):=\widehat{A}(u)+h(u):=\Delta u+\langle f(u),\nabla u\rangle+h(u),~u\in V_1.$$
Following from \cite[Example 3.2]{LR2}, it is easy to obtain that map $A$ satisfies $({\mathbf{A}}{\mathbf{1}})$-$({\mathbf{A}}{\mathbf{2}})$ with $\gamma_1=2$. For the condition $({\mathbf{A}}{\mathbf{3}})$, we have
$$\|\widehat{A}(u)+h(u)\|_{V_1^*}^2\leq C(\|\widehat{A}(u)\|_{V_1^*}^2+\|h(u)\|_{V_1^*}^2).$$
The first term of right hand side of the above inequality fulfills
$$\|\widehat{A}(u)\|_{V_1^*}^2\leq C(1+\|u\|_{V_1^*}^2)(1+\|u\|_{H_1^*}^\nu),$$
where $\nu=2$ in Case 1 and $\nu=0$ in Case 2. For the second term, making use of H\"{o}lder's inequality and Gagliardo-Nirenberg interpolation inequality (\ref{GN_inequality}),  we can get that
\begin{eqnarray*}
|{}_{V_1^*}\langle h(u),v\rangle_{V_1}|^2\leq\!\!\!\!\!\!\!\!&&
\left\{ \begin{aligned}
&\|v\|_{L^\infty}^2(1+\|u\|_{L^2}^4),~d=1,\\
&\|v\|_{L^2}^2(1+\|u\|_{L^4}^4),~d=2,\\
&\|v\|_{L^6}^2(1+\|u\|_{L^{\frac{12}{5}}}^4),~d=3
\end{aligned} \right.
\nonumber\\
\leq\!\!\!\!\!\!\!\!&&\|v\|_{V_1}^2(1+\|u\|_{V_1}^2\|u\|_{H_1}^2).
\end{eqnarray*}
Then the condition $({\mathbf{A}}{\mathbf{3}})$ are satisfied with $\gamma_1=\beta_1=2$. Consequently, the assertion follows from Theorem \ref{t1}. \hspace{\fill}$\Box$
\end{proof}

\begin{Rem}
If we take $d=1$,~$f(x)=x$~and~$h=0$, Theorem \ref{SLSPDE} can be used to deal with the classical stochastic Burgers equation. Moreover, it should be noted that one can also allow a polynomial control term $h$ in the drift of \eref{semilinear}. For instance, we may consider $g(x)=-x^3+c_1 x^2+c_2 x~(c_1 , c_2 \in \mathbb{R})$  and show that \eref{5.21}-\eref{5.22} are satisfied. Thus \eref{semilinear} also covers some two-time-scale stochastic reaction-diffusion type equations.
\end{Rem}

\subsection{Stochastic 2D hydrodynamical type systems}
The main purpose of this subsection is to consider the two-time-scale stochastic 2D hydrodynamical type systems, which cover a wide class of mathematical models from fluid dynamics (cf. \cite{CM,E}).

For the slow component, let $H_1$ be a separable Hilbert space equipped with norm $|\cdot|$, $A$ be an (unbounded) positive linear self-adjoint operator on $H_1$. Define
$V_1=\mathscr{D}(A^\frac{1}{2})$, and the associated norm $\|v\|=|A^\frac{1}{2}v|$ for any $v\in V_1$. Let $V_1^*$ be the dual space of $V_1$ with respect to the scalar product
$(\cdot,\cdot)$ on $H_1$. Due to this, one can consider a Gelfand triple $V_1 \subset H_1 \subset V_1^*$. Let us denote by $\langle{u},v\rangle$ the dualization between $u\in V_1$ and $v\in V_1^*$, and it is easy to see that $\langle{u},v\rangle =(u,v)$ if $u\in V_1$, $v\in H_1$. There exists an orthonormal basis $\{e_k\}_{k\geq 1}$ on $H_1$ of eigenfunctions of $A$, and the increasing eigenvalue sequence $0<\lambda_1\leq\lambda_2\leq...\leq\lambda_n\leq...\uparrow\infty.$

Let $B:V_1\times V_1 \to V_1^*$ be a continuous map fulfilling

\begin{enumerate}
\item[(C1)] $B: V_1 \times V_1 \to V_1^*$ is a continuous bilinear map.

\item [(C2)] For all $u_i \in V_1, i=1,2,3$
\begin{equation*}
\langle B(u_1,u_2),u_3 \rangle = -\langle B(u_1,u_3),u_2\rangle,~~\langle B(u_1,u_2),u_2 \rangle=0.
\end{equation*}

\item[(C3)] There exists a Banach space $\mathcal{H}$ such that \\
\\
(i) $V_1 \subset \mathcal{H} \subset H_1;$\\
\\
(ii) there exists a constant $a_0>0$ such that
\begin{equation*}
\|u\|_\mathcal{H}^2 \le a_0 |u|\|u\|~~\text{for all} ~ u\in V_1;
\end{equation*}
(iii) for every $\eta>0$ there exists a constant $C_\eta>0$ such that
\begin{equation*}
|\langle B(u_1,u_2),u_3 \rangle| \le \eta\|u_3\|^2+C_\eta\|u_1\|_{\mathcal{H}}^2\|u_2\|_{\mathcal{H}}^2
~~\text{for all}~u_i\in V_1, i=1,2,3.
\end{equation*}
\end{enumerate}
For simplicity of notations, we denote $B(u):=B(u,u)$. Moreover, we consider $H_2:=H_1$, $V_2:=V_1$ and the Gelfand triple for the fast component
$$V_2\subset H_2\subset V_2^*.$$

The following is the two-time-scale stochastic 2D hydrodynamical type systems,
\begin{equation}\label{hys}
\left\{ \begin{aligned}
&dX^{\epsilon,\alpha}_t+[A X^{\epsilon,\alpha}_t+B(X^{\epsilon,\alpha}_t,X^{\epsilon,\alpha}_t)]dt=F_1(X^{\epsilon,\alpha}_t,Y^{\epsilon,\alpha}_t)dt+\sqrt{\epsilon}G_1(X^{\epsilon,\alpha}_t)dW_t,\\
&dY^{\epsilon,\alpha}_t=\frac{1}{\alpha}[A Y^{\epsilon,\alpha}_t+F_2(X^{\epsilon,\alpha}_t,Y^{\epsilon,\alpha}_t)]dt+\frac{1}{\sqrt{\alpha}}G_2dW_t,\\
&X^{\epsilon,\alpha}_0=x\in H_1,~Y^{\epsilon,\alpha}_0=y\in H_2.
\end{aligned} \right.
\end{equation}

\begin{theorem}(stochastic 2D hydrodynamical type systems)\label{2Dhys}
Assume that $B$ satisfies (C1)-(C3) and $F_1,F_2,G_1$ satisfy (\ref{55})-(\ref{60}) above with $A$ replacing operator $L$, if
\begin{equation*}
\lim_{\epsilon\to0}\alpha(\epsilon)=0~\text{and}~\lim_{\epsilon\to0}\frac{\alpha}{\epsilon}=0,
\end{equation*}then $\{X^{\epsilon,\alpha}:\epsilon>0\}$ in (\ref{hys}) satisfies the LDP on $C([0,T]; H_1)\cap L^{2}([0,T]; V_1)$ with the
good rate function $I$ given by $(\ref{rf})$.
\end{theorem}

\begin{proof}
It is suffices to check the conditions $({\mathbf{A}}{\mathbf{1}})$-$({\mathbf{A}}{\mathbf{3}})$ hold for $\widetilde{A}(u):=-Au-B(u,u)$.

$({\mathbf{A}}{\mathbf{1}})$: The hemicontinuity follows from the linearity and bilinearity of maps $A$ and $B$, respectively.

$({\mathbf{A}}{\mathbf{2}})$: It is easy to see that for all $u,v\in V_1$,
\begin{equation}\label{61}
\langle - Au-(-Av),u-v\rangle\leq -\|u-v\|^2.
\end{equation}
According to \cite[Remark 2.1]{CM}, we know that for any constant $\eta>0$ the existence of $C_{\eta}>0$ such that for all $u,v\in V_1$
\begin{equation}\label{62}
|\langle B(u)-B(v),u-v\rangle|\leq \eta\|u-v\|^2+C_{\eta}|u-v|^2\|v\|_{\mathcal{H}}^4.
\end{equation}
Thus the condition $({\mathbf{A}}{\mathbf{2}})$ follows by (\ref{61}) and (\ref{62}) with $\gamma_1=2$.

Furthermore, (2.7) in \cite{CM} implies that $({\mathbf{A}}{\mathbf{3}})$ holds with $\beta_1=2$.

Using (\ref{55})-(\ref{60}), the rest of assumptions given in Theorem \ref{t1} are satisfied.
Then the conclusion here is a consequent result of Theorem \ref{t1}.   \hspace{\fill}$\Box$
\end{proof}

\begin{Rem}
(i) The well-posedness and Freidlin-Wentzell LDP of stochastic 2D hydrodynamical type systems have been investigated by Chueshov and Millet in \cite{CM}. In this work, we generalize the main results of \cite{CM} to the multi-scale case.

(ii) As in \cite{CM}, the main result obtained in this subsection is applicable to many concrete hydrodynamical type systems, for instance,  the stochastic 2D Navier-Stokes equation, stochastic 2D magneto-hydrodynamic equations, stochastic 2D Boussinesq equations, stochastic 2D magnetic B\'{e}nard problem, stochastic 3D Leray-$\alpha$ model and also shell models of turbulence. We also refer the reader to \cite{E,HLL1} and references within for the further studies of these models.
\end{Rem}

\subsection{Stochastic power
law fluid equation}
The stochastic power law fluid equation characterizes the velocity field of a viscous and incompressible non-Newtonian fluids,  one can
see \cite{FR08,MNRR} and references therein for more background of this model.

Let $u:\Lambda\rightarrow \mathbb{R}^d$ denote a
vector field. Set
$$e(u):\Lambda\rightarrow \mathbb{R}^d\otimes
\mathbb{R}^d;
~~e_{i,j}(u)=\frac{\partial_{i}u_{j}+\partial_{j}u_{i}}{2},i,j=1,\cdots,d.$$
$$\tau(u):\Lambda\rightarrow \mathbb{R}^d\otimes
\mathbb{R}^d; ~~~\tau(u)=2\nu(1+|e(u)|)^{p-2}e(u),$$
here $\nu>0$ represents the viscosity coefficient of the fluid, and $p>1$ is a constant.

We consider the following hydrodynamical equation with a power law property
$$\partial_{t}u=div(\tau(u))-(u\cdot\nabla)u-\nabla
p+f,~\text{for}~div(u)=0,$$
where $u=u(t,x)=(u_i(t,x))^d_{i=1}$ and $p$ stand for the
velocity field and pressure of the
fluid, respectively, $f$ denotes the
external force acting on the system,
$$div(\tau(u))=\left(\sum^d_{j=1}\partial_j\tau_{i,j}(u)\right)^d_{i=1}.$$
Moreover, it should mentioned that if one take $p=2$, the power law fluid equation reduces to the classical
Navier-Stokes equation.

For the slow component of two-time-scale situation, we consider the following Gelfand triple
$$V_1\subseteq H_1\subseteq V_1^*,$$
where we set
$$
\setlength{\abovedisplayskip}{8pt}
\setlength{\belowdisplayskip}{8pt} V_1=\Big\{u\in W^{1,p}_0(\Lambda;
\mathbb{R}^d): div(u)=0\Big\}; \ H_1=\Big\{u\in L^2(\Lambda;\mathbb{R}^d):
div(u)=0\Big\}.$$
Let $P_{H_1}$ denote a projection map onto $H_1$ on
$L^2(\Lambda;\mathbb{R}^d)$. Thus one can extend the
operators
$$\mathcal{A}: W^{2,p}(\Lambda;\mathbb{R}^d)\cap V_1\rightarrow
H_1,~\mathcal{A}(u)=P_{H_1}[div(\tau(u))];$$
$$F:\left( W^{2,p}(\Lambda;\mathbb{R}^d)\cap V_1  \right)  \times
\left( W^{2,p}(\Lambda;\mathbb{R}^d) \cap V_1 \right) \rightarrow H_1;$$
$$F(u,v)=-P_{H_1}[(u\cdot\nabla)v],~F(u):=F(u,u)$$
to following maps (see \cite{LR13} for details)
$$\mathcal{A}:V_1\rightarrow V_1^*;  ~~F:V_1\times
V_1\rightarrow V_1^*.$$
Moreover, it is easy to show that
$$\langle \mathcal{A}(u),v\rangle_{V_1}=-\int_\Lambda
\sum_{i,j=1}^{d}\tau_{i,j}(u)e_{i,j}(v) dx,~u,v\in {V_1};$$
$${ }_{V_1^*}\langle F(u,v),w\rangle_{V_1}=-{ }_{V_1^*}\langle F(u,w),v\rangle_{V_1},  ~~{ }_{V_1^*}\langle
F(u,v),v\rangle_{V_1}=0,~u,v,w\in V_1.$$
For the fast component of two-time-scale situation, we consider
$$V_2:=W_0^{1,2}(\Lambda;\mathbb{R}^d)\subset H_2:=L^2(\Lambda;\mathbb{R}^d)\subset(W_0^{1,2}(\Lambda;\mathbb{R}^d))^*=V_2^*.$$
Consequently, the multi-scale stochastic power law fluid equation can be written as follows in variational form,
\begin{equation}\label{PLF}
\left\{ \begin{aligned}
&dX^{\epsilon,\alpha}_t=(\nu \mathcal{A}X^{\epsilon,\alpha}_t+F(X^{\epsilon,\alpha}_t)+f)dt+F_1(X^{\epsilon,\alpha}_t,Y^{\epsilon,\alpha}_t)dt+\sqrt{\epsilon}G_1(X^{\epsilon,\alpha}_t)dW_t,\\
&dY^{\epsilon,\alpha}_t=\frac{1}{\alpha}[\Delta Y^{\epsilon,\alpha}_t+F_2(X^{\epsilon,\alpha}_t,Y^{\epsilon,\alpha}_t)]dt+\frac{1}{\sqrt{\alpha}}G_2dW_t,\\
&X^{\epsilon,\alpha}_0=x\in H_1,~Y^{\epsilon,\alpha}_0=y\in H_2.
\end{aligned} \right.
\end{equation}
for $f\in H_1$, $W_t$ is a cylindrical Wiener process on $U$.

\begin{theorem}\label{SPLF} (stochastic power law fluid equation) Let $p\geq\frac{d+2}{2}$ and $F_1,F_2,G_1$ satisfy (\ref{55})-(\ref{60}) above with $\Delta$ replacing operator $L$, if
\begin{equation*}
\lim_{\epsilon\to0}\alpha(\epsilon)=0~\text{and}~\lim_{\epsilon\to0}\frac{\alpha}{\epsilon}=0,
\end{equation*}then $\{X^{\epsilon,\alpha}:\epsilon>0\}$ in (\ref{PLF}) satisfies the LDP on $C([0,T]; H_1)\cap L^{2}([0,T]; V_1)$ with the
good rate function $I$ given by $(\ref{rf})$..
\end{theorem}
\begin{proof}
First, we want to check the conditions $({\mathbf{A}}{\mathbf{1}})$-$({\mathbf{A}}{\mathbf{3}})$. Without loss of generality, we assume that viscosity coefficient $\nu=1$.
According to \cite[Lemma 1.19]{MNRR}, one shows that
$$\int_\Lambda|e(u)|^p dx\geq
C_{p}\|u\|_{W^{1,p}},u\in W^{1,p}_0(\Lambda;\mathbb{R}^d);$$
$$\sum^d_{i,j=1}\tau_{i,j}(u)e_{i,j}(u)\geq C(|e(u)|^p-1);$$
$$\sum_{i,j=1}^
{d}(\tau_{i,j}(u)-\tau_{i,j}(v))(e_{i,j}(u)-e_{i,j}(v))\geq
C(|e(u)-e(v)|^2+|e(u)-e(v)|^p);$$
$$|\tau_{i,j}(u)|\leq C(1+|e(u)|)^{p-1},i,j=1...,d.$$
Therefore, by means of the estimates above, for all $u,v\in V_1$,
\begin{eqnarray*}
  { }_{V_1^*}\langle F(u)-F(v),u-v\rangle_{V_1} =&&\!\!\!\!\!\!\!\! -{ }_{V_1^*}\langle F(u-v),v\rangle_{V_1} \\
 =&&\!\!\!\!\!\!\!\!{ }_{V_1^*}\langle F(u-v,v),u-v\rangle_{V_1}\\
   \leq&&\!\!\!\!\!\!\!\! C\|v\|_{V_1}\|u-v\|_{L^{\frac{2p}{p-1}}}^{2}\\
  \leq&&\!\!\!\!\!\!\!\!
  C\|v\|_{V_1}\|u-v\|_{W^{1,2}}^{\frac{d}{p}}\|u-v\|_{H_1}^{\frac{2p-d}{p}}\\
   \leq&&\!\!\!\!\!\!\!\! \varepsilon\|u-v\|_{W^{1,2}}^{2}+C_\varepsilon\|v\|_{V_1}^{\frac{2p}{2p-d}}\|u-v\|_{H_1}^2.
\end{eqnarray*}
It follows that
\begin{eqnarray*}
&&{ }_{V_1^*}\langle
\mathcal{A}(u)+F(u)-\mathcal{A}(v)-F(v),u-v\rangle_{V_1} \\
=\!\!\!\!\!\!\!\!&&-\int_\Lambda\sum_{i,j=1}^{d}(\tau_{i,j}(u)-\tau_{i,j}(v))(e_{i,j}(u)-e_{i,j}(v))\
d x \\
\leq\!\!\!\!\!\!\!\!&&-C\|e(u)-e(v)\|_{H_1}^2\\
\leq\!\!\!\!\!\!\!\!&&-C\|u-v\|_{W^{1,2}}^{2}.
\end{eqnarray*}
Consequently, one can get that
$${ }_{V_1^*}\langle \mathcal{A}(u)+F(u)-\mathcal{A}(v)-F(v)\rangle_{V_1}
\leq -(C-\varepsilon)\|u-v\|_{W^{1,2}}^{2}
+C_{\varepsilon}\|v\|_{V_1}^{\frac{2p}{2p-d}}\|u-v\|_{H_1}^2,$$
then the condition $({\mathbf{A}}{\mathbf{2}})$ is satisfied with $\rho(v) = C_{\varepsilon} \|v \|_{V_1}^{\frac{4q}{4q-d}}$ and $\gamma_1 = p$.

For the condition $({\mathbf{A}}{\mathbf{3}})$, we infer that
$$|{ }_{V_1^*}\langle F(v),u\rangle_{V_1}|=|{ }_{V_1^*}\langle
F(v,u),v\rangle_{V_1}|\leq\|u\|_{V_1}\|v\|_{L^{\frac{2p}{p-1}}}^{2},~u,v\in V_1,$$
Then one can obtain
$$\|F(v)\|_{V_1^*}\leq\|v\|_{L^{\frac{2p}{p-1}}}^2,~v\in V_1.$$
Let $q=\frac{dp}{d-p}, \gamma=\frac{d}{(d+2)p-2d}$, using the Gagliardo-Nirenberg interpolation inequality \eref{GN_inequality} leads to
$$\|v\|_{L^{\frac{2p}{p-1}}}\leq\|v\|_{L^q}^{\gamma}\|v\|_{L^2}^{1-\gamma}\leq
C\|v\|_{V_1}^\gamma\|v\|_{H_1}^{1-\gamma}.$$ Since $p\geq\frac{d+2}{2}$,
obviously, the
condition $({\mathbf{A}}{\mathbf{3}})$ holds.

By means of (\ref{55})-(\ref{60}), the remainder of assumptions  in Theorem \ref{t1} hold.
Consequently, Theorem \ref{SPLF} is a consequent result of Theorem \ref{t1}.   \hspace{\fill}$\Box$
\end{proof}

\begin{Rem}
Besides the above example, our results are also applicable to e.g. the multi-scale stochastic Ladyzhenskaya model, which is a higher order variant of the power law fluid and pioneered by Ladyzhenskaya \cite{L70} (cf. e.g. \cite{ZD10}). We omit the detailed proof here to keep down the length of this paper.
\end{Rem}

In the sequel, we aim to prove Theorem \ref{t1} and \ref{t3}.
\section{Frozen equation and stochastic control problem}
\setcounter{equation}{0}
 \setcounter{definition}{0}
In this section, we consider the frozen equation corresponding to the fast component of system (\ref{e1}) for any fixed slow component $x\in H_1$ and the skeleton equation (\ref{e2}) associated with the slow equation of system (\ref{e1}). We will present the existence and uniqueness of invariant probability measures and the exponential ergodicity with respect to the frozen equations in order to define the coefficient $\bar{F_1}$ in the skeleton equation (\ref{e2}). Then we consider the stochastic control problem  for the system (\ref{e1}) and give some crucial lemmas which will be used frequently throughout this paper.

\subsection{Frozen and skeleton equation}
For each fixed slow component $x\in H_1$, the frozen equation with respect to the fast component of system (\ref{e1}) is given by
\begin{equation}\label{e3}
\left\{ \begin{aligned}
&dY_t=F_2(x,Y_t)dt+G_2d\widetilde{W}_t,\\
&Y_0=y\in H_2,
\end{aligned} \right.
\end{equation}
where $\widetilde{W}_t$ is a cylindrical Wiener process on Hilbert space $U$ and is independent of $W_t$. It is obvious that following from \cite[Theorem 4.2.4]{LR1}, by Hypothesis \ref{h2}, there is a unique solution denoted by $Y_t^{x,y}$ to Eq.~(\ref{e3}), which is a homogeneous Markov process.

Let $\{P^x_t\}_{t\geq 0}$ denote the Markov transition semigroup of process $\{Y_t^{x,y}\}_{t\geq 0}$, i.e. for any bounded measurable map $f$ on $H_2$,
$$P^x_tf(y):=\mathbb{E}f(Y_t^{x,y}),~y\in H_2,~t>0.$$

According to \cite[Theorem 4.3.9]{LR1}, we have the following exponential ergodicity result.
\begin{lemma}
Under Hypothesis \ref{h2}, there is a constant $C>0$ such that for all Lipschitz function $f:H_2\to H_1$ we have
\begin{equation}\label{er}
\Big\|P^x_tf(y)-\int_{H_2}f(z)\mu^x(dz)\Big\|_{H_1}\leq C(1+\|x\|_{H_1}+\|y\|_{H_2})e^{-\frac{\kappa t}{2}}\|f\|_{Lip},
\end{equation}
here $\mu^x$ is the unique invariant probability measure of $\{P^x_t\}_{t\geq 0}$, constant $\kappa>0$ is defined in $({\mathbf{H}}{\mathbf{2}})$ and $\|f\|_{Lip}$ is the Lipschitz constant of $f$.
\end{lemma}

\begin{lemma}\label{l4}
There is a constant $C>0$ such that for any $x_1,x_2\in H_1$ and $y\in H_2$ we have
$$\sup_{t\geq 0}\mathbb{E}\|Y^{x_1,y}_t-Y^{x_2,y}_t\|_{H_2}^2\leq C\|x_1-x_2\|_{H_1}^2.$$
\end{lemma}
\begin{proof}
Taking $Z_t^{x_1,x_2}:=Y^{x_1,y}_t-Y^{x_2,y}_t$, which satisfies
\begin{equation}\label{e4}
\frac{dZ_t^{x_1,x_2}}{dt}=F_2(x_1,Y^{x_1,y}_t)-F_2(x_2,Y^{x_2,y}_t),~Z_0=0.
\end{equation}
Using the energy equality of $Z_t^{x_1,x_2}$ we have
\begin{eqnarray*}
\frac{1}{2}\|Z_t^{x_1,x_2}\|_{H_2}^2=\int_0^t{}_{V_2^*}\langle F_2(x_1,Y^{x_1,y}_s)-F_2(x_2,Y^{x_2,y}_s),Z_s^{x_1,x_2}\rangle_{V_2}ds.
\end{eqnarray*}
By taking expectation it leads to
\begin{eqnarray*}
\!\!\!\!\!\!\!\!&&\frac{d\mathbb{E}\|Z_t^{x_1,x_2}\|_{H_2}^2}{dt}
\nonumber\\
~=\!\!\!\!\!\!\!\!&&2\mathbb{E}\Big[{}_{V_2^*}\langle F_2(x_1,Y^{x_1,y}_t)-F_2(x_1,Y^{x_2,y}_t),Z_t^{x_1,x_2}\rangle_{V_2}\Big] \nonumber\\
  \!\!\!\!\!\!\!\!&&+2\mathbb{E}\Big[{}_{V_2^*}\langle F_2(x_1,Y^{x_2,y}_t)-F_2(x_2,Y^{x_2,y}_t),Z_t^{x_1,x_2}\rangle_{V_2}\Big]
\nonumber\\
~\leq\!\!\!\!\!\!\!\!&&-\kappa\mathbb{E}\|Z_t^{x_1,x_2}\|_{H_2}^2+\varepsilon_0\mathbb{E}\|Z_t^{x_1,x_2}\|_{H_2}^2+C_{\varepsilon_0}\|x_1-x_2\|_{H_1}^2
\nonumber\\
~=\!\!\!\!\!\!\!\!&&-(\kappa-\varepsilon_0)\mathbb{E}\|Z_t^{x_1,x_2}\|_{H_2}^2+C_{\varepsilon_0}\|x_1-x_2\|_{H_1}^2,
\end{eqnarray*}
where we used the condition $({\mathbf{H}}{\mathbf{2}})$ and Young's inequality in the second step.

Choosing $\varepsilon_0$ small enough and using the comparison theorem yields that for all $t>0$,
\begin{eqnarray*}
\mathbb{E}\|Z_t^{x_1,x_2}\|_{H_2}^2\leq C\|x_1-x_2\|_{H_1}^2\int_0^te^{-\eta(t-s)}ds\leq C_\eta\|x_1-x_2\|_{H_1}^2,
\end{eqnarray*}
where we denote $\eta:=\kappa-\varepsilon_0>0$, which implies the assertion. \hspace{\fill}$\Box$
\end{proof}

Now we recall $\bar{X}^{\phi}_t$ defined in the skeleton equation (\ref{e2}).  The existence and uniqueness of solutions to Eq.~(\ref{e2}) is established in the following lemma, moreover, some important energy estimates for the skeleton equation (\ref{e2}) are also derived for later use.
\begin{lemma}\label{l5}
Suppose that Hypothesis \ref{h2} holds. For each $x\in H_1$ and $\phi\in L^2([0,T];U)$, there exists a unique solution $(\bar{X}^{\phi}_t)_{t\in[0, T]}$ to Eq.~(\ref{e2}), moreover, we have
\begin{equation}\label{a1}
\sup_{\phi\in S_M}\Big\{\sup_{t\in[0,T]}\|\bar{X}^{\phi}_t\|_{H_1}^2+\theta_1\int_0^T\|\bar{X}^{\phi}_t\|_{V_1}^{\gamma_1}dt\Big\}\leq C(1+\|x\|_{H_1}^2),
\end{equation}
where $C>0$ is a constant.
\end{lemma}
\begin{proof}
We split the proof into the following two steps.

\textbf{Step 1}: In order to prove the well-posedness of Eq.~(\ref{e2}), we first consider $\phi\in L^{\infty}([0,T];U)$ and take
$$\widetilde{A}_t(u):= A(u)+\bar{F}_1(u)+G_1(u)\phi_t.$$
Since $F_1(x,y)$ is Lipschitz continuous with respect to $x$ and $y$, one can show that the map $\bar{F}_1$ is also Lipschitz by Lemma \ref{l4}. Indeed, there is a constant $C>0$ such that
\begin{eqnarray*}
\!\!\!\!\!\!\!\!&&\|\bar{F}_1(u)-\bar{F}_1(v)\|_{H_1}
\nonumber\\
~=\!\!\!\!\!\!\!\!&&\Big\|\int_{H_1}F_1(u,z)\mu^{u}(dz)-\int_{H_1}F_1(v,z)\mu^{v}(dz)\Big\|_{H_1}
\nonumber\\
~\leq\!\!\!\!\!\!\!\!&&\Big\|\int_{H_1}F_1(u,z)\mu^{u}(dz)-\mathbb{E}F_1(u,Y^{u,y}_t)\Big\|_{H_1}+\Big\|\int_{H_1}F_1(v,z)\mu^{u}(dz)-\mathbb{E}F_1(v,Y^{v,y}_t)\Big\|_{H_1}
\nonumber\\
\!\!\!\!\!\!\!\!&&+\Big\|\mathbb{E}F_1(u,Y^{u,y}_t)-\mathbb{E}F_1(v,Y^{v,y}_t)\Big\|_{H_1}
\nonumber\\
~\leq\!\!\!\!\!\!\!\!&&C(1+\|x\|_{H_1}+\|y\|_{H_2})e^{-\frac{\kappa t}{2}}+C\big(\|u-v\|_{H_1}+\mathbb{E}\|Y^{u,y}_t-Y^{v,y}_t\|_{H_2}\big)
\nonumber\\
~\leq\!\!\!\!\!\!\!\!&&C(1+\|x\|_{H_1}+\|y\|_{H_2})e^{-\frac{\kappa t}{2}}+C\|u-v\|_{H_1}.
\end{eqnarray*}
Taking $t\uparrow+\infty$ yields that $\bar{F}_1$ is Lipschitz continuous.

By Hypothesis \ref{h1}, it is easy to check that $\widetilde{A}_t $ satisfies the local monotonicity, coercivity conditions in \cite[Theorem 5.1.3]{LR1} since $\bar{F}_1$ is Lipschitz and $\phi\in L^{\infty}([0,T];U)$. Hence Eq.~(\ref{e2}) admits a unique solution $\bar{X}^{\phi}\in C([0,T];H_1)\cap L^{\gamma_1}([0,T];V_1)$ for any $\phi\in L^{\infty}([0,T];U)$.

\textbf{Step 2}: For any $\phi\in L^2([0,T];U)$, one can choose a sequence
$\phi^n\in L^\infty([0,T];U)$ such that $\phi^n$ strongly converge to $\phi$ in $L^2([0,T];U)$ as $n\to\infty$. Let $\bar{X}^{\phi^n}$ denote the unique solution to Eq.~(\ref{e2}) with $\phi^n\in L^{\infty}([0,T];U)$. According to $({\mathbf{A}}{\mathbf{2}})$, it follows that there exist some constant $C>0$, for any $n,m\in\mathbb{N}$ we have
\begin{eqnarray}\label{1}
\!\!\!\!\!\!\!\!&&\frac{d}{dt}\|\bar{X}^{\phi^n}_t-\bar{X}^{\phi^m}_t\|_{H_1}^2
\nonumber\\
~=\!\!\!\!\!\!\!\!&&2{}_{V_1^*}\langle A(\bar{X}^{\phi^n}_t)-A(\bar{X}^{\phi^m}_t),\bar{X}^{\phi^n}_t-\bar{X}^{\phi^m}_t\rangle_{V_1}
+2\langle \bar{F}_1(\bar{X}^{\phi^n}_t)-\bar{F}_1(\bar{X}^{\phi^m}_t),\bar{X}^{\phi^n}_t-\bar{X}^{\phi^m}_t\rangle_{H_1}
\nonumber\\
\!\!\!\!\!\!\!\!&&+2\langle G_1(\bar{X}^{\phi^n}_t)\phi^n_t-G_1(\bar{X}^{\phi^m}_t)\phi^m_t,\bar{X}^{\phi^n}_t-\bar{X}^{\phi^m}_t\rangle_{H_1}
\nonumber\\
~\leq\!\!\!\!\!\!\!\!&&-\theta_1\|\bar{X}^{\phi^n}_t-\bar{X}^{\phi^m}_t\|_{V_1}^{\gamma_1}+(C+\rho(\bar{X}^{\phi^m}_t))\|\bar{X}^{\phi^n}_t-\bar{X}^{\phi^m}_t\|_{H_1}^2
+\|\phi^n_t\|_U^2\|\bar{X}^{\phi^n}_t-\bar{X}^{\phi^m}_t\|_{H_1}^2
\nonumber\\
\!\!\!\!\!\!\!\!&&+2\langle G_1(\bar{X}^{\phi^m}_t)(\phi^n_t-\phi^m_t),\bar{X}^{\phi^n}_t-\bar{X}^{\phi^m}_t\rangle_{H_1}
\nonumber\\
~\leq\!\!\!\!\!\!\!\!&&-\theta_1\|\bar{X}^{\phi^n}_t-\bar{X}^{\phi^m}_t\|_{V_1}^{\gamma_1}+\|\phi^n_t-\phi^m_t\|_U^2
\nonumber\\
\!\!\!\!\!\!\!\!&&+(C+\rho(\bar{X}^{\phi^m}_t)+\|\phi^n_t\|_U^2+\|G_1(\bar{X}^{\phi^m}_t)\|_{L_2(U,H_1)}^2)\|\bar{X}^{\phi^n}_t-\bar{X}^{\phi^m}_t\|_{H_1}^2.
\end{eqnarray}
Gronwall' lemma yields that
\begin{eqnarray}\label{4}
\!\!\!\!\!\!\!\!&&\|\bar{X}^{\phi^n}_t-\bar{X}^{\phi^m}_t\|_{H_1}^2+\theta_1\int_0^t\|\bar{X}^{\phi^n}_s-\bar{X}^{\phi^m}_s\|_{V_1}^{\gamma_1}ds
\nonumber\\
~\leq\!\!\!\!\!\!\!\!&&\exp\Big\{\int_0^T\Big(C+\rho(\bar{X}^{\phi^m}_t)+\|\phi^n_t\|_U^2+\|G_1(\bar{X}^{\phi^m}_t)\|_{L_2(U,H_1)}^2\Big)dt\Big\}\int_0^T\|\phi^n_t-\phi^m_t\|_U^2dt.~~
\end{eqnarray}
Following the similar calculations as in (\ref{1}) we have
\begin{eqnarray*}
\frac{d}{dt}\|\bar{X}^{\phi^m}_t\|_{H_1}^2\leq\!\!\!\!\!\!\!\!&&2{}_{V_1^*}\langle A(\bar{X}^{\phi^m}_t),\bar{X}^{\phi^m}_t\rangle_{V_1}+2\langle\bar{F}_1(\bar{X}^{\phi^m}_t)+G_1(\bar{X}^{\phi^m}_t)\phi^m_t,\bar{X}^{\phi^m}_t\rangle_{H_1}
\nonumber\\
\leq\!\!\!\!\!\!\!\!&&-\theta_1\|\bar{X}^{\phi^m}_t\|_{V_1}^{\gamma_1}+C(1+\|\phi^m_t\|_U^2)\|\bar{X}^{\phi^m}_t\|_{H_1}^2+C.
\end{eqnarray*}
Let $\phi^m\in S_M$, then  Gronwall' lemma implies that
\begin{eqnarray}\label{2}
\!\!\!\!\!\!\!\!&&\sup_{t\in[0,T]}\|\bar{X}^{\phi^m}_t\|_{H_1}^2+\theta_1\int_0^T\|\bar{X}^{\phi^m}_t\|_{V_1}^{\gamma_1}dt
\nonumber\\
~\leq\!\!\!\!\!\!\!\!&&C\exp\Big\{\int_0^T\Big(1+\|\phi^m_t\|_U^2\Big)dt\Big\}(\|x\|_{H_1}^2+T)\leq C_{T,M}(1+\|x\|_{H_1}^2),
\end{eqnarray}
where constant $C_{T,M}$ only depends on $T$ and $M$.

According to $({\mathbf{A}}{\mathbf{2}})$, it follows that
\begin{eqnarray}\label{3}
\int_0^T\|G_1(\bar{X}^{\phi^m}_t)\|_{L_2(U,H_1)}^2dt\leq C\int_0^T(1+\|\bar{X}^{\phi^m}_t\|_{H_1}^2)dt\leq C_{T,M}(1+\|x\|_{H_1}^2).
\end{eqnarray}
Substituting (\ref{2}) and (\ref{3}) into (\ref{4}) and letting $n\to\infty$, we obtain that $\{\bar{X}^{\phi^n}\}_{n\geq1}$ is a Cauchy net in $C([0,T];H_1)\cap L^{\gamma_1}([0,T];V_1)$, then the limit is denoted by $\bar{X}^{\phi}$. Following the standard monotonicity argument (see e.g.~\cite[Theorem 30.A]{z}) implies that $\bar{X}^{\phi}$ is the solution to Eq.~(\ref{e2}) associated with $\phi$. The uniqueness is a direct consequence of $({\mathbf{A}}{\mathbf{2}})$ and Gronwall's lemma, and the estimate (\ref{a1}) can be concluded by (\ref{2}). Hence, we complete the proof of Lemma \ref{l5}. \hspace{\fill}$\Box$
\end{proof}

\subsection{Stochastic control problem}
We now introduce the following stochastic control problem associated with Eq.~(\ref{e1}),
\begin{eqnarray}\label{e5}
\left\{ \begin{aligned}
&dX^{\epsilon,\alpha,\phi^\epsilon}_t=\big[A(X^{\epsilon,\alpha,\phi^\epsilon}_t)+F_1(X^{\epsilon,\alpha,\phi^\epsilon}_t,Y^{\epsilon,\alpha,\phi^\epsilon}_t)+G_1(X^{\epsilon,\alpha,\phi^\epsilon}_t)\phi^\epsilon_t\big]dt+\sqrt{\epsilon}G_1(X^{\epsilon,\alpha,\phi^\epsilon}_t)dW_t,\\
&dY^{\epsilon,\alpha,\phi^\epsilon}_t=\frac{1}{\alpha}F_2(X^{\epsilon,\alpha,\phi^\epsilon}_t,Y^{\epsilon,\alpha,\phi^\epsilon}_t)dt+\frac{1}{\sqrt{\alpha\epsilon}}G_2\phi^\epsilon_tdt+\frac{1}{\sqrt{\alpha}}G_2dW_t,\\
&X^{\epsilon,\alpha,\phi^\epsilon}_0=x,~Y^{\epsilon,\alpha,\phi^\epsilon}_0=y,
\end{aligned}\right.
\end{eqnarray}
where $\{\phi^\epsilon\}_{\epsilon>0}\subset \mathcal{A}_M$ for some $M<\infty$.

Let us define
$$\Phi(t):=\exp\Big\{-\frac{1}{\sqrt{\epsilon}}\int_0^t\langle\phi^\epsilon_s,dW_s\rangle_U-\frac{1}{2\epsilon}\int_0^t\|\phi^\epsilon_s\|_U^2ds\Big\}.$$
It is easy to verify  the Novikov's condition  for $\Phi(t)$ since $\phi^\epsilon\in\mathcal{A}_M$, therefore $\Phi(t)$ is a martingale. Hence, one can consider the following weighted probability measure on $(\Omega,\mathscr{F},\mathbb{P})$
$$d\widehat{\mathbb{P}}:=\Phi(T)d\mathbb{P}.$$
Thanks to the Girsanov's theorem, the process
\begin{equation}\label{wiener}
W_t^\epsilon:=W_t+\frac{1}{\sqrt{\epsilon}}\int_0^t\phi^\epsilon_sds,~t\in[0,T]
\end{equation}
is a cylindrical Wiener process with respect to the stochastic basis $(\Omega,\mathscr{F},\mathscr{F}_t,\widehat{\mathbb{P}})$. According to the uniqueness of solutions to Eq.~(\ref{e1}) and the Yamada-Watanabe theorem, it is easy to see that
\begin{equation}\label{5}
X^{\epsilon,\alpha,\phi^\epsilon}:=\mathcal{G}^\epsilon\Big(W_{\cdot}+\frac{1}{\sqrt{\epsilon}}\int_0^{\cdot}\phi^\epsilon_sds\Big)
\end{equation}
is the first part of  (pair) solution $(X^{\epsilon,\alpha,\phi^\epsilon},Y^{\epsilon,\alpha,\phi^\epsilon})$ to Eq.~(\ref{e5}) with $W_{\cdot}^\epsilon$  instead of $W_{\cdot}$ on $(\Omega,\mathscr{F},\mathscr{F}_t,\widehat{\mathbb{P}})$. Since the probability measure $\mathbb{P}$ and $\widehat{\mathbb{P}}$ are mutually absolutely continuous, this implies that (\ref{5}) is also the first part of solution $(X^{\epsilon,\alpha,\phi^\epsilon},Y^{\epsilon,\alpha,\phi^\epsilon})$ on $(\Omega,\mathscr{F},\mathscr{F}_t,\mathbb{P})$.

Some energy estimates of $(X^{\epsilon,\alpha,\phi^\epsilon},Y^{\epsilon,\alpha,\phi^\epsilon})$ to the controlled equation (\ref{e5}) are derived as follows.
\begin{lemma}\label{l6}
For each $\phi^\epsilon\in\mathcal{A}_M$, $M\in(0,+\infty)$, there are some constants $C>0$ such that for any $\epsilon,\alpha\in(0,1)$,
\begin{equation}\label{14}
\mathbb{E}\Big[\sup_{t\in[0,T]}\|X^{\epsilon,\alpha,\phi^\epsilon}_t\|_{H_1}^2+2\theta_1\int_0^T\|X^{\epsilon,\alpha,\phi^\epsilon}_t\|_{V_1}^{\gamma_1}dt\Big]\leq C(1+\|x\|_{H_1}^2+\|y\|_{H_2}^2)
\end{equation}
and
\begin{equation}\label{15}
\mathbb{E}\int_0^T\|Y^{\epsilon,\alpha,\phi^\epsilon}_t\|_{H_2}^2dt\leq C(1+\|x\|_{H_1}^2+\|y\|_{H_2}^2).
\end{equation}
\end{lemma}

\begin{proof}
Applying It\^{o}'s formula to $\|Y^{\epsilon,\alpha,\phi^\epsilon}_t\|_{H_2}^2$ gives that
\begin{eqnarray*}
\!\!\!\!\!\!\!\!&&\|Y^{\epsilon,\alpha,\phi^\epsilon}_t\|_{H_2}^2
\nonumber\\
~=\!\!\!\!\!\!\!\!&&\|y\|_{H_2}^2+\frac{2}{\alpha}\int_0^t{}_{V_2^*}\langle F_2(X^{\epsilon,\alpha,\phi^\epsilon}_s,Y^{\epsilon,\alpha,\phi^\epsilon}_s),Y^{\epsilon,\alpha,\phi^\epsilon}_s\rangle_{V_2}ds+\frac{1}{\alpha}\int_0^t \|G_2\|^2_{L_2(U,H_2)}ds
\nonumber\\
\!\!\!\!\!\!\!\!&&+\frac{2}{\sqrt{\alpha\epsilon}}\int_0^t\langle G_2\phi^\epsilon_s,Y^{\epsilon,\alpha,\phi^\epsilon}_s\rangle_{H_2}ds
+\frac{2}{\sqrt{\alpha}}\int_0^t\langle G_2dW_s,Y^{\epsilon,\alpha,\phi^\epsilon}_s\rangle_{H_2}.
\end{eqnarray*}

Then taking expectation  we have
\begin{eqnarray}\label{6}
\!\!\!\!\!\!\!\!&&\frac{d}{dt}\mathbb{E}\|Y^{\epsilon,\alpha,\phi^\epsilon}_t\|_{H_2}^2
\nonumber\\
~=\!\!\!\!\!\!\!\!&&\frac{2}{\alpha}\mathbb{E}\Big[{}_{V_2^*}\langle F_2(X^{\epsilon,\alpha,\phi^\epsilon}_t,Y^{\epsilon,\alpha,\phi^\epsilon}_t),Y^{\epsilon,\alpha,\phi^\epsilon}_t\rangle_{V_2}\Big]+\frac{1}{\alpha} \|G_2\|^2_{L_2(U,H_2)}
+\frac{2}{\sqrt{\alpha\epsilon}}\mathbb{E}\Big[\langle G_2\phi^\epsilon_t,Y^{\epsilon,\alpha,\phi^\epsilon}_t\rangle_{H_2}\Big].~~~~
\end{eqnarray}
Following the similar arguments as in the proof of \cite[Lemma 4.3.8]{LR1} via Hypothesis \ref{h2}, there is a constant $\eta\in(0,\theta_2)$,
$$2{}_{V_2^*}\langle F_2(u,v),v\rangle_{V_2}\leq-\eta\|v\|_{H_2}^2+C(1+\|u\|_{H_1}^2).$$
Hence the first term of the right hand side of (\ref{6}) is controlled by
\begin{eqnarray}\label{7}
\frac{1}{\alpha}\mathbb{E}\Big[2{}_{V_2^*}\langle F_2(X^{\epsilon,\alpha,\phi^\epsilon}_t,Y^{\epsilon,\alpha,\phi^\epsilon}_t),Y^{\epsilon,\alpha,\phi^\epsilon}_t\rangle_{V_2}\Big]
\leq-\frac{\eta}{\alpha}\mathbb{E}\|Y^{\epsilon,\alpha,\phi^\epsilon}_t\|_{H_2}^2+\frac{C}{\alpha}(1+\mathbb{E}\|X^{\epsilon,\alpha,\phi^\epsilon}_t\|_{H_1}^2).~~~~
\end{eqnarray}
The last term of the right hand side of (\ref{6}) can be estimated by
\begin{eqnarray}\label{8}
\frac{2}{\sqrt{\alpha\epsilon}}\mathbb{E}\Big[\langle G_2\phi^\epsilon_t,Y^{\epsilon,\alpha,\phi^\epsilon}_t\rangle_{H_2}\Big]
\leq\!\!\!\!\!\!\!\!&&\frac{2}{\sqrt{\alpha\epsilon}}\mathbb{E}\Big[\|G_2\|_{L(U,H_2)}\|\phi^\epsilon_t\|_U\|Y^{\epsilon,\alpha,\phi^\epsilon}_t\|_{H_2}\Big]
\nonumber\\
\leq\!\!\!\!\!\!\!\!&&\frac{C}{\epsilon}\mathbb{E}\Big[\|\phi^\epsilon_t\|_U^2\Big]+\frac{\varepsilon_0}{\alpha}\mathbb{E}\|Y^{\epsilon,\alpha,\phi^\epsilon}_t\|_{H_2}^2,
\end{eqnarray}
where $\varepsilon_0\in(0,\eta)$, and we used Young's inequality in the last one.

Substituting $(\ref{7})$ and $(\ref{8})$ into $(\ref{6})$ yields that
\begin{eqnarray*}
\frac{d}{dt}\mathbb{E}\|Y^{\epsilon,\alpha,\phi^\epsilon}_t\|_{H_2}^2\leq-\frac{C_{\eta}}{\alpha}\mathbb{E}\|Y^{\epsilon,\alpha,\phi^\epsilon}_t\|_{H_2}^2
+\frac{C}{\alpha}(1+\mathbb{E}\|X^{\epsilon,\alpha,\phi^\epsilon}_t\|_{H_1}^2)+\frac{C}{\epsilon}\mathbb{E}\Big[\|\phi^\epsilon_t\|_U^2\Big].
\end{eqnarray*}
The comparison theorem implies that
\begin{eqnarray}\label{9}
\mathbb{E}\|Y^{\epsilon,\alpha,\phi^\epsilon}_t\|_{H_2}^2\leq\!\!\!\!\!\!\!\!&&e^{-\frac{C_\eta}{\alpha}t}\|y\|_{H_2}^2+\frac{C}{\alpha}\int_0^te^{-\frac{C_\eta}{\alpha}(t-s)}(1+\mathbb{E}\|X^{\epsilon,\alpha,\phi^\epsilon}_s\|_{H_1}^2)ds
\nonumber\\
\!\!\!\!\!\!\!\!&&+\frac{C}{\epsilon}\int_0^te^{-\frac{C_\eta}{\alpha}(t-s)}\mathbb{E}\Big[\|\phi^\epsilon_s\|_U^2\Big]ds.
\end{eqnarray}
Integrating (\ref{9}) with respect to $t$ from $0$ to $T$ and using Fubini's theorem we get
\begin{eqnarray}\label{10}
\mathbb{E}\Big[\int_0^T\|Y^{\epsilon,\alpha,\phi^\epsilon}_t\|_{H_2}^2dt\Big]\leq\!\!\!\!\!\!\!\!&&\|y\|_{H_2}^2\int_0^Te^{-\frac{C_\eta}{\alpha}t}dt+\frac{C}{\alpha}\int_0^T\int_0^te^{-\frac{C_\eta}{\alpha}(t-s)}(1+\mathbb{E}\|X^{\epsilon,\alpha,\phi^\epsilon}_s\|_{H_1}^2)dsdt
\nonumber\\
\!\!\!\!\!\!\!\!&&+\frac{C}{\epsilon}\mathbb{E}\Big[\int_0^T\int_0^te^{-\frac{C_\eta}{\alpha}(t-s)}\|\phi^\epsilon_s\|_U^2dsdt\Big]
\nonumber\\
\leq\!\!\!\!\!\!\!\!&&\frac{\alpha}{C_\eta}\|y\|_{H_2}^2+C_\eta\mathbb{E}\Big[\int_0^T(1+\|X^{\epsilon,\alpha,\phi^\epsilon}_t\|_{H_1}^2)dt\Big]
\nonumber\\
\!\!\!\!\!\!\!\!&&+C_\eta\Big(\frac{\alpha}{\epsilon}\Big)\mathbb{E}\Big[\int_0^T\|\phi^\epsilon_t\|_U^2dt\Big]
\nonumber\\
\leq\!\!\!\!\!\!\!\!&&C_{\alpha,\eta,T}(1+\|y\|_{H_2}^2)+C_\eta\mathbb{E}\int_0^T\|X^{\epsilon,\alpha,\phi^\epsilon}_t\|_{H_1}^2dt
+C_{\eta,M}\Big(\frac{\alpha}{\epsilon}\Big).
\end{eqnarray}
Now we aim to estimate $\|X^{\epsilon,\alpha,\phi^\epsilon}_t\|_{H_1}^2$. First, according to It\^{o}'s formula
\begin{eqnarray*}
\!\!\!\!\!\!\!\!&&\|X^{\epsilon,\alpha,\phi^\epsilon}_t\|_{H_1}^2
\nonumber\\
~=\!\!\!\!\!\!\!\!&&\|x\|_{H_1}^2+2\int_0^t{}_{V_1^*}\langle A(X^{\epsilon,\alpha,\phi^\epsilon}_s),X^{\epsilon,\alpha,\phi^\epsilon}_s\rangle_{V_1}ds+\epsilon\int_0^t\|G_1(X^{\epsilon,\alpha,\phi^\epsilon}_s)\|_{L_2(U,H_1)}^2ds
\nonumber\\
\!\!\!\!\!\!\!\!&&+2\int_0^t\Big[\langle F_1(X^{\epsilon,\alpha,\phi^\epsilon}_s,Y^{\epsilon,\alpha,\phi^\epsilon}_s)+G_1(X^{\epsilon,\alpha,\phi^\epsilon}_s)\phi^\epsilon_s,X^{\epsilon,\alpha,\phi^\epsilon}_s\rangle_{H_1}\Big]ds
\nonumber\\
\!\!\!\!\!\!\!\!&&+2\sqrt{\epsilon}\int_0^t\langle G_1(X^{\epsilon,\alpha,\phi^\epsilon}_s)dW_s,X^{\epsilon,\alpha,\phi^\epsilon}_s\rangle_{H_1}.
\end{eqnarray*}
Following from Remark \ref{r1} and $({\mathbf{A}}{\mathbf{2}})$ that
\begin{eqnarray}\label{11}
\!\!\!\!\!\!\!\!&&\mathbb{E}\Big[\sup_{s\in[0,t]}\|X^{\epsilon,\alpha,\phi^\epsilon}_s\|_{H_1}^2\Big]+\theta_1\mathbb{E}\int_0^t\|X^{\epsilon,\alpha,\phi^\epsilon}_s\|_{V_1}^{\gamma_1}ds
\nonumber\\
~\leq\!\!\!\!\!\!\!\!&&\|x\|_{H_1}^2+CT+C\mathbb{E}\int_0^t\|X^{\epsilon,\alpha,\phi^\epsilon}_s\|_{H_1}^2ds+C\mathbb{E}\int_0^t\|Y^{\epsilon,\alpha,\phi^\epsilon}_s\|_{H_2}^2ds
\nonumber\\
\!\!\!\!\!\!\!\!&&+2\mathbb{E}\int_0^t|\langle G_1(X^{\epsilon,\alpha,\phi^\epsilon}_s)\phi^\epsilon_s,X^{\epsilon,\alpha,\phi^\epsilon}_s\rangle_{H_1}|ds+2\sqrt{\epsilon}\mathbb{E}\Big[\sup_{s\in[0,t]}\Big|\int_0^s\langle G_1(X^{\epsilon,\alpha,\phi^\epsilon}_r)dW_r,X^{\epsilon,\alpha,\phi^\epsilon}_r\rangle_{H_1}\Big|\Big].
\nonumber\\
\end{eqnarray}
Using Cauchy-Schwarz's inequality, H\"{o}lder's inequality and Young's inequality we have
\begin{eqnarray}\label{12}
\!\!\!\!\!\!\!\!&&2\mathbb{E}\int_0^t|\langle G_1(X^{\epsilon,\alpha,\phi^\epsilon}_s)\phi^\epsilon_s,X^{\epsilon,\alpha,\phi^\epsilon}_s\rangle_{H_1}|ds
\nonumber\\
~\leq\!\!\!\!\!\!\!\!&&2\mathbb{E}\Big[\sup_{s\in[0,t]}\|X^{\epsilon,\alpha,\phi^\epsilon}_s\|_{H_1}\int_0^t\|G_1(X^{\epsilon,\alpha,\phi^\epsilon}_s)\|_{L_2(U,H_1)}\|\phi^\epsilon_s\|_Uds\Big]
\nonumber\\
~\leq\!\!\!\!\!\!\!\!&&\frac{1}{4}\mathbb{E}\Big[\sup_{s\in[0,t]}\|X^{\epsilon,\alpha,\phi^\epsilon}_s\|_{H_1}^2\Big]+4\mathbb{E}\Big(\int_0^t\|G_1(X^{\epsilon,\alpha,\phi^\epsilon}_s)\|_{L_2(U,H_1)}\|\phi^\epsilon_s\|_Uds\Big)^2
\nonumber\\
~\leq\!\!\!\!\!\!\!\!&&\frac{1}{4}\mathbb{E}\Big[\sup_{s\in[0,t]}\|X^{\epsilon,\alpha,\phi^\epsilon}_s\|_{H_1}^2\Big]+4\mathbb{E}\Big[\Big(\int_0^t\|G_1(X^{\epsilon,\alpha,\phi^\epsilon}_s)\|_{L_2(U,H_1)}^2ds\Big)\Big(\int_0^T\|\phi^\epsilon_s\|_U^2ds\Big)\Big]
\nonumber\\
~\leq\!\!\!\!\!\!\!\!&&\frac{1}{4}\mathbb{E}\Big[\sup_{s\in[0,t]}\|X^{\epsilon,\alpha,\phi^\epsilon}_s\|_{H_1}^2\Big]+C_{M,T}+C_M\mathbb{E}\int_0^t\|X^{\epsilon,\alpha,\phi^\epsilon}_s\|_{H_1}^2ds,
\end{eqnarray}
where the last step is due to $\phi^\epsilon\in\mathcal{A}_M$.

From Burkholder-Davis-Gundy's inequality, the last term of right hand side of (\ref{11}) is estimated by
\begin{eqnarray}\label{13}
\!\!\!\!\!\!\!\!&&2\sqrt{\epsilon}\mathbb{E}\Big[\sup_{s\in[0,t]}\Big|\int_0^s\langle G_1(X^{\epsilon,\alpha,\phi^\epsilon}_r)dW_r,X^{\epsilon,\alpha,\phi^\epsilon}_r\rangle_{H_1}\Big|\Big]
\nonumber\\
~\leq\!\!\!\!\!\!\!\!&&8\sqrt{\epsilon}\mathbb{E}\Big[\int_0^t\|G_1(X^{\epsilon,\alpha,\phi^\epsilon}_s)\|_{L_2(U,H_1)}^2\|X^{\epsilon,\alpha,\phi^\epsilon}_s\|_{H_1}^2ds\Big]^{\frac{1}{2}}
\nonumber\\
~\leq\!\!\!\!\!\!\!\!&&8\sqrt{\epsilon}\mathbb{E}\Big[\sup_{s\in[0,t]}\|X^{\epsilon,\alpha,\phi^\epsilon}_s\|_{H_1}^2\int_0^t\|G_1(X^{\epsilon,\alpha,\phi^\epsilon}_s)\|_{L_2(U,H_1)}^2\Big]^{\frac{1}{2}}
\nonumber\\
~\leq\!\!\!\!\!\!\!\!&&\frac{1}{4}\mathbb{E}\Big[\sup_{s\in[0,t]}\|X^{\epsilon,\alpha,\phi^\epsilon}_s\|_{H_1}^2\Big]+C_\epsilon\mathbb{E}\int_0^t\|X^{\epsilon,\alpha,\phi^\epsilon}_s\|_{H_1}^2ds+C_{\epsilon,T}.
\end{eqnarray}
Substituting (\ref{12}) and (\ref{13}) into (\ref{11}) and recalling (\ref{10}), we infer that
\begin{eqnarray*}
\!\!\!\!\!\!\!\!&&\mathbb{E}\Big[\sup_{s\in[0,t]}\|X^{\epsilon,\alpha,\phi^\epsilon}_s\|_{H_1}^2\Big]+2\theta_1\mathbb{E}\int_0^t\|X^{\epsilon,\alpha,\phi^\epsilon}_s\|_{V_1}^{\gamma_1}ds
\nonumber\\
~\leq\!\!\!\!\!\!\!\!&&2\|x\|_{H_1}^2+C_{M,\epsilon,T}+C_{M,\epsilon,T}\mathbb{E}\int_0^t\|X^{\epsilon,\alpha,\phi^\epsilon}_s\|_{H_1}^2ds+C\mathbb{E}\int_0^t\|Y^{\epsilon,\alpha,\phi^\epsilon}_s\|_{H_2}^2ds
\nonumber\\
~\leq\!\!\!\!\!\!\!\!&&2\|x\|_{H_1}^2+C_{\alpha,\eta,T}(1+\|y\|_{H_2}^2)+C_{M,\epsilon,T}+C_{M,\epsilon,T,\eta}\mathbb{E}\int_0^t\|X^{\epsilon,\alpha,\phi^\epsilon}_s\|_{H_1}^2ds
+C_{M,\eta}\Big(\frac{\alpha}{\epsilon}\Big).
\end{eqnarray*}
Owing to the condition (\ref{h5}), one can take $\frac{\alpha}{\epsilon}<\frac{1}{C_{M,\eta}}$, then the Gronwall's lemma yields that
\begin{eqnarray}\label{16}
\mathbb{E}\Big[\sup_{s\in[0,t]}\|X^{\epsilon,\alpha,\phi^\epsilon}_s\|_{H_1}^2\Big]+2\theta_1\mathbb{E}\int_0^t\|X^{\epsilon,\alpha,\phi^\epsilon}_s\|_{V_1}^{\gamma_1}ds
\leq C_{M,\epsilon,T,\eta}\Big[1+\|x\|_{H_1}^2+\|y\|_{H_2}^2\Big],
\end{eqnarray}
which gives the estimate (\ref{14}).

Moreover, it is easy to get the estimate (\ref{15}) by substituting (\ref{16}) into (\ref{10}). The proof of this lemma is completed. \hspace{\fill}$\Box$
\end{proof}

Now we would like to formulate a technical lemma, which investigates the time
increments of solution to the controlled equation (\ref{e5}).  We first define the following stopping time
$$\tau_N^\epsilon:=\inf\Big\{t\in[0,T]:\|X^{\epsilon,\alpha,\phi^\epsilon}_t\|_{H_1}>N\Big\},   N>0.  $$

\begin{lemma}\label{l7}
For $x\in H_1$, $y\in H_2$, $T>0$, and $\epsilon,\delta\in(0,1)$ are small enough constants, there exists some constant $C_{N}>0$ depending on $N$ such that
$$\mathbb{E}\Big[\int_0^{T\wedge\tau_N^\epsilon}\|X^{\epsilon,\alpha,\phi^\epsilon}_t-X^{\epsilon,\alpha,\phi^\epsilon}_{t(\delta)}\|_{H_1}^2dt\Big]\leq C_{N}(1+\|x\|_{H_1}^2+\|y\|_{H_2}^2)\delta^{\frac{1}{2}},$$
here $t(\delta):=[\frac{t}{\delta}]\delta$ and $[s]$ is the largest integer smaller than $s$.
\end{lemma}

\begin{proof}
It is easy to see that
\begin{eqnarray}\label{17}
\!\!\!\!\!\!\!\!&&\mathbb{E}\Big[\int_0^{T\wedge\tau_N^\epsilon}\|X^{\epsilon,\alpha,\phi^\epsilon}_t-X^{\epsilon,\alpha,\phi^\epsilon}_{t(\delta)}\|_{H_1}^2dt\Big]
\nonumber\\
~\leq\!\!\!\!\!\!\!\!&&\mathbb{E}\Big[\int_0^{\delta}\|X^{\epsilon,\alpha,\phi^\epsilon}_t-x\|_{H_1}^2\mathbf{1}_{\{t\leq\tau_N^\epsilon\}}dt\Big]+
\mathbb{E}\Big[\int_{\delta}^{T}\|X^{\epsilon,\alpha,\phi^\epsilon}_t-X^{\epsilon,\alpha,\phi^\epsilon}_{t(\delta)}\|_{H_1}^2\mathbf{1}_{\{t\leq\tau_N^\epsilon\}}dt\Big]
\nonumber\\
~\leq\!\!\!\!\!\!\!\!&&C_N(1+\|x\|_{H_1}^2)\delta+2\mathbb{E}\Big[\int_{\delta}^{T}\|X^{\epsilon,\alpha,\phi^\epsilon}_t-X^{\epsilon,\alpha,\phi^\epsilon}_{t-\delta}\|_{H_1}^2\mathbf{1}_{\{t\leq\tau_N^\epsilon\}}dt\Big]
\nonumber\\
\!\!\!\!\!\!\!\!&&+2\mathbb{E}\Big[\int_{\delta}^{T}\|X^{\epsilon,\alpha,\phi^\epsilon}_{t(\delta)}-X^{\epsilon,\alpha,\phi^\epsilon}_{t-\delta}\|_{H_1}^2\mathbf{1}_{\{t\leq\tau_N^\epsilon\}}dt\Big].
\end{eqnarray}
We now estimate the second term of right hand side of (\ref{17}). Applying It\^{o}'s formula yields that
\begin{eqnarray}\label{23}
\!\!\!\!\!\!\!\!&&\|X^{\epsilon,\alpha,\phi^\epsilon}_t-X^{\epsilon,\alpha,\phi^\epsilon}_{t-\delta}\|_{H_1}^2
\nonumber\\
~=\!\!\!\!\!\!\!\!&&2\int_{t-\delta}^t{}_{V_1^*}\langle A(X^{\epsilon,\alpha,\phi^\epsilon}_s),X^{\epsilon,\alpha,\phi^\epsilon}_s-X^{\epsilon,\alpha,\phi^\epsilon}_{t-\delta}\rangle_{V_1}ds
\nonumber\\
\!\!\!\!\!\!\!\!&&+2\int_{t-\delta}^t
\langle F_1(X^{\epsilon,\alpha,\phi^\epsilon}_s,Y^{\epsilon,\alpha,\phi^\epsilon}_s),X^{\epsilon,\alpha,\phi^\epsilon}_s-X^{\epsilon,\alpha,\phi^\epsilon}_{t-\delta}\rangle_{H_1}ds
\nonumber\\
\!\!\!\!\!\!\!\!&&+2\int_{t-\delta}^t\langle G_1(X^{\epsilon,\alpha,\phi^\epsilon}_s)\phi^\epsilon_s,X^{\epsilon,\alpha,\phi^\epsilon}_s-X^{\epsilon,\alpha,\phi^\epsilon}_{t-\delta}\rangle_{H_1}ds
+\epsilon\int_{t-\delta}^t\|G_1(X^{\epsilon,\alpha,\phi^\epsilon}_s)\|_{L_2(U,H_1)}^2ds
\nonumber\\
\!\!\!\!\!\!\!\!&&+2\sqrt{\epsilon}\int_{t-\delta}^t\langle G_1(X^{\epsilon,\alpha,\phi^\epsilon}_s)dW_s,X^{\epsilon,\alpha,\phi^\epsilon}_s-X^{\epsilon,\alpha,\phi^\epsilon}_{t-\delta}\rangle_{H_1}
\nonumber\\
=:\!\!\!\!\!\!\!\!&&\sum_{i=1}^{5}K_i(t).
\end{eqnarray}

Let us now consider the terms $\mathbb{E}\Big[\int_\delta^T|K_i(t)|\mathbf{1}_{\{t\leq\tau_N^\epsilon\}}dt\Big]$, $i=1,2,...,5$, respectively. According to $({\mathbf{A}}{\mathbf{3}})$ and H\"{o}lder's inequality, there is a constant $C_{N,T}>0$,
\begin{eqnarray}\label{18}
\!\!\!\!\!\!\!\!&&\mathbb{E}\Big[\int_\delta^T|K_1(t)|\mathbf{1}_{\{t\leq\tau_N^\epsilon\}}dt\Big]
\nonumber\\
~\leq\!\!\!\!\!\!\!\!&& C\mathbb{E}\Big[\int_\delta^T\int_{t-\delta}^t\|A(X^{\epsilon,\alpha,\phi^\epsilon}_s)\|_{V_1^*}\|X^{\epsilon,\alpha,\phi^\epsilon}_s-X^{\epsilon,\alpha,\phi^\epsilon}_{t-\delta}\|_{V_1}\mathbf{1}_{\{t\leq\tau_N^\epsilon\}}dsdt\Big]
\nonumber\\
~\leq\!\!\!\!\!\!\!\!&& C\Big[\mathbb{E}\Big(\int_\delta^T\int_{t-\delta}^t\|A(X^{\epsilon,\alpha,\phi^\epsilon}_s)\|_{V_1^*}^{\frac{\gamma_1}{\gamma_1-1}}\mathbf{1}_{\{t\leq\tau_N^\epsilon\}}dsdt\Big)\Big]^{\frac{\gamma_1-1}{\gamma_1}}
\nonumber\\
\!\!\!\!\!\!\!\!&&\cdot \Big[\mathbb{E}\Big(\int_\delta^T\int_{t-\delta}^t\|X^{\epsilon,\alpha,\phi^\epsilon}_s-X^{\epsilon,\alpha,\phi^\epsilon}_{t-\delta}\|_{V_1}^{\gamma_1}\mathbf{1}_{\{t\leq\tau_N^\epsilon\}}dsdt\Big)\Big]^{\frac{1}{\gamma_1}}
\nonumber\\
~\leq\!\!\!\!\!\!\!\!&&C\Big[\delta\mathbb{E}\Big(\int_0^T(1+\|X^{\epsilon,\alpha,\phi^\epsilon}_t\|^{\gamma_1}_{V_1})(1+\|X^{\epsilon,\alpha,\phi^\epsilon}_t\|^{\beta_1}_{H_1})\mathbf{1}_{\{t\leq\tau_N^\epsilon\}}dt\Big)\Big]^{\frac{\gamma_1-1}{\gamma_1}}
\Big[\delta\mathbb{E}\int_0^T\|X^{\epsilon,\alpha,\phi^\epsilon}_t\|^{\gamma_1}_{V_1}dt\Big]^{\frac{1}{\gamma_1}}
\nonumber\\
~\leq\!\!\!\!\!\!\!\!&&C_{N,T}\delta(1+\|x\|_{H_1}^2+\|y\|_{H_2}^2),
\end{eqnarray}
where we used the definition of $\tau_N^\epsilon$ in the last step, and the third inequality is owing to
\begin{eqnarray*}
\!\!\!\!\!\!\!\!&&\mathbb{E}\Big(\int_\delta^T\int_{t-\delta}^t\|A(X^{\epsilon,\alpha,\phi^\epsilon}_s)\|_{V_1^*}^{\frac{\gamma_1}{\gamma_1-1}}\mathbf{1}_{\{t\leq\tau_N^\epsilon\}}dsdt\Big)
\nonumber\\
~=\!\!\!\!\!\!\!\!&&\mathbb{E}\Big[\int_0^\delta\|A(X^{\epsilon,\alpha,\phi^\epsilon}_s)\|_{V_1^*}^{\frac{\gamma_1}{\gamma_1-1}}\Big(\int_{\delta}^{s+\delta}\mathbf{1}_{\{t\leq\tau_N^\epsilon\}}dt\Big)ds+\int_\delta^{T-\delta}\|A(X^{\epsilon,\alpha,\phi^\epsilon}_s)\|_{V_1^*}^{\frac{\gamma_1}{\gamma_1-1}}\Big(\int_{s}^{s+\delta}\mathbf{1}_{\{t\leq\tau_N^\epsilon\}}dt\Big)ds\Big]
\nonumber\\
\!\!\!\!\!\!\!\!&&+\mathbb{E}\Big[\int_{T-\delta}^{T}\|A(X^{\epsilon,\alpha,\phi^\epsilon}_s)\|_{V_1^*}^{\frac{\gamma_1}{\gamma_1-1}}\Big(\int_{s}^{T}\mathbf{1}_{\{t\leq\tau_N^\epsilon\}}dt\Big)ds\Big]
\nonumber\\
~\leq\!\!\!\!\!\!\!\!&&\delta\mathbb{E}\Big[\int_0^{T}\|A(X^{\epsilon,\alpha,\phi^\epsilon}_s)\|_{V_1^*}^{\frac{\gamma_1}{\gamma_1-1}}\mathbf{1}_{\{s\leq\tau_N^\epsilon\}}ds\Big].
\end{eqnarray*}
By Remark \ref{r1}, (\ref{l4}), (\ref{l5}) and H\"{o}lder's inequality, it follows that
\begin{eqnarray}\label{19}
\!\!\!\!\!\!\!\!&&\mathbb{E}\Big[\int_\delta^T|K_2(t)|\mathbf{1}_{\{t\leq\tau_N^\epsilon\}}dt\Big]
\nonumber\\
~\leq\!\!\!\!\!\!\!\!&&C\Big[\mathbb{E}\Big(\int_\delta^T\int_{t-\delta}^t\|F_1(X^{\epsilon,\alpha,\phi^\epsilon}_s,Y^{\epsilon,\alpha,\phi^\epsilon}_s)\|_{H_1}^2\mathbf{1}_{\{t\leq\tau_N^\epsilon\}}dsdt\Big)\Big]^{\frac{1}{2}}
\nonumber\\
\!\!\!\!\!\!\!\!&&\cdot\Big[\mathbb{E}\Big(\int_\delta^T\int_{t-\delta}^t\|X^{\epsilon,\alpha,\phi^\epsilon}_s-X^{\epsilon,\alpha,\phi^\epsilon}_{t-\delta}\|_{H_1}^2\mathbf{1}_{\{t\leq\tau_N^\epsilon\}}dsdt\Big)\Big]^{\frac{1}{2}}
\nonumber\\
~\leq\!\!\!\!\!\!\!\!&&C\Big[\delta\mathbb{E}\Big(\int_0^T(1+\|X^{\epsilon,\alpha,\phi^\epsilon}_s\|_{H_1}^2+\|Y^{\epsilon,\alpha,\phi^\epsilon}_s\|_{H_2}^2)ds\Big)\Big]^{\frac{1}{2}}\Big[\delta\mathbb{E}\Big(\int_0^T\|X^{\epsilon,\alpha,\phi^\epsilon}_s\|_{H_1}^2ds\Big)\Big]^{\frac{1}{2}}
\nonumber\\
~\leq\!\!\!\!\!\!\!\!&&C_{T}\delta(1+\|x\|_{H_1}^2+\|y\|_{H_2}^2).
\end{eqnarray}
From the condition $({\mathbf{A}}{\mathbf{3}})$, the term $\mathbb{E}\Big[\int_\delta^T|K_3(t)|\mathbf{1}_{\{t\leq\tau_N^\epsilon\}}dt\Big]$ can be controlled by
\begin{eqnarray}\label{20}
\!\!\!\!\!\!\!\!&&\mathbb{E}\Big[\int_\delta^T|K_3(t)|\mathbf{1}_{\{t\leq\tau_N^\epsilon\}}dt\Big]
\nonumber\\
~\leq\!\!\!\!\!\!\!\!&&C\Big[\mathbb{E}\Big(\int_\delta^T\int_{t-\delta}^t\|G_1(X^{\epsilon,\alpha,\phi^\epsilon}_s)\|_{L_2(U,H_1)}^2\|\phi_s^\epsilon\|_U^2\mathbf{1}_{\{t\leq\tau_N^\epsilon\}}dsdt\Big)\Big]^{\frac{1}{2}}
\nonumber\\
\!\!\!\!\!\!\!\!&&\cdot\Big[\mathbb{E}\Big(\int_\delta^T\int_{t-\delta}^t\|X^{\epsilon,\alpha,\phi^\epsilon}_s-X^{\epsilon,\alpha,\phi^\epsilon}_{t-\delta}\|_{H_1}^2\mathbf{1}_{\{t\leq\tau_N^\epsilon\}}dsdt\Big)\Big]^{\frac{1}{2}}
\nonumber\\
~\leq\!\!\!\!\!\!\!\!&&C\Big[\delta\mathbb{E}\Big(\int_0^T(1+\|X^{\epsilon,\alpha,\phi^\epsilon}_t\|_{H_1}^2)\|\phi_t^\epsilon\|_U^2\mathbf{1}_{\{t\leq\tau_N^\epsilon\}}dt\Big)\Big]^{\frac{1}{2}}
\Big[\delta\mathbb{E}\Big(\int_0^T\|X^{\epsilon,\alpha,\phi^\epsilon}_s\|_{H_1}^2ds\Big)\Big]^{\frac{1}{2}}
\nonumber\\
~\leq\!\!\!\!\!\!\!\!&&C_{N,T}\delta(1+\|x\|_{H_1}^2+\|y\|_{H_2}^2).
\end{eqnarray}
Following the similar calculations, the term $\mathbb{E}\Big[\int_\delta^T|K_4(t)|\mathbf{1}_{\{t\leq\tau_N^\epsilon\}}dt\Big]$ is estimated by
\begin{eqnarray}\label{21}
\mathbb{E}\Big[\int_\delta^T|K_4(t)|\mathbf{1}_{\{t\leq\tau_N^\epsilon\}}dt\Big]\leq\!\!\!\!\!\!\!\!&&C_T\delta\epsilon\mathbb{E}\Big[1+\sup_{t\in[0,T]}\|X^{\epsilon,\alpha,\phi^\epsilon}_t\|_{H_1}^2\Big]
\nonumber\\
\leq\!\!\!\!\!\!\!\!&&C_T\delta(1+\|x\|_{H_1}^2+\|y\|_{H_2}^2).
\end{eqnarray}
Applying Burkholder-Davis-Gundy's inequality and (\ref{14}), we indicate that
\begin{eqnarray}\label{22}
\!\!\!\!\!\!\!\!&&\mathbb{E}\Big[\int_\delta^T|K_5(t)|\mathbf{1}_{\{t\leq\tau_N^\epsilon\}}dt\Big]
\nonumber\\
~\leq\!\!\!\!\!\!\!\!&&C\int_\delta^T\Big[\mathbb{E}\Big(\int_{t-\delta}^t\|G_1(X^{\epsilon,\alpha,\phi^\epsilon}_s)\|_{L_2(U,H_1)}^2
\|X^{\epsilon,\alpha,\phi^\epsilon}_s-X^{\epsilon,\alpha,\phi^\epsilon}_{t-\delta}\|_{H_1}^2\mathbf{1}_{\{t\leq\tau_N^\epsilon\}}ds\Big)^{\frac{1}{2}}\Big]dt
\nonumber\\
~\leq\!\!\!\!\!\!\!\!&&C_T\Big[\mathbb{E}\Big(\int_\delta^T\int_{t-\delta}^t(1+\|X^{\epsilon,\alpha,\phi^\epsilon}_s\|_{H_1}^2)\|X^{\epsilon,\alpha,\phi^\epsilon}_s-X^{\epsilon,\alpha,\phi^\epsilon}_{t-\delta}\|_{H_1}^2\mathbf{1}_{\{t\leq\tau_N^\epsilon\}}dsdt\Big)\Big]^{\frac{1}{2}}
\nonumber\\
~\leq\!\!\!\!\!\!\!\!&&C_{N,T}\delta^{\frac{1}{2}}\Big[\mathbb{E}\Big(1+\sup_{t\in[0,T]}\|X^{\epsilon,\alpha,\phi^\epsilon}_t\|_{H_1}^2\Big)\Big]^{\frac{1}{2}}
\nonumber\\
~\leq\!\!\!\!\!\!\!\!&&C_{N,T}\delta^{\frac{1}{2}}(1+\|x\|_{H_1}^2+\|y\|_{H_2}^2).
\end{eqnarray}
Substituting (\ref{18})-(\ref{22}) into (\ref{23}), we conclude that
\begin{eqnarray}\label{64}
\mathbb{E}\Big[\int_\delta^T\|X^{\epsilon,\alpha,\phi^\epsilon}_t-X^{\epsilon,\alpha,\phi^\epsilon}_{t-\delta}\|_{H_1}^2\mathbf{1}_{\{t\leq\tau_N^\epsilon\}}dt\Big]
\leq C_{N,T}\delta^{\frac{1}{2}}(1+\|x\|_{H_1}^2+\|y\|_{H_2}^2).
\end{eqnarray}
Following the similar arguments as in the proof of (\ref{64}) gives that
\begin{eqnarray}\label{65}
\mathbb{E}\Big[\int_{\delta}^{T}\|X^{\epsilon,\alpha,\phi^\epsilon}_{t(\delta)}-X^{\epsilon,\alpha,\phi^\epsilon}_{t-\delta}\|_{H_1}^2\mathbf{1}_{\{t\leq\tau_N^\epsilon\}}dt\Big]
\leq C_{N,T}\delta^{\frac{1}{2}}(1+\|x\|_{H_1}^2+\|y\|_{H_2}^2).
\end{eqnarray}
Finally, combining (\ref{64})-(\ref{65}) with (\ref{17}) implies Lemma \ref{l7}. The proof is completed.   \hspace{\fill}$\Box$
\end{proof}

\section{Proof of main results}
\setcounter{equation}{0}
 \setcounter{definition}{0}
In this section, we aim to prove the main results in Theorem \ref{t1} and \ref{t3}. We first consider an auxiliary equation associated with the fast component of Eq.~(\ref{e1}), which help us to prove the convergence of solutions to the stochastic control problem (\ref{e5}).

\subsection{The construction of an auxiliary process}
Since we want to use the approach of time discretization developed by Khasminskii \cite{K}, we first establish the following auxiliary process $\widehat{Y}_t^{\epsilon,\alpha}\in H_2$ and divide the time interval $[0,T]$ into some subintervals of size $\delta>0$ depending on $\alpha$, which will be chosen appropriately in the next subsection. Consider the following SPDE
\begin{equation}\label{e6}
\left\{ \begin{aligned}
&d\widehat{Y}^{\epsilon,\alpha}_t=\frac{1}{\alpha}F_2(X^{\epsilon,\alpha,\phi^\epsilon}_{t(\delta)},\widehat{Y}^{\epsilon,\alpha}_t)dt+\frac{1}{\sqrt{\alpha}}G_2dW_t,\\
&\widehat{Y}^{\epsilon,\alpha}_0=y.
\end{aligned} \right.
\end{equation}
It is easy to see that for each $k\in \mathbb{N}$ and $t\in[k\delta,(k+1)\delta\wedge T]$,
\begin{equation}\label{e7}
\widehat{Y}^{\epsilon,\alpha}_t=\widehat{Y}^{\epsilon,\alpha}_{k\delta}+\frac{1}{\alpha}\int_{k\delta}^tF_2(X^{\epsilon,\alpha,\phi^\epsilon}_{k\delta},\widehat{Y}^{\epsilon,\alpha}_s)ds
+\frac{1}{\sqrt{\alpha}}\int_{k\delta}^tG_2dW_s.
\end{equation}

Following almost same calculations as in the proof of Lemma \ref{l6}, one can easily get the energy estimate for $\widehat{Y}^{\epsilon,\alpha}_t$ as follows.
\begin{lemma}\label{l8}
For any initial values $x\in H_1$, $y\in H_2$ and $\epsilon,\alpha\in(0,1)$, there is a constant $C>0$ such that
\begin{equation}\label{24}
\sup_{t\in[0,T]}\mathbb{E}\|\widehat{Y}^{\epsilon,\alpha}_t\|_{H_2}^2\leq C(1+\|x\|_{H_1}^2+\|y\|_{H_2}^2).
\end{equation}
\end{lemma}

Now we would like to prove an important lemma characterizing the difference between processes $\widehat{Y}^{\epsilon,\alpha}_t$ and $Y^{\epsilon,\alpha,\phi^\epsilon}_t$.
\begin{lemma}\label{l9}
For any $x\in H_1$, $y\in H_2$ and $\epsilon,\alpha\in(0,1)$, there is a constant $C_N>0$ such that
\begin{equation*}
\mathbb{E}\Big[\int_0^{T\wedge\tau_N^\epsilon}\|Y^{\epsilon,\alpha,\phi^\epsilon}_t-\widehat{Y}^{\epsilon,\alpha}_t\|_{H_2}^2dt\Big]\leq C_N(1+\|x\|_{H_1}^2+\|y\|_{H_2}^2)\Big[\Big(\frac{\alpha}{\epsilon}\Big)+\delta^{\frac{1}{2}}\Big].
\end{equation*}
\end{lemma}
\begin{proof}
Letting $Z_t:=Y^{\epsilon,\alpha,\phi^\epsilon}_t-\widehat{Y}^{\epsilon,\alpha}_t$, which fulfills
\begin{eqnarray*}
dZ_t=\!\!\!\!\!\!\!\!&&\frac{1}{\alpha}\Big[F_2(X^{\epsilon,\alpha,\phi^\epsilon}_t,Y^{\epsilon,\alpha,\phi^\epsilon}_t)-F_2(X^{\epsilon,\alpha,\phi^\epsilon}_{t(\delta)},\widehat{Y}^{\epsilon,\alpha}_t)\Big]dt
+\frac{1}{\sqrt{\alpha\epsilon}}G_2\phi^\epsilon_tdt,~~Z_0=0.
\end{eqnarray*}
Then it follows that
\begin{eqnarray}\label{28}
\frac{d}{dt}\|Z_{t}\|_{H_2}^2=\!\!\!\!\!\!\!\!&&\frac{2}{\alpha}\Big[{}_{V_2^*}\langle
F_2(X^{\epsilon,\alpha,\phi^\epsilon}_t,Y^{\epsilon,\alpha,\phi^\epsilon}_t)-F_2(X^{\epsilon,\alpha,\phi^\epsilon}_{t(\delta)},\widehat{Y}^{\epsilon,\alpha}_t),Z_t\rangle_{V_2}\Big]
+\frac{2}{\sqrt{\alpha\epsilon}}\langle G_2\phi^\epsilon_t,Z_t\rangle_{H_2}
\nonumber\\
=:\!\!\!\!\!\!\!\!&&\sum_{i=1}^2I_i(t).
\end{eqnarray}
Let us estimate the terms $I_i(t)$, $i=1,2$, respectively. Taking the condition $({\mathbf{H}}{\mathbf{2}})$ into account, we have
\begin{eqnarray}\label{25}
I_1(t)=\!\!\!\!\!\!\!\!&&\frac{2}{\alpha}{}_{V_2^*}\langle
F_2(X^{\epsilon,\alpha,\phi^\epsilon}_t,Y^{\epsilon,\alpha,\phi^\epsilon}_t)-F_2(X^{\epsilon,\alpha,\phi^\epsilon}_{t(\delta)},Y^{\epsilon,\alpha,\phi^\epsilon}_t),Z_t\rangle_{V_2}
\nonumber\\
\!\!\!\!\!\!\!\!&&+\frac{2}{\alpha}{}_{V_2^*}\langle
F_2(X^{\epsilon,\alpha,\phi^\epsilon}_{t(\delta)},Y^{\epsilon,\alpha,\phi^\epsilon}_t)-F_2(X^{\epsilon,\alpha,\phi^\epsilon}_{t(\delta)},\widehat{Y}^{\epsilon,\alpha}_t),Z_t\rangle_{V_2}
\nonumber\\
\leq\!\!\!\!\!\!\!\!&&-\frac{2\kappa-\varepsilon_0}{\alpha}\|Z_{t}\|_{H_2}^2+\frac{C}{\alpha}\|X^{\epsilon,\alpha,\phi^\epsilon}_t-X^{\epsilon,\alpha,\phi^\epsilon}_{t(\delta)}\|_{H_1}^2,
\end{eqnarray}
where we used Young's inequality in the last step with a small enough constant $\varepsilon_0>0$.

Using Young's inequality gives that
\begin{eqnarray}\label{26}
I_2(t)\leq\!\!\!\!\!\!\!\!&&\frac{2}{\sqrt{\alpha\epsilon}}\sqrt{tr(G_2G_2^*)}\|\phi^\epsilon_t\|_U\|Z_t\|_{H_2}
\nonumber\\
\leq\!\!\!\!\!\!\!\!&&\frac{\varepsilon_0}{\alpha}\|Z_{t}\|_{H_2}^2+\frac{C}{\epsilon}\|\phi^\epsilon_t\|_U^2.
\end{eqnarray}
Substituting (\ref{25})-(\ref{26})  into (\ref{28}) leads to
\begin{eqnarray*}
\frac{d}{dt}\|Z_{t}\|_{H_2}^2\leq\!\!\!\!\!\!\!\!&&-\frac{2\kappa-2\varepsilon_0}{\alpha}\|Z_{t}\|_{H_2}^2+\frac{C}{\epsilon}\|\phi^\epsilon_t\|_U^2
+\frac{C}{\alpha}\|X^{\epsilon,\alpha,\phi^\epsilon}_t-X^{\epsilon,\alpha,\phi^\epsilon}_{t(\delta)}\|_{H_1}^2.
\end{eqnarray*}
By the comparison theorem we have
\begin{eqnarray*}
\|Z_{t}\|_{H_2}^2\leq\!\!\!\!\!\!\!\!&&\frac{C}{\epsilon}\int_0^te^{-\frac{\eta}{\alpha}(t-s)}\|\phi^\epsilon_s\|_U^2ds
+\frac{C}{\alpha}\int_0^te^{-\frac{\eta}{\alpha}(t-s)}\|X^{\epsilon,\alpha,\phi^\epsilon}_s-X^{\epsilon,\alpha,\phi^\epsilon}_{s(\delta)}\|_{H_1}^2ds,
\end{eqnarray*}
here we denote $\eta:=2\kappa-2\varepsilon_0>0$. Multiplying $\mathbf{1}_{\{t\leq\tau_N^\epsilon\}}$ for both sides of the above inequality and taking expectation yields that
\begin{eqnarray*}
\mathbb{E}\Big[\|Z_{t}\|_{H_2}^2\mathbf{1}_{\{t\leq\tau_N^\epsilon\}}\Big]\leq\!\!\!\!\!\!\!\!&&\frac{C}{\epsilon}\mathbb{E}\int_0^{t}e^{-\frac{\eta}{\alpha}(t-s)}\|\phi^\epsilon_s\|_U^2ds
+\frac{C}{\alpha}\mathbb{E}\int_0^{t\wedge\tau_N^\epsilon}e^{-\frac{\eta}{\alpha}(t-s)}\|X^{\epsilon,\alpha,\phi^\epsilon}_s-X^{\epsilon,\alpha,\phi^\epsilon}_{s(\delta)}\|_{H_1}^2ds.
\end{eqnarray*}
Hence, according to Fubini's theorem, one can conclude that
\begin{eqnarray*}
\mathbb{E}\int_0^{T\wedge\tau_N^\epsilon}\|Z_t\|_{H_2}^2dt=\!\!\!\!\!\!\!\!&&\int_0^T\mathbb{E}\Big[\|Z_{t}\|_{H_2}^2\mathbf{1}_{\{t\leq\tau_N^\epsilon\}}\Big]dt
\nonumber\\
\leq\!\!\!\!\!\!\!\!&&\frac{C}{\epsilon}\int_0^T\int_0^te^{-\frac{\eta}{\alpha}(t-s)}\mathbb{E}\Big[\|\phi^\epsilon_s\|_U^2\Big]dsdt
\nonumber\\
\!\!\!\!\!\!\!\!&&+\frac{C}{\alpha}\int_0^T\int_0^te^{-\frac{\eta}{\alpha}(t-s)}\mathbb{E}\Big[\|X^{\epsilon,\alpha,\phi^\epsilon}_s-X^{\epsilon,\alpha,\phi^\epsilon}_{s(\delta)}\|_{H_1}^2\mathbf{1}_{\{s\leq\tau_N^\epsilon\}}\Big]dsdt
\nonumber\\
\leq\!\!\!\!\!\!\!\!&&\frac{C}{\epsilon}\mathbb{E}\Big[\int_0^T\|\phi^\epsilon_s\|_U^2\Big(\int_s^Te^{-\frac{\eta}{\alpha}(t-s)}dt\Big)ds\Big]
\nonumber\\
\!\!\!\!\!\!\!\!&&+\frac{C}{\alpha}\mathbb{E}\Big[\int_0^T\mathbf{1}_{\{s\leq\tau_N^\epsilon\}}\|X^{\epsilon,\alpha,\phi^\epsilon}_s-X^{\epsilon,\alpha,\phi^\epsilon}_{s(\delta)}\|_{H_1}^2\Big(\int_s^Te^{-\frac{\eta}{\alpha}(t-s)}dt\Big)ds\Big]
\nonumber\\
\leq\!\!\!\!\!\!\!\!&&\frac{C}{\eta}\Big(\frac{\alpha}{\epsilon}\Big)\mathbb{E}\Big[\int_0^T\|\phi^\epsilon_t\|_U^2dt\Big]
+\frac{C}{\eta}\mathbb{E}\Big[\int_0^{T\wedge\tau_N^\epsilon}\|X^{\epsilon,\alpha,\phi^\epsilon}_t-X^{\epsilon,\alpha,\phi^\epsilon}_{t(\delta)}\|_{H_1}^2dt\Big]
\nonumber\\
\leq\!\!\!\!\!\!\!\!&&C_{N,T,M}(1+\|x\|_{H_1}^2+\|y\|_{H_2}^2)\Big[\Big(\frac{\alpha}{\epsilon}\Big)+\delta^{\frac{1}{2}}\Big],
\end{eqnarray*}
where the last inequality is owing to  Lemma \ref{l6} and Lemma \ref{l7}, which completes the proof.

\hspace{\fill}$\Box$
\end{proof}

\subsection{Weak convergence}
In this subsection, the aim is to prove that the process $X^{\epsilon,\alpha,\phi^\epsilon}_t$ defined in Eq.~(\ref{e5}) converges to the solution $\bar{X}^{\phi}_t$ of deterministic skeleton equation (\ref{e2}) in distribution, which verifies the \textbf{Condition (A)} (i).

Repeating the very similar arguments as in the proof of Lemma \ref{l7}, one can easily conclude the following lemma.
\begin{lemma}\label{l10}
For  $x\in H_2$ and $\delta>0$ small enough, there is a constant $C>0$ such that
$$\sup_{\phi\in\mathcal{S}_M}\int_0^T\|\bar{X}^\phi_t-\bar{X}^\phi_{t(\delta)}\|_{H_1}^2dt\leq C(1+\|x\|_{H_1}^2)\delta^{\frac{1}{2}}.$$
\end{lemma}

Now we define the following stopping time
$$\widetilde{\tau}_N^\epsilon:=\inf\Big\{t\in[0,T]:\|X^{\epsilon,\alpha,\phi^\epsilon}_t\|_{H_1}+\|\bar{X}^{\phi}_t\|_{H_1}+\int_0^t\|X^{\epsilon,\alpha,\phi^\epsilon}_s\|_{V_1}^{\gamma_1}ds+\int_0^t\|\bar{X}^{\phi}_s\|_{V_1}^{\gamma_1}ds>N\Big\}.$$

\begin{theorem}\label{t4}
Assume that the conditions in Theorem \ref{t1} hold. Let $\{\phi^\epsilon: \epsilon>0\}\subset \mathcal{A}_M$ for
some $M<\infty$. If $\phi^\epsilon$ converge to $\phi$ in distribution
as $S_M$-valued random elements, then
$$\mathcal{G}^\epsilon\left(W_\cdot+\frac{1}{\sqrt{\epsilon}} \int_0^\cdot
\phi^\epsilon_s\ d s \right)\rightarrow \mathcal{G}^0\left(\int_0^\cdot
 \phi_s\ d s \right)  $$
in distribution as $\epsilon\rightarrow 0$.
\end{theorem}

\begin{proof}
We  separate the proof into four steps to prove the convergence of solutions of Eq.~(\ref{e1}) to the solution of Eq.~(\ref{e2}) in probability, which implies the convergence in distribution as $\epsilon\to 0$.

\textbf{Step 1}: Denote $\widetilde{Z}^{\epsilon}_t:=X^{\epsilon,\alpha,\phi^\epsilon}_t-\bar{X}^{\phi}_t$, which satisfies the following SPDE
\begin{eqnarray*}
\left\{ \begin{aligned}
d\widetilde{Z}^{\epsilon}_t=&\big[A(X^{\epsilon,\alpha,\phi^\epsilon}_t)-A(\bar{X}^{\phi}_t)\big]dt
+\big[F_1(X^{\epsilon,\alpha,\phi^\epsilon}_t,Y^{\epsilon,\alpha,\phi^\epsilon}_t)-\bar{F}_1(\bar{X}^{\phi}_t)\big]dt\\
&+\big[G_1(X^{\epsilon,\alpha,\phi^\epsilon}_t)\phi^\epsilon_t-G_1(\bar{X}^{\phi}_t)\phi_t\big]dt
+\sqrt{\epsilon}G_1(X^{\epsilon,\alpha,\phi^\epsilon}_t)dW_t,\\
\widetilde{Z}^{\epsilon}_0=&0.
\end{aligned}\right.
\end{eqnarray*}
Applying It\^{o}'s formula to $\|\widetilde{Z}^{\epsilon}_t\|_{H_1}^2$ we obtain
\begin{eqnarray*}
\|\widetilde{Z}^{\epsilon}_t\|_{H_1}^2=\!\!\!\!\!\!\!\!&&2\int_0^t{}_{V_1^*}\langle A(X^{\epsilon,\alpha,\phi^\epsilon}_s)-A(\bar{X}^{\phi}_s),\widetilde{Z}^{\epsilon}_s\rangle_{V_1}ds+2\int_0^t\langle F_1(X^{\epsilon,\alpha,\phi^\epsilon}_s,Y^{\epsilon,\alpha,\phi^\epsilon}_s)-\bar{F}_1(\bar{X}^{\phi}_s),\widetilde{Z}^{\epsilon}_s\rangle_{H_1}ds
\nonumber\\
\!\!\!\!\!\!\!\!&&+2\int_0^t\langle G_1(X^{\epsilon,\alpha,\phi^\epsilon}_s)\phi^\epsilon_s-G_1(\bar{X}^{\phi}_s)\phi_s,\widetilde{Z}^{\epsilon}_s\rangle_{H_1}ds+\epsilon\int_0^t\|G_1(X^{\epsilon,\alpha,\phi^\epsilon}_s)\|_{L_2(U,H_1)}^2ds
\nonumber\\
\!\!\!\!\!\!\!\!&&+2\sqrt{\epsilon}\int_0^t\langle \widetilde{Z}^{\epsilon}_s, G_1(X^{\epsilon,\alpha,\phi^\epsilon}_s)dW_s\rangle_{H_1}.
\end{eqnarray*}
Then it is easy to get that
\begin{eqnarray}\label{34}
\|\widetilde{Z}^{\epsilon}_t\|_{H_1}^2=\!\!\!\!\!\!\!\!&&2\int_0^t{}_{V_1^*}\langle A(X^{\epsilon,\alpha,\phi^\epsilon}_s)-A(\bar{X}^{\phi}_s),\widetilde{Z}^{\epsilon}_s\rangle_{V_1}ds
+2\int_0^t\langle\bar{F}_1(X^{\epsilon,\alpha,\phi^\epsilon}_s)-\bar{F}_1(\bar{X}^{\phi}_s),\widetilde{Z}^{\epsilon}_s\rangle_{H_1}ds
\nonumber\\
\!\!\!\!\!\!\!\!&&+2\int_0^t\langle F_1(X^{\epsilon,\alpha,\phi^\epsilon}_s,Y^{\epsilon,\alpha,\phi^\epsilon}_s)-\bar{F}_1(X^{\epsilon,\alpha,\phi^\epsilon}_s)
-F_1(X^{\epsilon,\alpha,\phi^\epsilon}_{s(\delta)},\widehat{Y}^{\epsilon,\alpha}_s)+\bar{F}_1(X^{\epsilon,\alpha,\phi^\epsilon}_{s(\delta)}),\widetilde{Z}^{\epsilon}_s\rangle_{H_1}ds
\nonumber\\
\!\!\!\!\!\!\!\!&&+2\int_0^t\langle F_1(X^{\epsilon,\alpha,\phi^\epsilon}_{s(\delta)},\widehat{Y}^{\epsilon,\alpha}_s)-\bar{F}_1(X^{\epsilon,\alpha,\phi^\epsilon}_{s(\delta)}),\widetilde{Z}^{\epsilon}_s-\widetilde{Z}^{\epsilon}_{s(\delta)}\rangle_{H_1}ds
\nonumber\\
\!\!\!\!\!\!\!\!&&+2\int_0^t\langle F_1(X^{\epsilon,\alpha,\phi^\epsilon}_{s(\delta)},\widehat{Y}^{\epsilon,\alpha}_s)-\bar{F}_1(X^{\epsilon,\alpha,\phi^\epsilon}_{s(\delta)}),\widetilde{Z}^{\epsilon}_{s(\delta)}\rangle_{H_1}ds
\nonumber\\
\!\!\!\!\!\!\!\!&&+2\int_0^t\langle \big[G_1(X^{\epsilon,\alpha,\phi^\epsilon}_s)-G_1(\bar{X}^{\phi}_s)\big]\phi^\epsilon_s,\widetilde{Z}^{\epsilon}_s\rangle_{H_1}ds
+2\int_0^t\langle G_1(\bar{X}^{\phi}_s)(\phi^\epsilon_s-\phi_s),\widetilde{Z}^{\epsilon}_s\rangle_{H_1}ds
\nonumber\\
\!\!\!\!\!\!\!\!&&+\epsilon\int_0^t\|G_1(X^{\epsilon,\alpha,\phi^\epsilon}_s)\|_{L_2(U,H_1)}^2ds+2\sqrt{\epsilon}\int_0^t\langle \widetilde{Z}^{\epsilon}_s, G_1(X^{\epsilon,\alpha,\phi^\epsilon}_s)dW_s\rangle_{H_1}
\nonumber\\
=:\!\!\!\!\!\!\!\!&&\sum_{i=1}^9 I_i(t).
\end{eqnarray}
Taking $({\mathbf{A}}{\mathbf{2}})$ and Young's inequality into account we have
\begin{eqnarray}\label{29}
I_1(t)+I_6(t)\leq\!\!\!\!\!\!\!\!&&\int_0^t2{}_{V_1^*}\langle A(X^{\epsilon,\alpha,\phi^\epsilon}_s)-A(\bar{X}^{\phi}_s),\widetilde{Z}^{\epsilon}_s\rangle_{V_1}+\|G_1(X^{\epsilon,\alpha,\phi^\epsilon}_s)-G_1(\bar{X}^{\phi}_s)\|_{L_2(U,H_1)}^2ds
\nonumber\\
\!\!\!\!\!\!\!\!&&+\int_0^t\|\phi^\epsilon_s\|_U^2\|\widetilde{Z}^{\epsilon}_s\|_{H_1}^2ds
\nonumber\\
\leq\!\!\!\!\!\!\!\!&&-\theta_1\int_0^t\|\widetilde{Z}^{\epsilon}_s\|_{V_1}^{\gamma_1}ds+\int_0^t\big (K+\rho(X^{\epsilon,\alpha,\phi^\epsilon}_s)\big)\|\widetilde{Z}^{\epsilon}_s\|_{H_1}^2ds+\int_0^t\|\phi^\epsilon_s\|_U^2\|\widetilde{Z}^{\epsilon}_s\|_{H_1}^2ds.~~~~~
\end{eqnarray}
Since $\bar{F}_1$ is Lipschitz continuous,  following from the proof of Lemma \ref{l5} we have
\begin{eqnarray}\label{30}
I_2(t)\leq2\int_0^t\|\bar{F}_1(X^{\epsilon,\alpha,\phi^\epsilon}_s)-\bar{F}_1(\bar{X}^{\phi}_s)\|_{H_1}\|\widetilde{Z}^{\epsilon}_s\|_{H_1}ds\leq C\int_0^t\|\widetilde{Z}^{\epsilon}_s\|_{H_1}^2ds.
\end{eqnarray}
Similarly, by $({\mathbf{A}}{\mathbf{2}})$ and Young's inequality, it leads to
\begin{eqnarray}\label{31}
I_3(t)\leq\int_0^t\|\widetilde{Z}^{\epsilon}_s\|_{H_1}^2ds+C\int_0^t\Big(\|X^{\epsilon,\alpha,\phi^\epsilon}_s-X^{\epsilon,\alpha,\phi^\epsilon}_{s(\delta)}\|_{H_1}^2
+\|Y^{\epsilon,\alpha,\phi^\epsilon}_s-\widehat{Y}^{\epsilon,\alpha}_s\|_{H_2}^2\Big)ds.
\end{eqnarray}
Making use of Young's inequality and H\"{o}lder's inequality yields that
\begin{eqnarray}\label{32}
I_4(t)\leq\!\!\!\!\!\!\!\!&&2\int_0^t\|F_1(X^{\epsilon,\alpha,\phi^\epsilon}_{s(\delta)},\widehat{Y}^{\epsilon,\alpha}_s)-\bar{F}_1(X^{\epsilon,\alpha,\phi^\epsilon}_{s(\delta)})\|_{H_1}\|\widetilde{Z}^{\epsilon}_s-\widetilde{Z}^{\epsilon}_{s(\delta)}\|_{H_1}ds
\nonumber\\
\leq\!\!\!\!\!\!\!\!&&4\Big[\int_0^t\Big(\|F_1(X^{\epsilon,\alpha,\phi^\epsilon}_{s(\delta)},\widehat{Y}^{\epsilon,\alpha}_s)\|_{H_1}^2+\|\bar{F}_1(X^{\epsilon,\alpha,\phi^\epsilon}_{s(\delta)})\|_{H_1}^2\Big)ds\Big]^{\frac{1}{2}}
\nonumber\\
\!\!\!\!\!\!\!\!&&\cdot\Big[\int_0^t\Big(\|X^{\epsilon,\alpha,\phi^\epsilon}_s-X^{\epsilon,\alpha,\phi^\epsilon}_{s(\delta)}\|_{H_1}^2+\|\bar{X}^{\phi}_s-\bar{X}^{\phi}_{s(\delta)}\|_{H_1}^2\Big)ds\Big]^{\frac{1}{2}}
\nonumber\\
\leq\!\!\!\!\!\!\!\!&&C\Big[\int_0^t\Big(1+\|X^{\epsilon,\alpha,\phi^\epsilon}_{s(\delta)}\|_{H_1}^2+\|\widehat{Y}^{\epsilon,\alpha}_s\|_{H_1}^2\Big)ds\Big]^{\frac{1}{2}}
\nonumber\\
\!\!\!\!\!\!\!\!&&\cdot\Big[\int_0^t\Big(\|X^{\epsilon,\alpha,\phi^\epsilon}_s-X^{\epsilon,\alpha,\phi^\epsilon}_{s(\delta)}\|_{H_1}^2+\|\bar{X}^{\phi}_s-\bar{X}^{\phi}_{s(\delta)}\|_{H_1}^2\Big)ds\Big]^{\frac{1}{2}}.
\end{eqnarray}
Substituting (\ref{29})-(\ref{32}) into (\ref{34}) and then we have
\begin{eqnarray}\label{66}
\!\!\!\!\!\!\!\!&&\|\widetilde{Z}^{\epsilon}_t\|_{H_1}^2+\theta_1\int_0^t\|\widetilde{Z}^{\epsilon}_s\|_{V_1}^{\gamma_1}ds
\nonumber\\
~\leq\!\!\!\!\!\!\!\!&&C\int_0^t\big(1+\rho(X^{\epsilon,\alpha,\phi^\epsilon}_s)+\|\phi^\epsilon_s\|_U^2\big)\|\widetilde{Z}^{\epsilon}_s\|_{H_1}^2ds+C\epsilon\int_0^t\big(1+\|\bar{X}^{\phi}_s\|_{H_1}^2\big)ds
\nonumber\\
\!\!\!\!\!\!\!\!&&+C\int_0^t\Big(\|X^{\epsilon,\alpha,\phi^\epsilon}_s-X^{\epsilon,\alpha,\phi^\epsilon}_{s(\delta)}\|_{H_1}^2
+\|Y^{\epsilon,\alpha,\phi^\epsilon}_s-\widehat{Y}^{\epsilon,\alpha}_s\|_{H_2}^2\Big)ds
\nonumber\\
\!\!\!\!\!\!\!\!&&+
C\Big[\int_0^t\Big(1+\|X^{\epsilon,\alpha,\phi^\epsilon}_{s(\delta)}\|_{H_1}^2+\|\widehat{Y}^{\epsilon,\alpha}_s\|_{H_1}^2\Big)ds\Big]^{\frac{1}{2}}
\nonumber\\
\!\!\!\!\!\!\!\!&&~~~\cdot
\Big[\int_0^t\Big(\|X^{\epsilon,\alpha,\phi^\epsilon}_s-X^{\epsilon,\alpha,\phi^\epsilon}_{s(\delta)}\|_{H_1}^2+\|\bar{X}^{\phi}_s-\bar{X}^{\phi}_{s(\delta)}\|_{H_1}^2\Big)ds\Big]^{\frac{1}{2}}
\nonumber\\
\!\!\!\!\!\!\!\!&&+2\int_0^t\langle G_1(\bar{X}^{\phi}_s)(\phi^\epsilon_s-\phi_s),\widetilde{Z}^{\epsilon}_s\rangle_{H_1}ds+2\sqrt{\epsilon}\int_0^t\langle \widetilde{Z}^{\epsilon}_s, G_1(X^{\epsilon,\alpha,\phi^\epsilon}_s)dW_s\rangle_{H_1}
\nonumber\\
\!\!\!\!\!\!\!\!&&+\epsilon\int_0^t\|G_1(X^{\epsilon,\alpha,\phi^\epsilon}_s)\|_{L_2(U,H_1)}^2ds+2\int_0^t\langle F_1(X^{\epsilon,\alpha,\phi^\epsilon}_{s(\delta)},\widehat{Y}^{\epsilon,\alpha}_s)-\bar{F}_1(X^{\epsilon,\alpha,\phi^\epsilon}_{s(\delta)}),\widetilde{Z}^{\epsilon}_{s(\delta)}\rangle_{H_1}ds.~~~~
\end{eqnarray}
Applying Gronwall's lemma to (\ref{66}) and  using the definition of $\widetilde{\tau}_N^\epsilon$, it follows that
\begin{eqnarray*}
\!\!\!\!\!\!\!\!&&\sup_{t\in[0,T\wedge\widetilde{\tau}_N^\epsilon]}\|\widetilde{Z}^{\epsilon}_t\|_{H_1}^2+\theta_1\int_0^{T\wedge\widetilde{\tau}_N^\epsilon}\|\widetilde{Z}^{\epsilon}_t\|_{V_1}^{\gamma_1}dt
\nonumber\\
~\leq\!\!\!\!\!\!\!\!&&C_N\Big\{\epsilon\int_0^T\big(1+\|\bar{X}^{\phi}_t\|_{H_1}^2\big)dt+\int_0^{T\wedge\widetilde{\tau}_N^\epsilon}\Big(\|X^{\epsilon,\alpha,\phi^\epsilon}_t-X^{\epsilon,\alpha,\phi^\epsilon}_{t(\delta)}\|_{H_1}^2
+\|Y^{\epsilon,\alpha,\phi^\epsilon}_t-\widehat{Y}^{\epsilon,\alpha}_t\|_{H_2}^2\Big)dt
\nonumber\\
\!\!\!\!\!\!\!\!&&+\epsilon\int_0^{T\wedge\widetilde{\tau}_N^\epsilon}\|G_1(X^{\epsilon,\alpha,\phi^\epsilon}_t)\|_{L_2(U,H_1)}^2dt+
\Big[\int_0^T\Big(1+\|X^{\epsilon,\alpha,\phi^\epsilon}_{t(\delta)}\|_{H_1}^2+\|\widehat{Y}^{\epsilon,\alpha}_t\|_{H_1}^2\Big)dt\Big]^{\frac{1}{2}}
\nonumber\\
\!\!\!\!\!\!\!\!&&~~\cdot\Big[\int_0^{T\wedge\widetilde{\tau}_N^\epsilon}\|X^{\epsilon,\alpha,\phi^\epsilon}_t-X^{\epsilon,\alpha,\phi^\epsilon}_{t(\delta)}\|_{H_1}^2dt+\int_0^T\|\bar{X}^{\phi}_t-\bar{X}^{\phi}_{t(\delta)}\|_{H_1}^2dt\Big]^{\frac{1}{2}}
\nonumber\\
\!\!\!\!\!\!\!\!&&+\sup_{t\in[0,T\wedge\widetilde{\tau}_N^\epsilon]}\Big|\int_0^{t}\langle G_1(\bar{X}^{\phi}_s)(\phi^\epsilon_s-\phi_s),\widetilde{Z}^{\epsilon}_s\rangle_{H_1}ds\Big|+\sqrt{\epsilon}\sup_{t\in[0,T\wedge\widetilde{\tau}_N^\epsilon]}\Big|\int_0^t\langle \widetilde{Z}^{\epsilon}_t, G_1(X^{\epsilon,\alpha,\phi^\epsilon}_t)dW_t\rangle_{H_1}\Big|
\nonumber\\
\!\!\!\!\!\!\!\!&&+\sup_{t\in[0,T\wedge\widetilde{\tau}_N^\epsilon]}\Big|\int_0^t\langle F_1(X^{\epsilon,\alpha,\phi^\epsilon}_{s(\delta)},\widehat{Y}^{\epsilon,\alpha}_s)-\bar{F}_1(X^{\epsilon,\alpha,\phi^\epsilon}_{s(\delta)}),\widetilde{Z}^{\epsilon}_{s(\delta)}\rangle_{H_1}ds\Big|\Big\}
\cdot \exp\Big(\int_0^{T}\|\phi^\epsilon_t\|_U^2dt\Big).~~~~~~~
\end{eqnarray*}
Taking expectation for the above inequality and using Lemma \ref{l7}, \ref{l8}, \ref{l9} and \ref{l10} yields that
\begin{eqnarray}\label{35}
\!\!\!\!\!\!\!\!&&\mathbb{E}\Big[\sup_{t\in[0,T\wedge\widetilde{\tau}_N^\epsilon]}\|\widetilde{Z}^{\epsilon}_t\|_{H_1}^2\Big]+\theta_1\mathbb{E}\int_0^{T\wedge\widetilde{\tau}_N^\epsilon}\|\widetilde{Z}^{\epsilon}_t\|_{V_1}^{\gamma_1}dt
\nonumber\\
~\leq\!\!\!\!\!\!\!\!&&C_{N,M,T}\Big\{\epsilon\sup_{t\in[0,T]}\big(1+\|\bar{X}^{\phi}_t\|_{H_1}^2\big)+(1+\|x\|_{H_1}^2+\|y\|_{H_2}^2)\Big[\Big(\frac{\alpha}{\epsilon}\Big)+\delta^{\frac{1}{4}}\Big]
\nonumber\\
\!\!\!\!\!\!\!\!&&+\mathbb{E}\Big[\sup_{t\in[0,T\wedge\widetilde{\tau}_N^\epsilon]}\Big|\int_0^{t}\langle G_1(\bar{X}^{\phi}_s)(\phi^\epsilon_s-\phi_s),\widetilde{Z}^{\epsilon}_s\rangle_{H_1}ds\Big|\Big]
\nonumber\\
\!\!\!\!\!\!\!\!&&
+\sqrt{\epsilon}\mathbb{E}\Big[\sup_{t\in[0,T\wedge\widetilde{\tau}_N^\epsilon]}\Big|\int_0^t\langle \widetilde{Z}^{\epsilon}_t, G_1(X^{\epsilon,\alpha,\phi^\epsilon}_t)dW_t\rangle_{H_1}\Big|\Big]+\epsilon\mathbb{E}\int_0^{T\wedge\widetilde{\tau}_N^\epsilon}\|G_1(X^{\epsilon,\alpha,\phi^\epsilon}_t)\|_{L_2(U,H_1)}^2dt
\nonumber\\
\!\!\!\!\!\!\!\!&&+\mathbb{E}\Big[\sup_{t\in[0,T\wedge\widetilde{\tau}_N^\epsilon]}\Big|\int_0^t\langle F_1(X^{\epsilon,\alpha,\phi^\epsilon}_{s(\delta)},\widehat{Y}^{\epsilon,\alpha}_s)-\bar{F}_1(X^{\epsilon,\alpha,\phi^\epsilon}_{s(\delta)}),\widetilde{Z}^{\epsilon}_{s(\delta)}\rangle_{H_1}ds\Big|\Big]\Big\}.
\end{eqnarray}
Making use of Burkholder-Davis-Gundy's inequality and Young's inequality implies that
\begin{eqnarray}\label{36}
\!\!\!\!\!\!\!\!&&C_{N,M,T}\Big\{\sqrt{\epsilon}\mathbb{E}\Big[\sup_{t\in[0,T\wedge\widetilde{\tau}_N^\epsilon]}\Big|\int_0^t\langle \widetilde{Z}^{\epsilon}_t, G_1(X^{\epsilon,\alpha,\phi^\epsilon}_t)dW_t\rangle_{H_1}\Big|\Big]+\epsilon\mathbb{E}\int_0^{T\wedge\widetilde{\tau}_N^\epsilon}\|G_1(X^{\epsilon,\alpha,\phi^\epsilon}_t)\|_{L_2(U,H_1)}^2dt\Big\}
\nonumber\\
\leq\!\!\!\!\!\!\!\!&&C_{N,M,T}\sqrt{\epsilon}\mathbb{E}\Big(\int_0^{T\wedge\widetilde{\tau}_N^\epsilon}\|\widetilde{Z}^{\epsilon}_t\|_{H_1}^2\|G_1(X^{\epsilon,\alpha,\phi^\epsilon}_t)\|_{L_2(U,H_1)}^2dt\Big)^{\frac{1}{2}}
\nonumber\\
\!\!\!\!\!\!\!\!&&+C_{N,M,T}\epsilon\mathbb{E}\int_0^{T\wedge\widetilde{\tau}_N^\epsilon}\|G_1(X^{\epsilon,\alpha,\phi^\epsilon}_t)\|_{L_2(U,H_1)}^2dt
\nonumber\\
\leq\!\!\!\!\!\!\!\!&&C_{N,M,T}\sqrt{\epsilon}\mathbb{E}\Big(\sup_{t\in[0,T\wedge\widetilde{\tau}_N^\epsilon]}\|\widetilde{Z}^{\epsilon}_t\|_{H_1}^2\cdot\int_0^{T\wedge\widetilde{\tau}_N^\epsilon}\|G_1(X^{\epsilon,\alpha,\phi^\epsilon}_t)\|_{L_2(U,H_1)}^2dt\Big)^{\frac{1}{2}}
\nonumber\\
\!\!\!\!\!\!\!\!&&+C_{N,M,T}\epsilon\mathbb{E}\int_0^{T\wedge\widetilde{\tau}_N^\epsilon}\|G_1(X^{\epsilon,\alpha,\phi^\epsilon}_t)\|_{L_2(U,H_1)}^2dt
\nonumber\\
\leq\!\!\!\!\!\!\!\!&&\frac{1}{2}\mathbb{E}\Big[\sup_{t\in[0,T\wedge\widetilde{\tau}_N^\epsilon]}\|\widetilde{Z}^{\epsilon}_t\|_{H_1}^2\Big]+C_{N,M,T}\epsilon\mathbb{E}\int_0^{T\wedge\widetilde{\tau}_N^\epsilon}(1+\|X^{\epsilon,\alpha,\phi^\epsilon}_t\|_{H_1}^2)dt
\nonumber\\
\leq\!\!\!\!\!\!\!\!&&\frac{1}{2}\mathbb{E}\Big[\sup_{t\in[0,T\wedge\widetilde{\tau}_N^\epsilon]}\|\widetilde{Z}^{\epsilon}_t\|_{H_1}^2\Big]+C_{N,M,T}\epsilon.
\end{eqnarray}
Therefore, substituting (\ref{36}) into (\ref{35}) leads to
\begin{eqnarray}\label{37}
\!\!\!\!\!\!\!\!&&\mathbb{E}\Big[\sup_{t\in[0,T\wedge\widetilde{\tau}_N^\epsilon]}\|\widetilde{Z}^{\epsilon}_t\|_{H_1}^2\Big]+2\theta_1\mathbb{E}\int_0^{T\wedge\widetilde{\tau}_N^\epsilon}\|\widetilde{Z}^{\epsilon}_t\|_{V_1}^{\gamma_1}dt
\nonumber\\
~\leq\!\!\!\!\!\!\!\!&&C_{N,M,T}\Big\{\epsilon\sup_{t\in[0,T]}\big(1+\|\bar{X}^{\phi}_t\|_{H_1}^2\big)+(1+\|x\|_{H_1}^2+\|y\|_{H_2}^2)\Big[\Big(\frac{\alpha}{\epsilon}\Big)+\delta^{\frac{1}{4}}\Big]
\nonumber\\
\!\!\!\!\!\!\!\!&&+\mathbb{E}\Big[\sup_{t\in[0,T\wedge\widetilde{\tau}_N^\epsilon]}\Big|\int_0^{t}\langle G_1(\bar{X}^{\phi}_s)(\phi^\epsilon_s-\phi_s),\widetilde{Z}^{\epsilon}_s\rangle_{H_1}ds\Big|\Big]
\nonumber\\
\!\!\!\!\!\!\!\!&&+\mathbb{E}\Big[\sup_{t\in[0,T\wedge\widetilde{\tau}_N^\epsilon]}\Big|\int_0^t\langle F_1(X^{\epsilon,\alpha,\phi^\epsilon}_{s(\delta)},\widehat{Y}^{\epsilon,\alpha}_s)-\bar{F}_1(X^{\epsilon,\alpha,\phi^\epsilon}_{s(\delta)}),\widetilde{Z}^{\epsilon}_{s(\delta)}\rangle_{H_1}ds\Big|\Big]\Big\}.
\end{eqnarray}

\textbf{Step 2}: In this step, we aim to estimate the term $\mathbb{E}\Big[\sup_{t\in[0,T\wedge\widetilde{\tau}_N^\epsilon]}\Big|\int_0^{t}\langle G_1(\bar{X}^{\phi}_s)(\phi^\epsilon_s-\phi_s),\widetilde{Z}^{\epsilon}_s\rangle_{H_1}ds\Big|\Big]$ in (\ref{37}). First it is easy to see that
\begin{equation}\label{38}
\mathbb{E}\Big[\sup_{t\in[0,T\wedge\widetilde{\tau}_N^\epsilon]}\Big|\int_0^{t}\langle G_1(\bar{X}^{\phi}_s)(\phi^\epsilon_s-\phi_s),\widetilde{Z}^{\epsilon}_s\rangle_{H_1}ds\Big|\Big]\leq \sum_{i=1}^3\widetilde{I}_i(N,\epsilon)+\mathbb{E}\big(\widehat{I}_4(N,\epsilon)\big),
\end{equation}
where we denote
\begin{eqnarray*}
\widetilde{I}_1(N,\epsilon):=\!\!\!\!\!\!\!\!&&\mathbb{E}\Big[\sup_{t\in[0,T\wedge\widetilde{\tau}_N^\epsilon]}\Big|\int_0^{t}\langle G_1(\bar{X}^{\phi}_s)(\phi^\epsilon_s-\phi_s),\widetilde{Z}^{\epsilon}_s-\widetilde{Z}^{\epsilon}_{s(\delta)}\rangle_{H_1}ds\Big|\Big],
\nonumber\\
\widetilde{I}_2(N,\epsilon):=\!\!\!\!\!\!\!\!&&\mathbb{E}\Big[\sup_{t\in[0,T\wedge\widetilde{\tau}_N^\epsilon]}\Big|\int_0^{t}\langle \big(G_1(\bar{X}^{\phi}_s)-G_1(\bar{X}^{\phi}_{s(\delta)})\big)(\phi^\epsilon_s-\phi_s),\widetilde{Z}^{\epsilon}_{s(\delta)}\rangle_{H_1}ds\Big|\Big],
\nonumber\\
\widetilde{I}_3(N,\epsilon):=\!\!\!\!\!\!\!\!&&\mathbb{E}\Big[\sup_{t\in[0,T\wedge\widetilde{\tau}_N^\epsilon]}\Big|\int_{t(\delta)}^{t}\langle G_1(\bar{X}^{\phi}_{s(\delta)})(\phi^\epsilon_s-\phi_s),\widetilde{Z}^{\epsilon}_{s(\delta)}\rangle_{H_1}ds\Big|\Big],
\nonumber\\
\widehat{I}_4(N,\epsilon):=\!\!\!\!\!\!\!\!&&\sum_{k=0}^{[(T\wedge\widetilde{\tau}_N^\epsilon)/\delta]-1}\Big|\langle G_1(\bar{X}^{\phi}_{k\delta})\int_{k\delta}^{(k+1)\delta}(\phi^\epsilon_s-\phi_s)ds,\widetilde{Z}^{\epsilon}_{k\delta}\rangle_{H_1}\Big|.
\end{eqnarray*}
From Cauchy-Schwarz's inequality, condition $({\mathbf{A}}{\mathbf{2}})$, Lemma \ref{l7} and Lemma \ref{l10}, it follows that
\begin{eqnarray}\label{39}
\!\!\!\!\!\!\!\!&&\widetilde{I}_1(N,\epsilon)
\nonumber\\
~\leq\!\!\!\!\!\!\!\!&&\mathbb{E}\Big[\int_0^{T\wedge\widetilde{\tau}_N^\epsilon}\|G_1(\bar{X}^{\phi}_t)\|_{L_2(U,H_1)}\|\phi^\epsilon_t-\phi_t\|_U\|\widetilde{Z}^{\epsilon}_t -\widetilde{Z}^{\epsilon}_{t(\delta)}\|_{H_1}dt\Big]
\nonumber\\
~\leq\!\!\!\!\!\!\!\!&&\Big\{\mathbb{E}\Big[\int_0^{T\wedge\widetilde{\tau}_N^\epsilon}\|G_1(\bar{X}^{\phi}_t)\|_{L_2(U,H_1)}^2\|\phi^\epsilon_t-\phi_t\|_U^2dt\Big]\Big\}^{\frac{1}{2}}
\Big\{\mathbb{E}\Big[\int_0^{T\wedge\widetilde{\tau}_N^\epsilon}\|\widetilde{Z}^{\epsilon}_t -\widetilde{Z}^{\epsilon}_{t(\delta)}\|_{H_1}^2dt\Big]\Big\}^{\frac{1}{2}}
\nonumber\\
~\leq\!\!\!\!\!\!\!\!&&\Big\{\mathbb{E}\Big[\int_0^{T\wedge\widetilde{\tau}_N^\epsilon}C(1+\|\bar{X}^{\phi}_t\|_{H_1}^2)\|\phi^\epsilon_t-\phi_t\|_U^2dt\Big]\Big\}^{\frac{1}{2}}
\nonumber\\
\!\!\!\!\!\!\!\!&&\cdot
\Big\{\mathbb{E}\Big[\int_0^{T\wedge\widetilde{\tau}_N^\epsilon}2(\|\bar{X}^{\phi}_{t}-\bar{X}^{\phi}_{t(\delta)}\|_{H_1}^2+\|X^{\epsilon,\alpha,\phi^\epsilon}_{t}-X^{\epsilon,\alpha,\phi^\epsilon}_{t(\delta)}\|_{H_1}^2)dt\Big]\Big\}^{\frac{1}{2}}
\nonumber\\
~\leq\!\!\!\!\!\!\!\!&&C_N\delta^{\frac{1}{4}}(1+\|x\|_{H_1}+\|y\|_{H_2})\Big\{\mathbb{E}\Big[\int_0^{T}\|\phi^\epsilon_t-\phi_t\|_U^2dt\Big]\Big\}^{\frac{1}{2}}
\nonumber\\
~\leq\!\!\!\!\!\!\!\!&&C_{N,M}\delta^{\frac{1}{4}}(1+\|x\|_{H_1}+\|y\|_{H_2}).
\end{eqnarray}
By Lemma \ref{l10}, the second term can be controlled as follows
\begin{eqnarray}\label{40}
\!\!\!\!\!\!\!\!&&\widetilde{I}_2(N,\epsilon)
\nonumber\\
~\leq\!\!\!\!\!\!\!\!&&\mathbb{E}\Big[\int_0^{T\wedge\widetilde{\tau}_N^\epsilon} \|G_1(\bar{X}^{\phi}_t)-G_1(\bar{X}^{\phi}_{t(\delta)})\|_{L_2(U,H_1)}\|\phi^\epsilon_t-\phi_t\|_U\|\widetilde{Z}^{\epsilon}_{t(\delta)}\|_{H_1}dt\Big]
\nonumber\\
~\leq\!\!\!\!\!\!\!\!&&\Big\{\mathbb{E}\Big[\int_0^{T\wedge\widetilde{\tau}_N^\epsilon}\|\bar{X}^{\phi}_t-\bar{X}^{\phi}_{t(\delta)}\|_{H_1}^2\|X^{\epsilon,\alpha,\phi^\epsilon}_{t(\delta)}-\bar{X}^{\phi}_{t(\delta)}\|_{H_1}^2dt\Big]\Big\}^{\frac{1}{2}}\Big\{\mathbb{E}\Big[\int_0^{T}\|\phi^\epsilon_t-\phi_t\|_U^2dt\Big]\Big\}^{\frac{1}{2}}
\nonumber\\
~\leq\!\!\!\!\!\!\!\!&&C_{N,M}\delta^{\frac{1}{4}}(1+\|x\|_{H_1}+\|y\|_{H_2}).
\end{eqnarray}
Using H\"{o}lder's inequality twice, we obtain
\begin{eqnarray}\label{41}
\!\!\!\!\!\!\!\!&&\widetilde{I}_3(N,\epsilon)
\nonumber\\
~\leq\!\!\!\!\!\!\!\!&&\Big\{\mathbb{E}\sup_{t\in[0,T\wedge\widetilde{\tau}_N^\epsilon]}\Big|\int_{t(\delta)}^t\|G_1(\bar{X}^{\phi}_{s(\delta)})\|_{L_2(U,H_1)}\|\phi^\epsilon_s-\phi_s\|_Uds\Big|^2\Big\}^{\frac{1}{2}}
\nonumber\\
\!\!\!\!\!\!\!\!&&\cdot\Big\{\mathbb{E}\Big[\sup_{t\in[0,T\wedge\widetilde{\tau}_N^\epsilon]}\|X^{\epsilon,\alpha,\phi^\epsilon}_{t}-\bar{X}^{\phi}_{t}\|_{H_1}^2\Big]\Big\}^{\frac{1}{2}}
\nonumber\\
~\leq\!\!\!\!\!\!\!\!&&\delta^{\frac{1}{2}}\Big\{\mathbb{E}\int_{0}^{T\wedge\widetilde{\tau}_N^\epsilon}(1+\|\bar{X}^{\phi}_{s(\delta)}\|_{H_1}^2)\|\phi^\epsilon_s-\phi_s\|_U^2ds\Big\}^{\frac{1}{2}}
\Big\{\mathbb{E}\Big[\sup_{t\in[0,T\wedge\widetilde{\tau}_N^\epsilon]}\|X^{\epsilon,\alpha,\phi^\epsilon}_{t}-\bar{X}^{\phi}_{t}\|_{H_1}^2\Big]\Big\}^{\frac{1}{2}}
\nonumber\\
~\leq\!\!\!\!\!\!\!\!&&C_N\delta^{\frac{1}{2}}\Big\{\mathbb{E}\Big[\int_0^{T}\|\phi^\epsilon_t-\phi_t\|_U^2dt\Big]\Big\}^{\frac{1}{2}}
\nonumber\\
~\leq\!\!\!\!\!\!\!\!&&C_{N,M}\delta^{\frac{1}{2}}.
\end{eqnarray}

Now let us consider the convergence of  the last term $\widehat{I}_4(N,\epsilon)$. Since $\mathcal{A}_M$ is a Polish space and $\{\phi^\epsilon:\epsilon>0\}\subset\mathcal{A}_M$ converges to $\phi$ in distribution
as $S_M$-valued random elements, we are able to use the Skorokhod representation theorem to construct a probability space $\left(\widetilde{\Omega},\widetilde{\mathscr{F}},\widetilde{\mathscr{F}}_{t\geq0},\widetilde{\mathbb{P}}\right)$ and processes $(\widetilde{\phi}^\epsilon,\widetilde{\phi},\widetilde{W}^\epsilon)$ such that the joint distribution of $(\widetilde{\phi}^\epsilon,\widetilde{W}^\epsilon)$ is the same as $(\phi^\epsilon,W^\epsilon)$ and $\widetilde{\phi}^\epsilon\to\widetilde{\phi}$, $\widetilde{\mathbb{P}}$-a.s., in the weak topology of $S_M$, where $W^\epsilon$ is defined in (\ref{wiener}).
Therefore, for each $a,b\in[0,T]$,~$a<b$, the integral $\int_a^b\widetilde{\phi}_s^\epsilon ds\to\int_a^b\widetilde{\phi}_sds$ weakly in $U$. Without loss of generality, we will use the notations $\left(\Omega,\mathscr{F},\mathscr{F}_{t\geq0},\mathbb{P}\right)$ and $(\phi^\epsilon,\phi,W)$ replacing $\left(\widetilde{\Omega},\widetilde{\mathscr{F}},\widetilde{\mathscr{F}}_{t\geq0},\widetilde{\mathbb{P}}\right)$ and $(\widetilde{\phi}^\epsilon,\widetilde{\phi},\widetilde{W}^\epsilon)$, respectively.

Since $G_1(\bar{X}^{\phi}_{k\delta})$ is a Hilbert-Schmidt operator hence is compact operator, we infer that
$$\Big\|G_1(\bar{X}^{\phi}_{k\delta})\Big(\int_{k\delta}^{(k+1)\delta}\phi^\epsilon_sds-\int_{k\delta}^{(k+1)\delta}\phi_sds\Big)\Big\|_{H_1}\to 0,~~\text{as}~\epsilon\to 0,$$
which implies that  $\widehat{I}_4(N,\epsilon,\omega)\to 0$, $\mathbb{P}\text{-a.s.}$ as $\epsilon\to 0$. Furthermore, it is easy to see that for any fixed $N>0$, $\widehat{I}_4(N,\epsilon,\omega)\leq C_{N,M}$ by the similar arguments as in (\ref{40}), then the dominated convergence theorem yields that for any $N>0$,
\begin{eqnarray}\label{42}
\mathbb{E}\big(\widehat{I}_4(N,\epsilon)\big)\to 0,~\text{as}~\epsilon\to 0.
\end{eqnarray}
Finally, taking (\ref{39})-(\ref{42}) into (\ref{38}) account we conclude that for any fixed $N>0$,
\begin{eqnarray}\label{43}
\limsup_{\delta\to 0}\limsup_{\epsilon\to 0}\mathbb{E}\Big[\sup_{t\in[0,T\wedge\widetilde{\tau}_N^\epsilon]}\Big|\int_0^{t}\langle G_1(\bar{X}^{\phi}_s)(\phi^\epsilon_s-\phi_s),\widetilde{Z}^{\epsilon}_s\rangle_{H_1}ds\Big|\Big]=0.
\end{eqnarray}

\textbf{Step 3}: This step is devoted to investigating the term $\mathbb{E}\Big[\sup_{t\in[0,T\wedge\widetilde{\tau}_N^\epsilon]}\Big|\int_0^t\langle F_1(X^{\epsilon,\alpha,\phi^\epsilon}_{s(\delta)},\widehat{Y}^{\epsilon,\alpha}_s)-\bar{F}_1(X^{\epsilon,\alpha,\phi^\epsilon}_{s(\delta)}),\widetilde{Z}^{\epsilon}_{s(\delta)}\rangle_{H_1}ds\Big|\Big]\Big\}$ in (\ref{37}).
It is obvious that
\begin{eqnarray}\label{44}
\!\!\!\!\!\!\!\!&&\Big|\int_0^t\langle F_1(X^{\epsilon,\alpha,\phi^\epsilon}_{s(\delta)},\widehat{Y}^{\epsilon,\alpha}_s)-\bar{F}_1(X^{\epsilon,\alpha,\phi^\epsilon}_{s(\delta)}),\widetilde{Z}^{\epsilon}_{s(\delta)}\rangle_{H_1}ds\Big|
\nonumber\\
\leq\!\!\!\!\!\!\!\!&&\sum_{k=0}^{[t/\delta]-1}\Big|\int_{k\delta}^{(k+1)\delta}\langle F_1(X^{\epsilon,\alpha,\phi^\epsilon}_{s(\delta)},\widehat{Y}^{\epsilon,\alpha}_s)-\bar{F}_1(X^{\epsilon,\alpha,\phi^\epsilon}_{s(\delta)}),\widetilde{Z}^{\epsilon}_{s(\delta)}\rangle_{H_1}ds\Big|
\nonumber\\
\!\!\!\!\!\!\!\!&&+\Big|\int_{t(\delta)}^{t}\langle F_1(X^{\epsilon,\alpha,\phi^\epsilon}_{s(\delta)},\widehat{Y}^{\epsilon,\alpha}_s)-\bar{F}_1(X^{\epsilon,\alpha,\phi^\epsilon}_{s(\delta)}),\widetilde{Z}^{\epsilon}_{s(\delta)}\rangle_{H_1}ds\Big|
\nonumber\\
=:\!\!\!\!\!\!\!\!&&J_1(t)+J_2(t).
\end{eqnarray}
According to the condition $({\mathbf{A}}{\mathbf{2}})$, the Lipschitz continuity of $\bar{F}_1$ and the definition of $\widetilde{\tau}_N^\epsilon$, the term $J_2(t)$ can be controlled by
\begin{eqnarray}\label{45}
\!\!\!\!\!\!\!\!&&\mathbb{E}\Big[\sup_{t\in[0,T\wedge\widetilde{\tau}_N^\epsilon]}J_2(t)\Big]
\nonumber\\
\leq\!\!\!\!\!\!\!\!&&\Big[\mathbb{E}\Big(\sup_{t\in[0,T\wedge\widetilde{\tau}_N^\epsilon]}\|X^{\epsilon,\alpha,\phi^\epsilon}_t-\bar{X}^{\phi}_t\|_{H_1}^2\Big)\Big]^{\frac{1}{2}}
\nonumber\\
\!\!\!\!\!\!\!\!&&\cdot
\Big[\mathbb{E}\Big(\sup_{t\in[0,T\wedge\widetilde{\tau}_N^\epsilon]}\Big|\int_{t(\delta)}^{t}(1+\|X^{\epsilon,\alpha,\phi^\epsilon}_{s(\delta)}\|_{H_1}+\|\widehat{Y}^{\epsilon,\alpha}_s\|_{H_2})ds\Big|^2\Big)\Big]^{\frac{1}{2}}
\nonumber\\
\leq\!\!\!\!\!\!\!\!&&\Big\{\mathbb{E}\Big(\sup_{t\in[0,T\wedge\widetilde{\tau}_N^\epsilon]}\|X^{\epsilon,\alpha,\phi^\epsilon}_t-\bar{X}^{\phi}_t\|_{H_1}^2\Big)\Big\}^{\frac{1}{2}}
\nonumber\\
\!\!\!\!\!\!\!\!&&\cdot
\Big\{\mathbb{E}\Big[\int_{0}^{T\wedge\widetilde{\tau}_N^\epsilon}(1+\|X^{\epsilon,\alpha,\phi^\epsilon}_{t(\delta)}\|_{H_1}^2)dt\Big]+\sup_{t\in[0,T]}\mathbb{E}\|\widehat{Y}^{\epsilon,\alpha}_t\|_{H_2}^2\Big\}^{\frac{1}{2}}\delta^{\frac{1}{2}}
\nonumber\\
\leq\!\!\!\!\!\!\!\!&&C_{N,T}(1+\|x\|_{H_1}^2+\|y\|_{H_2}^2)\delta^{\frac{1}{2}},~~
\end{eqnarray}
where we used Lemma \ref{l8} in the last step.

The term $J_1(t)$ will be controlled as follows,
\begin{eqnarray*}
\!\!\!\!\!\!\!\!&&\mathbb{E}\Big[\sup_{t\in[0,T\wedge\widetilde{\tau}_N^\epsilon]}J_1(t)\Big]
\nonumber\\
\leq\!\!\!\!\!\!\!\!&&\mathbb{E}\sum_{k=0}^{[T\wedge\widetilde{\tau}_N^\epsilon/\delta]-1}\Big|\int_{k\delta}^{(k+1)\delta}\langle F_1(X^{\epsilon,\alpha,\phi^\epsilon}_{k\delta},\widehat{Y}^{\epsilon,\alpha}_s)-\bar{F}_1(X^{\epsilon,\alpha,\phi^\epsilon}_{k\delta}),\widetilde{Z}^{\epsilon}_{k\delta}\rangle_{H_1}ds\Big|
\nonumber\\
\leq\!\!\!\!\!\!\!\!&&\frac{C_T}{\delta}\sup_{0\leq k\leq[T\wedge\widetilde{\tau}_N^\epsilon/\delta]-1}\mathbb{E}\Big|\int_{k\delta}^{(k+1)\delta}\langle F_1(X^{\epsilon,\alpha,\phi^\epsilon}_{k\delta},\widehat{Y}^{\epsilon,\alpha}_s)-\bar{F}_1(X^{\epsilon,\alpha,\phi^\epsilon}_{k\delta}),\widetilde{Z}^{\epsilon}_{k\delta}\rangle_{H_1}ds\Big|
\nonumber\\
\leq\!\!\!\!\!\!\!\!&&\frac{C_T\alpha}{\delta}\sup_{0\leq k\leq[T\wedge\widetilde{\tau}_N^\epsilon/\delta]-1}\Big(\mathbb{E}\|X^{\epsilon,\alpha,\phi^\epsilon}_{k\delta}-\bar{X}^{\phi}_{k\delta}\|_{H_1}^2\Big)^{\frac{1}{2}}
\Big(\mathbb{E}\Big\|\int_0^{\frac{\delta}{\alpha}}F_1(X^{\epsilon,\alpha,\phi^\epsilon}_{k\delta},\widehat{Y}^{\epsilon,\alpha}_{s\alpha+k\delta})-\bar{F}_1(X^{\epsilon,\alpha,\phi^\epsilon}_{k\delta})ds\Big\|_{H_1}^2\Big)^{\frac{1}{2}}
\nonumber\\
\leq\!\!\!\!\!\!\!\!&&\frac{C_{N,T}\alpha}{\delta}\sup_{0\leq k\leq[T\wedge\widetilde{\tau}_N^\epsilon/\delta]-1}\Big(\int_0^{\frac{\delta}{\alpha}}\int_{r}^{\frac{\delta}{\alpha}}\Psi_k(s,r)dsdr\Big)^{\frac{1}{2}},
\end{eqnarray*}
where for each $0\leq r\leq s\leq \frac{\delta}{\alpha}$,
$$\Psi_k(s,r):=\mathbb{E}\Big[\langle F_1(X^{\epsilon,\alpha,\phi^\epsilon}_{k\delta},\widehat{Y}^{\epsilon,\alpha}_{s\alpha+k\delta})-\bar{F}_1(X^{\epsilon,\alpha,\phi^\epsilon}_{k\delta}),
F_1(X^{\epsilon,\alpha,\phi^\epsilon}_{k\delta},\widehat{Y}^{\epsilon,\alpha}_{r\alpha+k\delta})-\bar{F}_1(X^{\epsilon,\alpha,\phi^\epsilon}_{k\delta})\rangle_{H_1}\Big].$$

Now we devote to estimating the term $\Psi_k(s,r)$. For each $s>0$ and $\mathscr{F}_s$-measurable $H_1$-valued random variable $X$ and $H_2$-valued random variable $Y$, let $\{\widetilde{Y}_t^{\alpha,s,X,Y}\}_{t\geq 0}$ be a unique solution of the following SPDE
\begin{eqnarray*}
\left\{ \begin{aligned}
&dY_t=\frac{1}{\alpha}F_2(X,Y_t)dt+\frac{1}{\sqrt{\alpha}}G_2dW_t,~t\geq s,\\
&Y_s=Y.
\end{aligned}\right.
\end{eqnarray*}
According to the definition of process $\widehat{Y}_t^{\epsilon,\alpha}$, for each $k\in\mathbb{N}$ and $t\in[k\delta,(k+1)\delta]$, we can get that
$$\widehat{Y}_t^{\epsilon,\alpha}=\widetilde{Y}_t^{\alpha,k\delta,X^{\epsilon,\alpha,\phi^\epsilon}_{k\delta},\widehat{Y}_{k\delta}^{\epsilon,\alpha}},~\mathbb{P}\text{-a.s.},$$
which yields the following identity
\begin{eqnarray*}
\!\!\!\!\!\!\!\!&&\Psi_k(s,r)
\nonumber\\
=\!\!\!\!\!\!\!\!&&\mathbb{E}\Big[\langle F_1(X^{\epsilon,\alpha,\phi^\epsilon}_{k\delta},\widetilde{Y}_{s\alpha+k\delta}^{\alpha,k\delta,X^{\epsilon,\alpha,\phi^\epsilon}_{k\delta},\widehat{Y}_{k\delta}^{\epsilon,\alpha}})-\bar{F}_1(X^{\epsilon,\alpha,\phi^\epsilon}_{k\delta}),
F_1(X^{\epsilon,\alpha,\phi^\epsilon}_{k\delta},\widetilde{Y}_{r\alpha+k\delta}^{\alpha,k\delta,X^{\epsilon,\alpha,\phi^\epsilon}_{k\delta},\widehat{Y}_{k\delta}^{\epsilon,\alpha}})-\bar{F}_1(X^{\epsilon,\alpha,\phi^\epsilon}_{k\delta})\rangle_{H_1}\Big]
\nonumber\\
=\!\!\!\!\!\!\!\!&&\mathbb{E}\Big\{\mathbb{E}\Big[\langle F_1(X^{\epsilon,\alpha,\phi^\epsilon}_{k\delta},\widetilde{Y}_{s\alpha+k\delta}^{\alpha,k\delta,X^{\epsilon,\alpha,\phi^\epsilon}_{k\delta},\widehat{Y}_{k\delta}^{\epsilon,\alpha}})-\bar{F}_1(X^{\epsilon,\alpha,\phi^\epsilon}_{k\delta}),
\nonumber\\
\!\!\!\!\!\!\!\!&&~~~~~~~~~F_1(X^{\epsilon,\alpha,\phi^\epsilon}_{k\delta},\widetilde{Y}_{r\alpha+k\delta}^{\alpha,k\delta,X^{\epsilon,\alpha,\phi^\epsilon}_{k\delta},\widehat{Y}_{k\delta}^{\epsilon,\alpha}})-\bar{F}_1(X^{\epsilon,\alpha,\phi^\epsilon}_{k\delta})\rangle_{H_1}\big|\mathscr{F}_{k\delta}\Big]\Big\}
\nonumber\\
=\!\!\!\!\!\!\!\!&&\mathbb{E}\Big\{\mathbb{E}\Big[\langle F_1(x,\widetilde{Y}_{s\alpha+k\delta}^{\alpha,k\delta,x,y})-\bar{F}_1(x),
F_1(x,\widetilde{Y}_{r\alpha+k\delta}^{\alpha,k\delta,x,y})-\bar{F}_1(x)\rangle_{H_1}\Big]\Big|_{(x,y)=(X^{\epsilon,\alpha,\phi^\epsilon}_{k\delta},\widehat{Y}_{k\delta}^{\epsilon,\alpha})}\Big\},
\end{eqnarray*}
where the last step follows the fact that $X^{\epsilon,\alpha,\phi^\epsilon}_{k\delta}$ and $\widehat{Y}_{k\delta}^{\epsilon,\alpha}$ are $\mathscr{F}_{k\delta}$-measurable, and $\{\widetilde{Y}_{s\alpha+k\delta}^{\alpha,k\delta,x,y}\}_{s\geq0}$ is independent of $\mathscr{F}_{k\delta}$ for each fixed $x\in H_1$ and $y\in H_2$.

By the construction of $\widetilde{Y}_{s\alpha+k\delta}^{\alpha,k\delta,x,y}$, for each $k\in\mathbb{N}$,
\begin{eqnarray}\label{63}
\widetilde{Y}_{s\alpha+k\delta}^{\alpha,k\delta,x,y}\!\!\!\!\!\!\!\!&&=y+\frac{1}{\alpha}\int_{k\delta}^{s\alpha+k\delta}F_2(x,\widetilde{Y}_{r}^{\alpha,k\delta,x,y})dr+
\frac{1}{\sqrt{\alpha}}\int_{k\delta}^{s\alpha+k\delta}G_2dW_r
\nonumber\\
=\!\!\!\!\!\!\!\!&&y+\frac{1}{\alpha}\int_{0}^{s\alpha}F_2(x,\widetilde{Y}_{r+k\delta}^{\alpha,k\delta,x,y})dr+
\frac{1}{\sqrt{\alpha}}\int_{0}^{s\alpha}G_2dW_r^{k\delta}
\nonumber\\
=\!\!\!\!\!\!\!\!&&y+\int_{0}^{s}F_2(x,\widetilde{Y}_{r\alpha+k\delta}^{\alpha,k\delta,x,y})dr+
\int_{0}^{s}G_2d\bar{W}_r^{k\delta},
\end{eqnarray}
here we denote $W_r^{k\delta}:=W_{r+k\delta}-W_{k\delta}$ that is a shift version of $W_r$, furthermore, $\bar{W}_r^{k\delta}:=\frac{1}{\sqrt{\alpha}}W_{r\alpha}^{k\delta}$.

It is easy to see that the uniqueness of solutions to Eq.~(\ref{63}) and Eq.~(\ref{e3}) gives that the distribution of random sequence $\{\widetilde{Y}_{s\alpha+k\delta}^{\alpha,k\delta,x,y}\}_{0\leq s\leq\frac{\delta}{\alpha}}$ coincides with the distribution of $\{Y_s^{x,y}\}_{0\leq s\leq\frac{\delta}{\alpha}}$. Thus using (\ref{er}), (\ref{14}) and (\ref{24}), in terms of the Markov and time-homogenous properties of process $Y_s^{x,y}$, we obtain
\begin{eqnarray*}
\!\!\!\!\!\!\!\!&&\Psi_k(s,r)
\nonumber\\
=\!\!\!\!\!\!\!\!&&\mathbb{E}\Big\{\mathbb{E}\Big[\langle F_1(x,Y_s^{x,y})-\bar{F}_1(x),
F_1(x,Y_r^{x,y})-\bar{F}_1(x)\rangle_{H_1}\Big]\Big|_{(x,y)=(X^{\epsilon,\alpha,\phi^\epsilon}_{k\delta},\widehat{Y}_{k\delta}^{\epsilon,\alpha})}\Big\}
\nonumber\\
=\!\!\!\!\!\!\!\!&&\mathbb{E}\Big\{\mathbb{E}\Big\langle \mathbb{E}\Big[ F_1(x,Y_{s-r}^{x,z})-\bar{F}_1(x)\Big]\Big|_{\{z=Y_r^{x,y}\}},
F_1(x,Y_r^{x,y})-\bar{F}_1(x)\Big\rangle_{H_1}\Big\}
\nonumber\\
\leq\!\!\!\!\!\!\!\!&&\mathbb{E}\Big\{\mathbb{E}\Big\{\Big[1+\|x\|_{H_1}+\|Y_r^{x,y}\|_{H_2}\Big]e^{-\frac{(s-r)\kappa}{2}}
\cdot\Big[1+\|x\|_{H_1}+\|Y_r^{x,y}\|_{H_2}\Big]\Big\}\Big|_{(x,y)=(X^{\epsilon,\alpha,\phi^\epsilon}_{k\delta},\widehat{Y}_{k\delta}^{\epsilon,\alpha})}\Big\}
\nonumber\\
\leq\!\!\!\!\!\!\!\!&&C_T\mathbb{E}\Big[1+\|X^{\epsilon,\alpha,\phi^\epsilon}_{k\delta}\|_{H_1}^2+\|\widehat{Y}_{k\delta}^{\epsilon,\alpha}\|_{H_2}^2
\Big]e^{-\frac{(s-r)\kappa}{2}}
\nonumber\\
\leq\!\!\!\!\!\!\!\!&&C_T(1+\|x\|_{H_1}^2+\|y\|_{H_2}^2)e^{-\frac{(s-r)\kappa}{2}}.
\end{eqnarray*}
where $\kappa>0$ is defined in (\ref{er}), which leads to
\begin{eqnarray}\label{46}
\!\!\!\!\!\!\!\!&&\mathbb{E}\Big[\sup_{t\in[0,T\wedge\widetilde{\tau}_N^\epsilon]}J_1(t)\Big]
\nonumber\\
\leq\!\!\!\!\!\!\!\!&&C_{N,T}(1+\|x\|_{H_1}^2+\|y\|_{H_2}^2)\frac{\alpha}{\delta}\Big(\int_0^{\frac{\delta}{\alpha}}\int_{r}^{\frac{\delta}{\alpha}}e^{-\frac{(s-r)\kappa}{2}}dsdr\Big)^{\frac{1}{2}}
\nonumber\\
\leq\!\!\!\!\!\!\!\!&&C_{N,T}(1+\|x\|_{H_1}^2+\|y\|_{H_2}^2)\frac{\alpha}{\delta}\Big(\frac{\delta}{\alpha\kappa}-\frac{1}{\kappa^2}+\frac{1}{\kappa^2}e^{-\frac{\kappa\delta}{2\alpha}}\Big)^{\frac{1}{2}}.
\end{eqnarray}
Combining (\ref{45}), (\ref{46}) with (\ref{44}) implies that
\begin{eqnarray}\label{47}
\!\!\!\!\!\!\!\!&&\mathbb{E}\Big[\sup_{t\in[0,T\wedge\widetilde{\tau}_N^\epsilon]}\Big|\int_0^t\langle F_1(X^{\epsilon,\alpha,\phi^\epsilon}_{s(\delta)},\widehat{Y}^{\epsilon,\alpha}_s)-\bar{F}_1(X^{\epsilon,\alpha,\phi^\epsilon}_{s(\delta)}),\widetilde{Z}^{\epsilon}_{s(\delta)}\rangle_{H_1}ds\Big|\Big]
\nonumber\\
\leq\!\!\!\!\!\!\!\!&&C_{N,T}(1+\|x\|_{H_1}^2+\|y\|_{H_2}^2)\Big(\frac{\alpha}{\delta}+\frac{\alpha^{\frac{1}{2}}}{\delta^{\frac{1}{2}}}+\delta^{\frac{1}{2}}\Big).
\end{eqnarray}

\textbf{Step 4}: After all  preparations above, we are in the position to derive the desired results on the convergence of $\widetilde{Z}^{\epsilon}_t$ in distribution.

For any $\varepsilon_0>0$, using Chebyshev's inequality we obtain that
\begin{eqnarray}\label{48}
\!\!\!\!\!\!\!\!&&\mathbb{P}\Big\{\Big(\sup_{t\in[0,T]}\|\widetilde{Z}^{\epsilon}_t\|_{H_1}^2+\theta_1\int_0^T\|\widetilde{Z}^{\epsilon}_t\|_{V_1}^{\gamma_1}dt\Big)^{\frac{1}{2}}>\varepsilon_0\Big\}
\nonumber\\
\leq\!\!\!\!\!\!\!\!&&\frac{1}{\varepsilon_0}\mathbb{E}\Big[\Big(\sup_{t\in[0,T]}\|\widetilde{Z}^{\epsilon}_t\|_{H_1}^2+\theta_1\int_0^T\|\widetilde{Z}^{\epsilon}_t\|_{V_1}^{\gamma_1}dt\Big)^{\frac{1}{2}}\mathbf{1}_{\{T\leq\widetilde{\tau}_N^\epsilon\}}\Big]
\nonumber\\
\!\!\!\!\!\!\!\!&&+\frac{1}{\varepsilon_0}\mathbb{E}\Big[\Big(\sup_{t\in[0,T]}\|\widetilde{Z}^{\epsilon}_t\|_{H_1}^2+\theta_1\int_0^T\|\widetilde{Z}^{\epsilon}_t\|_{V_1}^{\gamma_1}dt\Big)^{\frac{1}{2}}\mathbf{1}_{\{T>\widetilde{\tau}_N^\epsilon\}}\Big].
\end{eqnarray}
We now focus on the second term of right hand side of (\ref{48}), by using Markov's inequality and H\"{o}lder's inequality,
\begin{eqnarray}\label{49}
\!\!\!\!\!\!\!\!&&\frac{1}{\varepsilon_0}\mathbb{E}\Big[\Big(\sup_{t\in[0,T]}\|\widetilde{Z}^{\epsilon}_t\|_{H_1}^2+\theta_1\int_0^T\|\widetilde{Z}^{\epsilon}_t\|_{V_1}^{\gamma_1}dt\Big)^{\frac{1}{2}}\mathbf{1}_{\{T>\widetilde{\tau}_N^\epsilon\}}\Big]
\nonumber\\
\leq\!\!\!\!\!\!\!\!&&\frac{1}{\varepsilon_0}\Big[\mathbb{E}\Big(\sup_{t\in[0,T]}\|\widetilde{Z}^{\epsilon}_t\|_{H_1}^2+\theta_1\int_0^T\|\widetilde{Z}^{\epsilon}_t\|_{V_1}^{\gamma_1}dt\Big)\Big]^{\frac{1}{2}}\Big[\mathbb{P}(T>\widetilde{\tau}_N^\epsilon)\Big]^{\frac{1}{2}}
\nonumber\\
\leq\!\!\!\!\!\!\!\!&&\frac{C(1+\|x\|_{H_1}^2+\|y\|_{H_2}^2)}{\varepsilon_0\sqrt{N}},
\end{eqnarray}
where we used Lemma \ref{l5} and Lemma \ref{l6} in the last step.

Finally, letting $\delta:=\alpha^{\frac{1}{2}}$ and collecting (\ref{37}), (\ref{43}) and (\ref{47}), by the condition (\ref{h5}), we can get that
$$\limsup_{\epsilon\to 0}\mathbb{P}\Big\{\sup_{t\in[0,T]}\|\widetilde{Z}^{\epsilon}_t\|_{H_1}^2+\theta_1\int_0^T\|\widetilde{Z}^{\epsilon}_t\|_{V_1}^{\gamma_1}dt>\varepsilon_0\Big\}\leq \frac{C(1+\|x\|_{H_1}^2+\|y\|_{H_2}^2)}{\sqrt{N}},$$
where the constant $C$ does not depend on $N$, which implies the desired assertion by taking $N\to \infty$. We complete the verification of \textbf{Condition (A)} (i).    \hspace{\fill}$\Box$
\end{proof}

\subsection{Compactness}
This subsection is devoted to proving the compactness result, which implies that the rate function $I$ defined in (\ref{rf}) is a good rate function. After that, combining with Theorem \ref{t4}, we prove that  $\{X^{\epsilon,\alpha}\}$ satisfies the Laplace principle (Theorem \ref{t1}), which is equivalent to the LDP on $C([0,T]; H_1)\cap L^{\gamma_1}([0,T]; V_1)$.

\begin{theorem}\label{t5}
Assume that the conditions in Theorem \ref{t1} hold. For fixed $M>0$, $x\in H_1$ and $y\in H_2$, let $K_M=\{\bar{X}^{\phi}:\phi\in S_M\}$, here $\bar{X}^{\phi}$ is a unique solution to the skeleton equation (\ref{e2}). Then $K_M$ is a compact set of $C([0,T]; H_1)\cap L^{\gamma_1}([0,T]; V_1)$.
\end{theorem}

\begin{proof}
Take any sequence $\{\bar{X}^{\phi^n}\}$ in $K_M$, which is the solution of Eq.~(\ref{e2}) with $\phi^n\in S_M$ instead of $\phi$, i.e.,
\begin{equation*}
\left\{ \begin{aligned}
&\frac{d\bar{X}^{\phi^n}_t}{dt}=\big[A(\bar{X}^{\phi^n}_t)+\bar{F}_1(\bar{X}^{\phi^n}_t)\big]+G_1(\bar{X}^{\phi^n}_t)\phi^n_t,\\
&\bar{X}^{\phi^n}_0=x\in H_1.
\end{aligned} \right.
\end{equation*}
Note that $S_M$ is a bounded closed subset in $L^2([0,T]; U)$, hence it is weakly compact, so there exists a subsequence
also denoted by $\phi^n$, which weakly converges to a
limit $\phi\in S_M$ in $L^2([0,T]; U)$. Then the priori estimates formulated in Lemma \ref{l5} imply that
\begin{eqnarray*}
\bar{X}^{\phi^n}\to\!\!\!\!\!\!\!\!&&\bar{X}^{\phi}~\text{weakly star in}~L^{\infty}([0,T];H_1),
\nonumber\\
\bar{X}^{\phi^n}\to\!\!\!\!\!\!\!\!&&\bar{X}^{\phi}~\text{weakly in}~L^{\gamma_1}([0,T];V_1).
\end{eqnarray*}
It is easy to show that $\bar{X}^{\phi}$ is the unique solution of the following limit equation
$$\frac{d\bar{X}^{\phi}_t}{dt}=\big[A(\bar{X}^{\phi}_t)+\bar{F}_1(\bar{X}^{\phi}_t)\big]+G_1(\bar{X}^{\phi}_t)\phi_t,~\bar{X}^{\phi}_0=x.$$

In order to study the compactness of set $K_M$, it suffices to prove that $\bar{X}^{\phi^n}$ strong converges to $\bar{X}^{\phi}$ in $C([0,T]; H_1)\cap L^{\gamma_1}([0,T]; V_1)$ as $n\to \infty$. Denote $\widehat{Z}^{n}_t:=\bar{X}^{\phi^n}_t-\bar{X}^{\phi}_t$ fulfilling
\begin{equation*}
\left\{ \begin{aligned}
&\frac{d\widehat{Z}^{n}_t}{dt}=\big[A(\bar{X}^{\phi^n}_t)-A(\bar{X}^{\phi}_t)+\bar{F}_1(\bar{X}^{\phi^n}_t)-\bar{F}_1(\bar{X}^{\phi}_t)\big]dt+\big[G_1(\bar{X}^{\phi^n}_t)\phi^n_t-G_1(\bar{X}^{\phi}_t)\phi_t\big]dt,\\
&\widehat{Z}^{n}_0=0.
\end{aligned} \right.
\end{equation*}
It is easy to get the following energy estimate by the condition $({\mathbf{A}}{\mathbf{2}})$ and Young's inequality,
\begin{eqnarray*}
\|\widehat{Z}^{n}_t\|_{H_1}^2+\theta_1\int_0^t\|\widehat{Z}^{n}_s\|_{V_1}^{\gamma_1}ds\leq\!\!\!\!\!\!\!\!&&\int_0^t\big(C+\rho(\bar{X}^{\phi}_s)+\|\phi^n_s\|_U^2\big)\|\widehat{Z}^{n}_s\|_{H_1}^2ds
\nonumber\\
\!\!\!\!\!\!\!\!&&+2\int_0^t\langle G_1(\bar{X}^{\phi}_s)(\phi^n_s-\phi_s),\widehat{Z}^{n}_s\rangle_{H_1}ds,
\end{eqnarray*}
where $C>0$ is a constant independent of $n$.

The priori estimate (\ref{2}) implies that there exists a constant $K_0>0$ independent of $n$ such that
\begin{eqnarray}\label{50}
\sup_n\Big[\sup_{t\in[0,T]}(\|\bar{X}^{\phi^n}_t\|_{H_1}^2+\|\bar{X}^{\phi}_t\|_{H_1}^2)+\theta_1\int_0^T(\|\bar{X}^{\phi^n}_t\|_{V_1}^{\gamma_1}+\|\bar{X}^{\phi}_t\|_{V_1}^{\gamma_1})dt\Big]=K_0.
\end{eqnarray}
Therefore, making use of Gronwall's lemma and (\ref{50}), it follows that
\begin{eqnarray}\label{51}
\sup_{t\in[0,T]}\|\widehat{Z}^{n}_t\|_{H_1}^2+\theta_1\int_0^T\|\widehat{Z}^{n}_t\|_{V_1}^{\gamma_1}dt\leq C_{K_0,M,T}\sum_{i=1}^{4}I_{n,i},
\end{eqnarray}
here we denote
\begin{eqnarray*}
I_{n,1}:=\!\!\!\!\!\!\!\!&&\int_0^{T}\Big|\langle G_1(\bar{X}^{\phi}_s)(\phi^n_s-\phi_s),\widehat{Z}^{n}_s-\widehat{Z}^{n}_{s(\delta)}\rangle_{H_1}\Big|ds,
\nonumber\\
I_{n,2}:=\!\!\!\!\!\!\!\!&&\int_0^{T}\Big|\langle \big(G_1(\bar{X}^{\phi}_s)-G_1(\bar{X}^{\phi}_{s(\delta)})\big)(\phi^n_s-\phi_s),\widehat{Z}^{n}_{s(\delta)}\rangle_{H_1}\Big|ds,
\nonumber\\
I_{n,3}:=\!\!\!\!\!\!\!\!&&\sup_{t\in[0,T]}\Big|\int_{t(\delta)}^{t}\langle G_1(\bar{X}^{\phi}_{s(\delta)})(\phi^n_s-\phi_s),\widehat{Z}^{n}_{s(\delta)}\rangle_{H_1}ds\Big|,
\nonumber\\
I_{n,4}:=\!\!\!\!\!\!\!\!&&\sup_{t\in[0,T]}\sum_{k=0}^{[t/\delta]-1}\Big|\langle G_1(\bar{X}^{\phi}_{k\delta})\int_{k\delta}^{(k+1)\delta}(\phi^n_s-\phi_s)ds,\widehat{Z}^{n}_{k\delta}\rangle_{H_1}\Big|.
\end{eqnarray*}
Following the almost same arguments as in the proof of Theorem \ref{t4}, one can obtain
\begin{eqnarray}
I_{n,1}\leq\!\!\!\!\!\!\!\!&&\Big\{\int_0^TC(1+\|\bar{X}^{\phi}_t\|_{H_1}^2)\|\phi^n_t-\phi_t\|_U^2dt\Big\}^{\frac{1}{2}}
\Big\{\int_0^T2(\|\bar{X}^{\phi}_t-\bar{X}^{\phi}_{t(\delta)}\|_{H_1}^2+\|\bar{X}^{\phi^n}_t-\bar{X}^{\phi^n}_{t(\delta)}\|_{H_1}^2)dt\Big\}^{\frac{1}{2}}\nonumber
\nonumber\\
\leq\!\!\!\!\!\!\!\!&&C_{K_0,M}\delta^{\frac{1}{4}}(1+\|x\|_{H_1}),\label{52}
\\
I_{n,2}\leq\!\!\!\!\!\!\!\!&&\Big\{\int_0^T\|\phi^n_t-\phi_t\|_U^2dt\Big\}^{\frac{1}{2}}
\Big\{\int_0^T\|\bar{X}^{\phi}_t-\bar{X}^{\phi}_{t(\delta)}\|_{H_1}^2\|\bar{X}^{\phi^n}_{t(\delta)}-\bar{X}^{\phi}_{t(\delta)}\|_{H_1}^2dt\Big\}^{\frac{1}{2}}
\nonumber\\
\leq\!\!\!\!\!\!\!\!&&C_{K_0,M}\delta^{\frac{1}{4}}(1+\|x\|_{H_1}),\label{53}
\\
I_{n,3}\leq\!\!\!\!\!\!\!\!&&\delta^{\frac{1}{2}}\Big\{\int_{0}^{T}(1+\|\bar{X}^{\phi}_{t(\delta)}\|_{H_1}^2)\|\phi^n_t-\phi_t\|_U^2dt\Big\}^{\frac{1}{2}}
\Big\{\sup_{t\in[0,T]}\|\bar{X}^{\phi^n}_{t}-\bar{X}^{\phi}_{t}\|_{H_1}^2\Big\}^{\frac{1}{2}}
\nonumber\\
\leq\!\!\!\!\!\!\!\!&&C_{K_0,M}\delta^{\frac{1}{2}}.\label{54}
\end{eqnarray}
For the term $I_{n,4}$, since $G_1(\bar{X}^{\phi}_{k\delta})$ is a compact operator, the sequence $G_1(\bar{X}^{\phi}_{k\delta})\int_{k\delta}^{(k+1)\delta}(\phi^n_s-\phi_s)ds$ strongly converges to $0$ in $H_1$ for any fixed $k$, as $n\to\infty$. This combines with the boundedness of $I_{n,4}$ implies that $\lim_{n\to\infty}I_{n,4}=0$.

Furthermore, according to (\ref{52})-(\ref{54}), for any $\delta>0$,
$$\lim_{n\to\infty}\Big\{\sup_{t\in[0,T]}\|\widehat{Z}^{n}_t\|_{H_1}^2+\theta_1\int_0^T\|\widehat{Z}^{n}_t\|_{V_1}^{\gamma_1}dt\Big\}\leq C\delta^{\frac{1}{4}},$$
where the constant $C>0$ is independent of $\delta$.

Taking $\delta\to 0$, one can show that every sequence in $K_M$ has a
convergent subsequence, therefore $K_M$ is a pre-compact subset of $C([0,T];H_1)\cap L^{\gamma_1}([0,T]; V_1)$.

It suffices to prove that $K_M$ is a closed subset of $C([0,T];H_1)\cap L^{\gamma_1}([0,T]; V_1)$.  It should be noted that the above arguments also implies that there exists a subsequence $\{\bar{X}^{\phi^{n_k}},k\geq1\}$ converges to an element $\bar{X}^{\phi}\in K_M$ in the same topology of $C([0,T];H_1)\cap L^{\gamma_1}([0,T];V_1)$, which yields the desired results. Hence complete the verification of \textbf{Condition (A)} (ii).   \hspace{\fill}$\Box$
\end{proof}

Now we are in the position to complete the proof of our main results in this paper.\\
\textbf{Proof of Theorem \ref{t1}}. Following from Theorem \ref{t4} and Theorem \ref{t5}, we can infer that  $\{X^{\epsilon,\alpha}\}$ fulfills the Laplace principle by Lemma 2.4, which is equivalent to the LDP on $C([0,T]; H_1)\cap L^{\gamma_1}([0,T]; V_1)$ with a good rate function $I$ defined in (\ref{rf}). \hspace{\fill}$\Box$\\

\noindent\textbf{Proof of Theorem \ref{t3}}.
 Note that the condition $({\mathbf{A}}{\mathbf{2}})$ is only used to verify the additional convergence in $L^\alpha([0,T]; V_1)$, so if we only concern the LDP on $C([0,T];H_1)$, one can follow the similar arguments as in proof of Theorem \ref{t1} to show Theorem \ref{t3} directly. Since the proof is just a very minor modification of Theorem \ref{t1}, we omit the details here.  \hspace{\fill}$\Box$

\vspace{5mm}
\noindent\textbf{Acknowledgements} {The authors would like to thank anonymous referees for the suggestions and comments, and also thank Xiaobin Sun for helpful discussion. The research
of S. Li is supported by NSFC (No.~12001247),
 NSF of Jiangsu Province (No.~BK20201019),  NSF of Jiangsu Higher Education Institutions of China (No. 20KJB110015)
and the Foundation of Jiangsu Normal University (No.~19XSRX023). The research of W. Liu is supported by NSFC (No.~11822106, 11831014, 12090011) and the PAPD of Jiangsu Higher Education Institutions.}

\end{document}